%% file: 0_main_ContResDim.tex
\documentclass[a4paper,reqno]{amsart}

\usepackage[a4paper,left=3cm,right=3cm,top=3cm,bottom=3cm]{geometry}

\usepackage{graphicx} 
\usepackage{tikz}
\usetikzlibrary{patterns,shapes,decorations.pathmorphing,decorations.pathreplacing,calc,arrows,cd}
\usetikzlibrary{tikzmark, fit, matrix, decorations.markings}
\usepackage{mathdots}
\usepackage{amssymb}
\usepackage{amstext}
\usepackage{amsmath}
\usepackage{amscd}
\usepackage{amsthm}
\usepackage{amsfonts}

\usepackage{array}
\usepackage{enumerate}
\usepackage{latexsym}
\usepackage{mathrsfs}
\usepackage{hyperref}
\usepackage[all,color]{xy}
\usepackage{mathtools}

\usepackage{pgfplots}
\pgfplotsset{compat=1.16}

\usepackage{booktabs}
\usepackage{multirow}
\xyoption{all}
\usepackage{pstricks}
\usepackage{comment}

\usepackage{booktabs}
\usepackage{multirow}

\usepackage{mathtools}
\mathtoolsset{showonlyrefs=true}

\usepackage{todonotes}
\usepackage{tikz}

\tikzstyle{vecArrow} = [thick, decoration={markings,mark=at position
   1 with {\arrow[semithick]{open triangle 60}}},
   double distance=1.4pt, shorten >= 5.5pt,
   preaction = {decorate},
   postaction = {draw,line width=1.4pt, white,shorten >= 4.5pt}]

\tikzstyle{noArrow} = [thick, decoration={markings,mark=at position
   1 with {\arrow[semithick]{}}},
   double distance=1.4pt, 
   preaction = {decorate},
   postaction = {decorate}]

\numberwithin{equation}{section}
\newtheorem{theorem}{Theorem}[section]

\newtheorem{corollary}[theorem]{Corollary}
\newtheorem{lemma}[theorem]{Lemma}
\newtheorem{proposition}[theorem]{Proposition}

\newtheorem{defprop}[theorem]{Definition-Proposition}
\newtheorem{problem}[theorem]{Problem}

\theoremstyle{definition}

\newtheorem{definition}[theorem]{Definition}

\newtheorem{remark}[theorem]{Remark}
\newtheorem{example}[theorem]{Example}

\newtheorem{notation}[theorem]{Notation}

\newcommand{\Ext}{\operatorname{Ext}\nolimits}
\newcommand{\Tor}{\operatorname{Tor}\nolimits}
\newcommand{\Hom}{\operatorname{Hom}\nolimits}

\newcommand{\End}{\operatorname{End}\nolimits}

\newcommand{\RHom}{\mathbf{R}\strut\kern-.2em\operatorname{Hom}\nolimits}

\newcommand{\Image}{\operatorname{im}\nolimits}

\newcommand{\im}{\operatorname{im}\nolimits}

\renewcommand{\ker}{\operatorname{ker}\nolimits}
\newcommand{\coker}{\operatorname{coker}\nolimits}
\DeclareMathOperator{\moduleCategory}{\mathsf{mod}}
\renewcommand{\mod}{\moduleCategory}

\newcommand{\Hasse}{\operatorname{Hasse}\nolimits}

\DeclareMathOperator{\Rep}{\mathsf{Rep}}

\DeclareMathOperator{\op}{ {\rm op} }

\DeclareMathOperator{\intresdim}{\mathrm{int}\mbox{-}\mathrm{res}\mbox{-}\mathrm{dim}}
\DeclareMathOperator{\intresgldim}{\mathrm{int}\mbox{-}\mathrm{res}\mbox{-}\mathrm{gl{.}dim}}

\DeclareMathOperator{\Vect}{\mathsf{Vect}}
\DeclareMathOperator{\incl}{\mathrm{incl}}
\DeclareMathOperator{\rep}{\mathsf{rep}}
\DeclareMathOperator{\fprep}{\mathsf{fprep}}
\DeclareMathOperator{\colim}{\mathsf{colim}}
\DeclareMathOperator{\leaf}{\mathrm{leaf}}
\DeclareMathOperator{\Coind}{\mathrm{Coind}}
\DeclareMathOperator{\Ind}{\mathrm{Ind}}
\DeclareMathOperator{\Res}{\mathrm{Res}}
\DeclareMathOperator{\Cont}{\mathrm{Cont}}
\DeclareMathOperator{\Lan}{\mathrm{Lan}}
\DeclareMathOperator{\Ran}{\mathrm{Ran}}

\renewcommand{\lim}{\operatorname{\mathsf{lim}}\nolimits}
\DeclareMathOperator{\gldim}{\mathrm{gl{.}dim}}

\DeclareMathOperator{\projdim}{\mathrm{proj{.}dim}}
\DeclareMathOperator{\injdim}{\mathrm{inj{.}dim}}

\DeclareMathOperator{\id}{\mathrm{id}}

\DeclareMathOperator{\floorQ}{\it \lfloor \hspace{1.8mm} \rfloor_Q}

\DeclareMathOperator{\indeg}{\mathrm{indeg}}
\DeclareMathOperator{\outdeg}{\mathrm{outdeg}}

\DeclareMathOperator{\nlessdot}{\not\!\!\lessdot}

\usepackage{float}

\newcommand{\qand}{\operatorname{\quad \text{and} \quad}\nolimits}

\newcommand{\Int}{\operatorname{\mathrm{Int}}\nolimits}

\DeclareMathOperator{\rP}{ {\rm P} }
\DeclareMathOperator{\rI}{ {\rm I} }
\DeclareMathOperator{\rS}{ {\rm S} }

\newcommand{\new}[1]{{\blue #1}}
\newcommand{\old}[1]{{\red #1}}

\newcommand{\com}[1]{{\green #1}}

\begin{document}
\title[On preservation of relative resolutions for  
poset representations]{On preservation of relative resolutions \\ for poset representations}

\author[Toshitaka Aoki]{Toshitaka Aoki\textsuperscript{*}}
\address{Toshitaka Aoki,
  Graduate School of Human Development and Environment, Kobe University, 3-11 Tsurukabuto, Nada-ku, Kobe 657-8501 Japan}
\email{toshitaka.aoki@people.kobe-u.ac.jp}
\thanks{\textsuperscript{*}Corresponding author. Email: toshitaka.aoki@people.kobe-u.ac.jp}
\author[Shunsuke Tada]{Shunsuke Tada}
\address{Shunsuke Tada, Mathematical Science Center for Co-Creative Society, Tohoku University, 6-3 Aoba, Aramaki, Aoba-ku, Sendai, Miyagi 980-8578, Japan}
\email{shunsuke.tada.e6@tohoku.ac.jp}

\begin{abstract}
The concept of Galois connections (i.e., adjoint pairs between posets) is ubiquitous in mathematics. In representation theory, it is interesting because it naturally induces the adjoint quadruple between the categories of persistence modules (representations) of the posets via Kan extensions. One of central subjects in multiparameter persistent homology analysis is to understand structures of persistence modules. 

In this paper, we mainly study a class of Galois connections whose left adjoint is the canonical inclusion of a full subposet. We refer to such a subposet as an interior system, with its corresponding right adjoint given by the floor function. 
In the induced adjoint quadruple, we call the left Kan extension along its floor function the contraction functor. From its construction, it is left adjoint to the induction functor. Under this setting, we firstly prove that this adjoint pair gives an adjoint pair between finitely presentable persistence modules. Moreover, we introduce a special class of interior systems called aligned interior systems, and prove that both induction and contraction functors over them preserve interval-decomposability of modules. Then, we use them to analyze interval covers and resolutions. We also compute interval resolution global dimensions for certain classes of finite posets.
\end{abstract}


\maketitle

\tableofcontents

\input{1_Introduction}

\input{2_Preliminaries}

\input{3_ContractionFunctor}
\input{4_Invariants}

\input{5_StabilizingProperty}

\input{6_FinitePosets}

\noindent 
{\bf Acknowledgements.}
This work is supported by JSPS Grant-in-Aid for Transformative Research Areas (A) (22H05105). 
S.T. is supported by JST SPRING, Grant Number JPMJSP2148 and by JST, CREST Grant Number JPMJCR24Q6.

\bibliographystyle{alpha} 
\bibliography{main.bib}
\end{document}

%% file: 1_Introduction.tex
\section{Introduction}
Persistent homology analysis~\cite{ELZ02, ZC2004} has become a central method in Topological Data Analysis, with successful applications across various fields~\cite{HBH+16, MVS20, CVJ2021}. 
In the persistent homology pipeline, 
the topological structure of given data are encoded as 
algebraic objects called (one-parameter) persistence modules. In fact, each interval-summand represents a  topological feature of the data. 

Nowadays, the concept of Galois connections (adjoint pairs between posets) appears in various branch of mathematics,
and it also provides one of fundamental tools in computer science \cite{MSS86}. 
In representation theory, it is also interesting because a Galois connection 
$f \dashv g \colon Q \rightleftarrows P$ naturally induces an adjoint quadruple between the categories of (not necessarily pointwise finite dimensional) persistence modules via the left and right Kan extensions as follows: 
\begin{equation} \label{intro:adj quadrupe}
    \quad \ 
    f \dashv g \colon Q\rightleftarrows P, 
    \quad \quad \quad \quad \quad
    \xymatrix@C=68pt@R=38pt{
    \Rep Q 
    \ar@/^3.5mm/[r]|{g^*=\Lan_f} \ar@/^10mm/@{<-}[r]|{\Lan_g} 
    \ar@{}[r]|{\perp}
    \ar@{}@<7mm>[r]|{\perp}
    \ar@{}@<-7mm>[r]|{\perp}
    & 
    \ar@/^3.5mm/[l]|{f^*= \Ran_g} \ar@/^10mm/@{<-}[l]|{\Ran_f} 
    \Rep P. 
    }
\end{equation}

Recently, the study of persistence modules (or representations) over arbitrary posets has played a promising role in extending the one-parameter persistent homology pipeline.
In particular, efforts have been made to apply multiparameter persistent homology, although several theoretical and computational challenges still remain \cite{CZ09,BBK2020}.
In the context of poset representations, the class of interval modules, which are defined as indicator functors of certain subsets called intervals, 
provides a common generalization of indecomposable projective, indecomposable injective, and simple modules. 
Several numerical invariants defined from interval modules have been proposed in persistence theory, 
including the fibered barcode \cite{LW15}, the generalized rank invariant \cite{KM21}, the compressed multiplicity \cite{AENY,AGL24}, and others \cite{BOOS24,AL24,LW15}.
Based on relative homological algebra, approximations and resolutions relative to interval-decomposable modules, called interval approximations and interval resolutions respectively, have been studied \cite{AET,AENY,BBH1,BBH2}. 
In addition, a computation of their relative Betti numbers is presented in \cite{Asa25,CGR+24}.

An aim of this paper is to develop the study of persistence modules from a perspective of Galois connections. 
In the context of persistence theory, it has been adopted in \cite{GM} for giving a new approach to bottleneck stability for one-parameter persistent homology through so-called Rota's Galois connection theorem. 
A class of Galois connections whose right or left adjoint is an order-embedding is fundamental, as it always appears as a component in the decomposition of Galois connections (see Section~\ref{subsec:contraction}). 
We recall from \cite{EKMS} that an interior system is a full subposet $Q\subseteq P$ whose canonical embedding has a right adjoint as $\incl_Q \dashv \floorQ \colon Q\rightleftarrows P$. 
In this setting, we will write the diagram \eqref{intro:adj quadrupe} as 
\begin{equation}\label{intro:idemp_quad}
    \incl_Q \dashv \floorQ \colon Q\rightleftarrows P, 
    \quad \quad \quad 
    \xymatrix@C=68pt@R=38pt{
    \Rep Q 
    \ar@/^3.5mm/[r]|{\Ind_Q} \ar@/^10mm/@{<-}[r]|{\Cont_Q} 
    \ar@{}[r]|{\perp}
    \ar@{}@<7mm>[r]|{\perp}
    \ar@{}@<-7mm>[r]|{\perp}
    & 
    \ar@/^3.5mm/[l]|{\Res_Q} \ar@/^10mm/@{<-}[l]|{\Coind_Q} 
    \Rep P. 
    }
\end{equation}

In this diagram, the left Kan extension $\Cont_Q := \Lan_{\floorQ}$ is called \emph{contraction functor}. 
This functor first appears in the study of persistence modules over a finite product of totally ordered sets \cite{BBH2}, where they define contraction functors over finite aligned subgrids. 
Their definition looks slightly different from the above one, but they are naturally isomorphic by Proposition~\ref{prop:Cont_another}, see Example \ref{ex:aligned posets}(3) for more details. 
On the other hand, we note that $\Ind_Q$, $\Res_Q$, $\Coind_Q$ coincide with the usual induction, restriction, co-induction functors respectively 
when we work on incidence algebras of finite posets, and they can be defined for any full subposet (See Section~\ref{subsec:Kan extension}). 
It is important that $\Ind_Q$ is given by the pullback functor along $\floorQ$ and hence easily described in this setting. 
It is a basic property that $\Ind_Q$ preserves indecomposable projectives. 
In addition, we find that $\Cont_Q$ also satisfies this property (Proposition \ref{prop:basic Cont}(4)), 
which turns out that 
the adjoint pair $\Cont_Q \dashv \Ind_Q$ provides an adjoint pair between finitely presentable persistence modules (Proposition \ref{prop:basic Cont}(5)).

Mostly, the existence of operations that preserve properties, invariants, etc. we are focusing on brings us much information. 
In our study, it is mainly played by a class of functors that preserve interval-decomposability of modules.
We discuss here this preservation under relevant functors for a given full subposet $Q\subseteq P$. 
\begin{enumerate}[$\bullet$]
    \item $\Res_Q$ clearly preserves interval-decomposability of modules. 
    \item $\Ind_Q$ may send interval modules to non-interval-decomposable modules \cite[Example 4.9]{AET}.
    \item If $Q$ is an interior system, $\Ind_Q$ preserves interval-decomposability of modules (Proposition~\ref{prop:interval-Ind-Cont}). 
    \item Even if $Q$ is an interior system, $\Cont_Q$ may send interval modules to non-interval-decomposable modules (Remark~\ref{rem:preservation_intervals}).
\end{enumerate}

In this reason, we introduce a class of \emph{aligned interior systems} (Definition \ref{def:aligned intersys}) based on combinatorics of posets, over which contraction functors admit simpler descriptions using filtered colimits.
Thanks to this description, we prove in Proposition \ref{prop:interval-Ind-Cont} that $\Cont_Q$ over an aligned interior system $Q$ preserves interval-decomposability of modules, and in particular, 
the adjoint pair $\Cont_Q \dashv \Ind_Q$ gives an adjoint pair between interval-decomposable persistence modules. 
Then, we use this relationship to analyze interval covers and interval resolutions. 
One of our main results is Theorem \ref{thm:Ind-intervalcover}, which asserts that the induction functor over an aligned interior system preserves interval covers and interval resolutions. 
Here, we remark that there are no structural requirements on the base poset in this result, such as upper-semilattices or locally finiteness. 
We also remark that, both the existence of its left adjoint, contraction functor, and preservation of interval-decomposability are crucial in our proof. 
In persistent homology analysis, this result is helpful because it gives a method to reduce computational burden for computing interval covers/resolutions of a given persistence module, see Remark \ref{rem:computational interval cover}. 
In particular, it is applicable to any finitely presentable multiparameter persistence modules, often arising from multiparameter filtrations of point cloud data,
since finite aligned subgrids over a finite product of totally ordered sets \cite[Definition 7.1]{BBH2} form a special case of our aligned interior systems. 
On the other hand, it is shown in \cite[Theorem 4.1]{AET} that the intermediate extension, which originates in \cite{BBD}, preserves interval-decomposability of modules and is used to prove monotonicity of interval resolution global dimensions with respect to full subposet inclusions. 
Although this functor also arises from the adjoint triple, it is different from the contraction functor.

As an application, we provide a technique for transforming a given finite poset into a smaller one 
with preserving interval resolution global dimension (Theorem \ref{thm:cont chains}). 
Notice that this is closely related to conjectures \cite[Conjectures 4.11, 4.12]{AENY} that ask stabilization phenomena on interval resolution global dimensions of finite commutative grids. 

At the end of this paper, we compute interval resolution global dimensions of certain classes of finite posets. 
Especially, we give an explicit formula of interval resolution global dimensions for tree-type posets and $\tilde{A}$-type posets in Propositions \ref{prop:dimTree} and \ref{prop:dimAtilde} respectively. 
We also work on a classification of posets having low interval resolution global dimensions in Section \ref{subsec:low intresgildim}. \\

\noindent 
{\bf Notation.}
For two categories $\mathcal{C}$ and $\mathcal{D}$, 
a pair $F\colon \mathcal{C} \to \mathcal{D}$ and $G\colon \mathcal{D}\to \mathcal{C}$ of (covariant) functors is called \emph{adjoint pair}, with $F$ left adjoint and $G$ right adjoint, 
if there is an isomorphism $\Hom_{\mathcal{D}}(F(X),Y) \simeq \Hom_{\mathcal{C}}(X,G(Y))$ which is natural in $X\in \mathcal{C}$ and $Y\in \mathcal{D}$. 
In this case, we will write $F\dashv G\colon \mathcal{C}\rightleftarrows \mathcal{D}$.

%% file: 2_Preliminaries.tex
\section{Preliminaries}
\label{sec:prelim}
In this section, we recall basic definitions and results on persistence modules over posets. Throughout this paper, we fix a base field $k$. 

\subsection{Persistence modules}
Let $P$ be a partially ordered set equipped with partial order $\leq$. 
We regard it as a category whose objects are elements of $P$ and there is a unique morphism from $a$ to $b$ for $a, b\in P$ if and only if $a\leq b$. 
A \emph{full subposet} of $P$ is a subset $S$ together with the induced partial order. 
We denote by $\min(S)$ (resp., $\max(S)$) the set of minimal (resp., maximal) elements of $S$. 
In addition, let $S^{\uparrow}$ be the set of elements $x\in P$ such that $a\leq x$ for some $a\in S$. 
The subset $S^{\downarrow}$ is defined dually.  
If $S=S^{\uparrow}$ (resp., $S=S^{\downarrow}$), 
then it is called \emph{upset} (resp., \emph{down set}) of $P$. 
When $S$ is a singleton $\{a\}$, we sometimes identify it with the element. 
For example, we will use notations $a^{\uparrow}$ and $a^{\downarrow}$. 
We say that an upper bound (resp., lower bound) of $S$ is an element $c\in P$ such that $S\subseteq c^{\downarrow}$ (resp., $S\subseteq c^{\uparrow}$). 
Then, we say that a poset $P$ is \emph{filtered} if every pair of elements has an upper bound. 
For example, any upper-semilattice is filtered, but the converse fails in general. 
Finally, we write $P^{\op}$ for the \emph{opposite poset} of $P$, which is obtained from $P$ by reversing the partial order. 


Next, we discuss persistence modules over a poset $P$. 
By a \emph{$P$-persistence module}, we mean a functor from $P$ to the category $\Vect_k$ of (not necessarily finite dimensional) vector spaces over $k$. 
From its definition, a $P$-persistence module $M$ consists of correspondences such that 
\begin{itemize}
    \item it associates an element $a\in P$ to a $k$-vector space $M(a)$, and 
    \item it associates a relation $a\leq b$ in $P$ to a $k$-linear map $M(a \leq b)\colon M(a)\to M(b)$ in such a way that $M(a\leq a) = \id_{M(a)}$ and $M(a\leq c) = M(b\leq c)\circ M(a\leq b)$ for all $a\leq b \leq c$ in $P$. 
\end{itemize}
We say that $M$ is \emph{pointwise finite dimensional} (pfd for short) if $M(a)$ is finite dimensional for every $a\in P$.
A \emph{morphism} $\varphi \colon M\to N$ of $P$-persistence modules is a natural transformation from $M$ to $N$, 
that is, a family of $k$-linear maps $\varphi(a)\colon M(a) \to N(a)$ indexed by elements of $P$ and satisfying $\varphi(b) \circ M(a\leq b) = N(a\leq b)\circ \varphi(a)$ for all $a\leq b$ in $P$.
\begin{equation*}
    \xymatrix@C38pt@R27pt{
    M(a) \ar[r]^{M(a\leq b)} \ar[d]^{\varphi(a)} & M(b) \ar[d]^{\varphi(b)} \\ 
    N(a) \ar[r]^{N(a\leq b)} \ar@{}[ru]|{\circlearrowleft}& N(b). 
    }
\end{equation*}

We denote by $\Rep P$ the category of $P$-persistence modules, by $\rep P$ the full subcategory of $\Rep P$ consisting of those $P$-persistence modules which are pfd. 
It is well-known (see \cite{BCB}) that every pfd $P$-persistence module $M$ can be written as $\bigoplus_{\lambda\in \Lambda}M_{\lambda}$ with indecomposable $M_{\lambda}$. 
Moreover, this description is unique up to isomorphism and a permutation of the index set.

Next, we recall the definitions of intervals and interval modules over a given poset $P$. 

\begin{definition}
    Let $S\subseteq P$ be a subset. 
    \begin{enumerate}[\rm (a)]
        \item We say that $S$ is \emph{convex} if for any $x,y\in S$ and $z\in P$, 
        $x \leq z \leq y$ implies $z\in S$. 
        \item We say that $S$ is \emph{connected} if for any $x,y\in S$ there is a sequence $x=z_0,z_1,\ldots, z_m = y$ of elements of $S$ such that $z_{i-1}$ and $z_{i}$ are comparable for all $i\in \{1,\ldots,m\}$. 
        \item We say that $S$ is an \emph{interval} if it is convex and connected (This is called \emph{spread} in \cite{BBH1}). 
    \end{enumerate}
We denote by $\Int(P)$ the set of all intervals of $P$. 
\end{definition}


For a given interval $S \subseteq P$, one defines a pfd $P$-persistence module $\mathbb{I}_S$ called \emph{interval module} as follows: 
\begin{equation}\label{eq:def_interval}
    \mathbb{I}_S(x) := 
    \begin{cases}
        k & \text{if $x\in S$}, \\ 
        0 & \text{otherwise} 
    \end{cases} 
    \qand 
    \mathbb{I}_S(x\leq y ) := 
    \begin{cases}
        \id_k & \text{if $x,y\in S$}, \\ 
        0 & \text{otherwise}. 
    \end{cases}     
\end{equation}
Then, $\mathbb{I}_S$ is indecomposable whenever $S$ is non-empty. 
Indeed, it satisfies $\End_{P}(\mathbb{I}_S)\cong k$ \cite[Proposition 2.7]{BBH2}. 
Moreover, we say that a $P$-persistence module is \emph{interval-decomposable} if it is isomorphic to a direct sum of interval modules. 
For a convex subset $S\subseteq P$, we also denote by  $\mathbb{I}_S$ the interval-decomposable module given by \eqref{eq:def_interval}.  

\begin{proposition} \cite[Theorem 2.10]{BBH2} \label{prop:proj as interval}
    Let $M\in \Rep P$ be indecomposable. 
    Then, it is projective in $\Rep P$ if and only if there exists $a\in P$ such that $M\cong \mathbb{I}_{a^{\uparrow}}$. 
\end{proposition}

We say that a $P$-persistence module $M$ is \emph{finitely presentable} if it is isomorphic to the cokernel of a morphism between finitely generated projectives, namely, of the form $\coker\big(\bigoplus_{b\in \mathcal{J}} \mathbb{I}_{b^{\uparrow}} \to \bigoplus_{a\in \mathcal{I}} \mathbb{I}_{a^{\uparrow}}\big)$, 
where $\mathcal{I}$ and $\mathcal{J}$ are finite multisets of elements of $P$. 
We denote by $\fprep P$ the full subcategory consisting of all finitely presentable $P$-persistence modules. 
By definition, we have $\fprep P \subseteq \rep P \subseteq \Rep P$. 
In addition, we denote by $\mathscr{I}_P$ the full subcategory consisting of finitely presentable interval-decomposable $P$-persistence modules. 
Notice that, if $P$ is finite, then every pfd $P$-persistence module is finitely presentable and hence $\fprep P = \rep P$.

\begin{remark}\label{rem:colimit}
In the context of category theory, 
every $P$-persistence module $M\colon P\to \Vect_k$ is nothing but a diagram of shape $P$ in the category $\Vect_k$. 
Thus, one can consider the colimit and the limit of this diagram. 
We recall these notions as follows. 
We say that a \emph{co-cone} of the diagram $M$ is a $k$-vector space $C$ together with a family of $k$-linear maps $\beta_a\colon M(a) \to C$ indexed by elements $a\in P$ satisfying $\beta_b\circ M(a\leq b) = \beta_a$ for any $a\leq b$ in $P$. 
In addition, a \emph{colimit} of the diagram $M$ is a co-cone $(L,\alpha)$ of $M$ such that, for any co-cone $(C,\beta)$ of $M$, there is a unique $k$-linear map $\rho \colon L\to C$ such that $\rho \circ \alpha_a = \beta_a$ for all $a\in P$. 
\begin{equation}
    \xymatrix{
    M(a) \ar[rd]_{\alpha_a} \ar[rr]^-{M(a\leq b)} 
    \ar@/_5mm/[rdd]_{\beta_a} 
    && M(b) \ar[ld]^{\alpha_b} 
    \ar@/^5mm/[ldd]^{\beta_b} 
    \\ 
    & L \ar@{.>}[d]^-{\rho} 
    & \\ 
    & C. &
    }
\end{equation}
By the universal property, a colimit is unique up to a unique isomorphism. 
A \emph{limit} of $M$ is defined dually. 
We will write $\colim M$ (resp., $\lim M$) for a $k$-vector space of this colimit (resp., limit). 
We note that, even if $M$ is pfd, $\colim M$ and/or $\lim M$ can be infinite dimensional in general. 

Moreover, we have a \emph{colimit functor} 
$\colim \colon \Rep P\to \Vect_k$ that assigns to each $M$ its colimit $\colim M$. 
In fact, for a morphism $\varphi \colon M \to N$ of $P$-persistence modules, 
there is a unique $k$-linear map from $\colim M$ to $\colim N$ 
given by the universal property of colimits. 
We will denote this map by $\colim \varphi$. 
It commutes the following diagram, 
where $(\colim M,\alpha)$ and $(\colim N,\alpha')$ are colimits of $M$ and $N$ respectively.  
\begin{equation}
    \xymatrix@C20pt{
    M(a) \ar[rd]_{\alpha_a} \ar[rr]^-{M(a\leq b)} 
    \ar[dd]_{\varphi_a} 
    && M(b) \ar[ld]^{\alpha_b} 
    \ar[dd]^{\varphi_b} 
    \\ 
    & \colim M \ar@{.>}[dd]^(0.3){\colim \varphi} & \\ 
    N(a) \ar[rd]_{\alpha'_a} \ar[rr]^(0.3){N(a\leq b)}|\hole
    &  & N(b)\ar[ld]^{\alpha'_b}  \\ 
    & \colim N.&
    }
\end{equation}
Similarly, we have a \emph{limit functor} $\lim\colon \Rep P \to \Vect_k$. 

Let $Q$ be a full subposet of $P$ and $M\colon P\to \Vect_k$ a $P$-persistence module. 
We write $M|_Q$ 
for the $Q$-persistence module obtained from $M$ by restricting its domain to $Q$. 
In this situation, we have the induced maps 
$\colim M|_Q \to \colim M$ and 
$\lim M|_Q \to \lim M$ by the universal properties of colimits and limits respectively.

\end{remark}

The following computational results for colimits will be useful. 

\begin{lemma}\label{lem:colim_convex}
    Let $P$ be a poset. 
    Let $S\subseteq P$ be an interval and $\mathbb{I}_S\in \Rep P$ the corresponding interval module. 
    If $S$ is non-empty and an upset of $P$, 
    then the induced map $\colim \mathbb{I}_S|_S \to \colim \mathbb{I}_S$ is a $k$-linear isomorphism between one-dimensional $k$-vector spaces. 
    Otherwise, we have $\colim \mathbb{I}_S =0$. 
\end{lemma}

\begin{lemma}\label{lem:Scolim}
    Let $P$ be a poset and $S\subseteq P$ a full subposet of $P$. 
    If $S$ is filtered and satisfies $P\subseteq S^{\downarrow}$, 
    then for any $M\in \Rep P$, the induced map $\colim M|_{S} \to \colim M$ is an isomorphism. 
\end{lemma}

\begin{proof}
    It is sufficient to check that $\colim M|_S$ is a co-cone of $M$. 
    Let $a\in P$. 
    By $P\subseteq S^{\downarrow}$, 
    we have an element $s\in S$ such that $a\leq s$. 
    Consider a composition  
    \[
    \beta_{a,s} := \left( M(a) \xrightarrow{M(a\leq s)} M(s) \xrightarrow{\alpha_s} \colim M|_S \right), 
    \]
    where $\alpha_s$ is a map for the colimit of the diagram $M|_S$ at $s$. 
    Since $S$ is filtered, we find that 
    $\beta_{a,s} = \beta_{a,s'}$ holds for any $s,s'\in S$ with $a\leq s,s'$. 
    So, we write this map by $\beta_a$. 
    Then, they clearly give a co-cone $(\colim M|_S, (\beta_a)_{a\in P})$ of the diagram $M$, as desired.
\end{proof}

In describing a given finite poset, the notion of the Hasse diagram is useful. 

\begin{definition}\label{def:Hasse}
Suppose that $P$ is a finite poset.
For two elements $x,y\in P$, we say that \emph{$y$ covers $x$} 
and write $x\lessdot y$ if $x < y$ and every element $z\in P$ such that $x \leq z < y$ satisfies $x=z$. 
More generally, for $y\in P$ and $S\subseteq P$, 
we write $S\lessdot y$ if 
    $y\in S^{\uparrow}\setminus S$ and every $z\in S^{\uparrow}$ with $z < y$ belongs to $S$. 
Dually, $y\lessdot S$ is defined. 

The \emph{Hasse diagram} $\Hasse(P)$ of $P$ is a directed graph whose vertices are elements of $P$ and there is an arrow $x\to y$ if and only if $x\lessdot y$ (i.e., $y$ covers $x$). 
It is clear that the Hasse diagram recovers the original finite poset $P$. 
\end{definition}

\subsection{Kan extensions along order-preserving maps} \label{subsec:Kan extension}
For an order-preserving map $f\colon Q\to P$, 
we have a functor $f^*\colon \Rep P \to \Rep Q$ called \emph{pullback} along $f$, 
where a $Q$-persistence module $f^{\ast}M$ is defined by 
\begin{equation}
    f^*M(x) := M(f(x)) \qand 
    f^*M(x\leq y) := M(f(x)\leq f(y)) 
\end{equation}
for any $P$-persistence module $M$, 
and a morphism $f^*\varphi \colon f^*M\to f^*N$ is defined by $f^*\varphi(x) := \varphi(f(x))$ for any morphism $\varphi\colon M\to N$ between $P$-persistence modules. 
For a full subposet $Q\subseteq P$, 
the pullback along the canonical inclusion $\incl_Q \colon Q\hookrightarrow P$ is called \emph{restriction functor} and denoted by $\Res_Q$. 
Namely, it sends each $P$-persistence module $M$ to a $Q$-persistence module $M|_Q$, where we note that this notation is consistent with the one used in the last paragraph of Remark~\ref{rem:colimit}.

Next, we recall the notion of Kan extensions (We refer to  \cite{riehl2017category} for basics of category theory). 
For an order-preserving map $f\colon Q\to P$, 
we say that the \emph{left Kan extension} along $f$ is a functor $\Lan_f \colon \Rep Q \to \Rep P$ defined as follows: 
For a $Q$-persistence module $M$, we set   
\begin{equation}
    \Lan_f M(a) := \colim (M|_{f^{-1}(a^{\downarrow})}) 
    \qand 
    \Lan_f M(a\leq b) := \left(\colim (M|_{f^{-1}(a^{\downarrow})}) \to 
    \colim (M|_{f^{-1}(b^{\downarrow})})\right), 
\end{equation}
where the right-most arrow is the $k$-linear map given by the universal property of colimits 
(We recall from Remark \ref{rem:colimit} a meaning of colimits of persistence modules). 
In addition, for a morphism $\varphi \colon M\to N$ between $Q$-persistence modules, 
we have a morphism $\Lan_f\varphi \colon \Lan_f M\to \Lan_f N$ such that a $k$-linear map $\Lan_f \varphi(a) \colon \Lan_f M(a)\to \Lan_f N(a)$ is given by the universal property of colimits for each $a\in P$. 
They, indeed, define a functor $\Lan_f$. 
Dually, the \emph{right Kan extension} $\Ran_f\colon \Rep Q \to \Rep P$ is defined. 
From its construction, $\Lan_f$ is left adjoint to $f^*$, and $\Ran_f$ is right adjoint to $f^*$. 
Therefore, they form an adjoint triple $\Lan_f \dashv f^{*} \dashv \Ran_f$: 
\begin{equation} 
    \xymatrix@C=68pt@R=38pt{
    \Rep Q 
    \ar@/^5mm/[r]^{\Lan_f} \ar@/_5mm/[r]_{\Ran_f} 
    & 
    \ar[l]_{f^*} 
    \Rep P.  
    }
\end{equation}

Among others, the left and right Kan extensions along the canonical inclusion $\incl_Q \colon Q\hookrightarrow P$ is called the \emph{induction} and \emph{co-induction functors} and denoted by $\Ind_Q$ and $\Coind_Q$ respectively. 
Here, we recall basic properties of these functors. 

\begin{proposition}\cite{MacLane, Weibel1994} 
\label{prop:idemp_embed}
Let $Q\subseteq P$ be a full subposet. 
    \begin{enumerate}[\rm (a)]
        \item $\Ind_Q$ and $\Coind_Q$ are fully faithful functors such that $\Res_Q\circ \Ind_Q \simeq \id_{\Rep Q} \simeq \Res_Q\circ \Coind_Q $. 
        \item $\Ind_Q$ is left adjoint to $\Res_Q$, 
        and $\Coind_Q$ is right adjoint to $\Res_Q$. 
        \item $\Res_Q$ is exact, $\Ind_Q$ is right exact, and $\Coind_Q$ is left exact. 
        \item $\Ind_Q$ and $\Coind_Q$ preserve indecomposability of objects. 
        \item $\Ind_Q$ sends 
        projective modules to projective modules. 
        In particular, it restricts to the functor 
        $$\Ind_Q\colon \fprep Q \to \fprep P$$ 
        between finitely presentable persistence modules. 
    \end{enumerate}
\end{proposition}

We denote by $\mathcal{L}_Q$ the essential image of $\Ind_Q$ on $\fprep Q$, that is, the full subcategory of $\fprep P$ consisting of all $P$-persistence modules isomorphic to $\Ind_Q M$ for some $M\in \fprep Q$. 
Since $\Ind_Q$ is fully faithful, we have $\mathcal{L}_Q\simeq \fprep Q$.

%% file: 3_ContractionFunctor.tex
\section{Galois connections and relevant functors}

\label{subsec:adjoint_posets}

Regarding posets as categories, an adjoint pair $f \dashv g \colon Q\rightleftarrows P$ between two posets is known as a \emph{Galois connection}. 
It consists of a pair of order-preserving maps $f\colon Q \to P$ and $g \colon P\to Q$ satisfying 
\begin{equation} \label{eq:poset adjoint}
f(y)\leq x \Longleftrightarrow y \leq g(x) \quad 
\text{for all $x \in P$, $y \in Q$}.
\end{equation}
In this case, $f$ and $g$ are determined by each other and 
given by the formulas
$f(y) = \min\{x\in P \mid y\leq g(x)\}$ and 
$g(x) = \max\{y\in Q \mid f(y)\leq x\}$. 
By \eqref{eq:poset adjoint}, they satisfy $y\leq gf(y)$ and $fg(x)\leq x$ for all $x\in P$ and $y\in Q$.
Moreover, we have $f = fgf$ and $gfg = g$. 

It is well-known (see for example \cite[Proposition 4.4.6]{riehl2017category}) that adjoint pairs between categories induces those between the functor categories, providing the next result.

\begin{proposition} \label{prop:adj_adj} 
    A Galois connection $f\dashv g \colon Q \rightleftarrows P$ induces an adjoint pair $g^*\dashv f^* \colon \Rep Q \rightleftarrows \Rep P$. 
    Furthermore, it can be extend to the adjoint quadruple $\Lan_g \dashv g^*\dashv f^* \dashv \Ran_f\colon$  
    \begin{equation} \label{eq:adj quadrupe}
    \xymatrix@C=68pt@R=38pt{
    \Rep Q 
    \ar@/^2mm/[r]^{g^*=\Lan_f} \ar@/^8mm/@{<-}[r]^{\Lan_g} & 
    \ar@/^2mm/[l]^{f^*= \Ran_g} \ar@/^8mm/@{<-}[l]^{\Ran_f} 
    \Rep P.  
    }
    \end{equation}
\end{proposition}

\subsection{Contraction functors as left Kan extensions}
\label{subsec:contraction}

We especially focus on a class of Galois connections in which one side of the adjunction is the canonical inclusion. This class is fundamental, as it always arises from a given Galois connection $f\dashv g\colon Q\rightleftarrows P$ through the following decomposition:
\begin{equation}\label{eq:Galois_decomp}
    \xymatrix@C38pt{
    Q \ar@/^3mm/[r]^{gf} \ar@{}[r]|{\perp} &  
    \ar@/^3mm/[l]^{\incl_{\Image g}} \Image g \ar@/^3mm/[r]^{f|_{\Image g}} \ar@{}[r]|{\cong}& 
    \ar@/^3mm/[l]^{g|_{\Image f}} \Image f \ar@/^3mm/[r]^{\incl_{\Image f}} \ar@{}[r]|{\perp} & 
    \ar@/^3mm/[l]^{fg} P. 
    }
\end{equation}

This leads the following definition \cite{EKMS} (see also \cite{Everett44}).

\begin{definition} 
    Let $P$ be a poset.  
    An \emph{interior system} of $P$ is a full subposet $Q\subseteq P$ 
    whose canonical inclusion $\incl_Q\colon Q\hookrightarrow P$ has a right adjoint. 
    Defining a floor function by
    $\lfloor x \rfloor_Q := \max \{y\in Q \mid y\leq x\}$ for $x\in P$, 
    it gives a right adjoint as
    $\incl_Q \dashv \floorQ \colon Q \rightleftarrows P$ (see \eqref{eq:poset adjoint}). 

    Dually, a \emph{closure system} is defined, for which the left adjoint to the canonical inclusion is denoted as a ceil function. 
\end{definition}

We mainly study interior systems. 
Suppose that $Q$ is an interior system of $P$. 
For a subset $S\subseteq P$, 
let $\lceil S \rceil_Q = \{a \in P \mid \lfloor a \rfloor_Q \in S\}$ be the inverse image of $S$ under $\floorQ$. 
In addition, we write $\lceil y \rceil_Q = \{a\in P \mid \lfloor a \rfloor_Q = y\}$ for each $y\in Q$. 
Notice that they provide a decomposition $P = \bigsqcup_{y\in Q} \lceil y \rceil_Q$ of the base poset. 
In particular, we have $P = Q^{\uparrow}$. 
In this situation, 
one can compute the induction functor $\Ind_Q$ by using the floor function $\floorQ$ on $P$.
In fact, there exists a natural isomorphism $\Ind_Q \simeq (\floorQ)^*$, see Proposition \ref{prop:adj_adj}. 
Thus, for $M\in \Rep Q$, we have 
\begin{equation}\label{eq:Ind as floor}
    \Ind_QM(a) = M(\lfloor a \rfloor_Q) \qand 
    \Ind_QM(a\leq b) = M(\lfloor a \rfloor_Q \leq \lfloor b \rfloor_Q).
\end{equation}

Furthermore, according to Proposition \ref{prop:adj_adj}, 
$\Ind_Q$ admits the left adjoint provided by the left Kan extension along $\floorQ$. 
We refer to it as \emph{contraction functor} and write 
$\Cont_Q := \Lan_{\floorQ} \colon \Rep P \to \Rep Q$. 
In our notations, 
$\Cont_Q M$ of $M\in \Rep P$ can be written as 
\begin{equation}\label{eq:contM}
    \hspace{-1mm}
    \Cont_Q M(x) = \colim(M|_{\lceil x^{\downarrow} \rceil_Q})
    \ \text{and} \
    \Cont_Q M(x\leq y) = \left(\colim(M|_{\lceil x^{\downarrow} \rceil_Q}) \to \colim(M|_{\lceil y^{\downarrow} \rceil_Q})\right)
\end{equation}
where $x^{\downarrow}$ is taken as a subset of $Q$ and $\lceil x^{\downarrow} \rceil_Q = \{a\in P \mid \lfloor a \rfloor_Q \leq x\}$.  
One can rewrite the diagram of adjunctions \eqref{eq:adj quadrupe} as
    \begin{equation}\label{eq:idemp_quad}
    \xymatrix@C=68pt@R=38pt{
    \Rep Q 
    \ar@/^2mm/[r]^{\Ind_Q} \ar@/^8mm/@{<-}[r]^{\Cont_Q} & 
    \ar@/^2mm/[l]^{\Res_Q} \ar@/^8mm/@{<-}[l]^{\Coind_Q} 
    \Rep P.  
    }
\end{equation}

We summarize below important properties of contraction functors. 
Here, we remark that there are no structural requirements on $P$ in the next result, 
such as $P$ being an upper-semilattice or locally finite. 
In addition, we do not require $Q$ to be finite.

\begin{proposition}\label{prop:basic Cont}
    For any interior system $Q$ of a poset $P$, 
    the following statements hold.
    \begin{enumerate}[\rm (1)]
        \item We have an adjoint pair 
        \begin{equation}\label{eq:Ind-Cont}
            \Cont_Q \dashv \Ind_Q \colon \Rep P \rightleftarrows \Rep Q. 
        \end{equation}
        \item $\Cont_Q\circ \Ind_Q \simeq \id_{\Rep Q}$. 
        \item $\Cont_Q$ is right exact, and $\Ind_Q$ is exact.
        \item $\Cont_Q$ preserves projectives, and also indecomposable projectives. 
        \item The adjoint pair \eqref{eq:Ind-Cont} restricts to the adjoint pair between finitely presentable persistence modules as 
        \begin{equation} \label{eq:fp Ind-Cont}
            \Cont_Q \dashv \Ind_Q \colon \fprep P \rightleftarrows \fprep Q. 
        \end{equation}
    \end{enumerate}
\end{proposition}

\begin{proof}
    (1) This is given in \eqref{eq:idemp_quad}. 

    (2) This is because the right adjoint $\Ind_Q$ is fully faithful by Proposition \ref{prop:idemp_embed}(a).
    
    (3) We recall that $\Ind_Q$ is right exact by Proposition \ref{prop:idemp_embed}(c). 
    In addition, since $\Cont_Q\dashv \Ind_Q$ by (1), we have that $\Cont_Q$ is right exact and $\Ind_Q$ is left exact.

    (4) It follows from a basic fact in category theory (see \cite[Proposition 2.3.10]{Weibel1994} for example) that $\Cont_Q$ preserves projectives since its right adjoint $\Ind_Q$ is exact by (3).

    In the following, we prove that $\Cont_Q$ sends  indecomposable projectives to indecomposable projectives, 
    by showing that $\Cont_Q \mathbb{I}_{a^{\uparrow}} \cong \mathbb{I}_{(\lfloor a\rfloor_Q)^{\uparrow}}$ for any $a\in P$ (cf. Proposition \ref{prop:proj as interval}). 

    Let $a\in P$. 
    We only check that there is an isomorphism 
\begin{equation}\label{eq:cont_proj}
        \Cont_Q \mathbb{I}_{a^{\uparrow}}(x) 
        \cong
        \begin{cases}
        k & \text{if $\lfloor a\rfloor_Q \leq x$}, \\ 
        0 & \text{otherwise}  
        \end{cases}
    \end{equation}
    induced by universal properties of colimits 
    for any $x\in Q$. 
    To compute the left-hand side, we consider a convex subset 
    $S_x := a^{\uparrow} \cap \lceil x^{\downarrow} \rceil_Q$ of $P$. 

    Assume that $S_x$ is non-empty and take an element $b\in S_x$.
    By definition, it satisfies $a\leq b$ and $\lfloor a\rfloor_Q \leq \lfloor b\rfloor_Q \leq x$. 
    This shows that $a\in S_x$, and hence it is the minimum element of $S_x$. In particular, $S_x$ is connected in $\lceil x^{\downarrow} \rceil_Q$.
    Moreover, it is clearly an upset of $\lceil x^{\downarrow} \rceil_Q$. 
    Then, we can deduce
    \begin{equation}
    \Cont_Q \mathbb{I}_{a^{\uparrow}}(x) = \colim \mathbb{I}_{a^{\uparrow}}|_{\lceil x^{\downarrow}\rceil_Q} =
    \colim \mathbb{I}_{S_x}|_{\lceil x^{\downarrow} \rceil_Q} \overset{\rm Lem.\,\ref{lem:colim_convex}}{\cong} k   
    \end{equation} 
    as desired. 
    
    From the above argument, we find that $S_x$ is non-empty if and only if $\lfloor a \rfloor_Q \leq x$. 
    Then, we get the remaining assertion because 
    we have $\Cont_Q\mathbb{I}(x)=\colim \mathbb{I}_{a^{\uparrow}}|_{\lceil x^{\downarrow} \rceil_Q}=\colim \mathbb{I}_{S_x}|_{\lceil x^{\downarrow} \rceil_Q} = 0$ if $S_x$ is empty. 
    
    (5) It is immediate from (1), (4) and Proposition \ref{prop:idemp_embed}(e).
\end{proof}

We remark that the contraction functors have appeared in \cite[Definition 7.14]{BBH2} in the study of persistence modules over a finite product of totally ordered sets, and their contraction functor forms a special case of ours (see Example \ref{ex:aligned posets}(3) for the detail). In addition, the above Proposition \ref{prop:basic Cont}(4) is obtained in \cite[Section 7.2]{BBH2} for this class of posets.

\subsection{Contraction functors over aligned interior systems}

In this section, we introduce a special class of interior systems called \emph{aligned interior systems}, 
over which the contraction functors enjoy nice behavior. 
In fact, they play a central role in the study of interval modules in the next section.

\begin{defprop}
    \label{def:aligned intersys} \rm
    For an interior system $Q$ of $P$, 
    the following conditions (a) and (b) are equivalent.
    \begin{itemize}
        \item [\rm (a)] $\lceil y^{\downarrow} \rceil_Q = \{a\in P \mid \lfloor a \rfloor_Q \leq y\}$ is filtered for every $y\in Q$. 
        \item [\rm (b)]  $Q$ satisfies both {\rm (AL1)} and {\rm (AL2)}.
        \begin{itemize}
        \item[{\rm (AL1)}] $\lceil y^{\downarrow} \rceil_Q = (\lceil y \rceil_Q)^{\downarrow}$ holds for all $y\in Q$.
        \item[{\rm (AL2)}] $\lceil y \rceil_Q$ is filtered for every $y\in Q$.  
        \end{itemize}
    \end{itemize}
    An interior system of $P$ is said to be \emph{aligned} if it satisfies one of the above equivalent conditions.
    
\end{defprop}

\begin{proof} 
    Let $Q$ be an interior system of $P$. 
    In this case, we note that one can replace the condition (AL1) in (b) with the following equivalent condition:
    \begin{itemize}
        \item Let $x,y\in Q$ with $x\leq y$. Then, for any $a\in \lceil x \rceil_Q$, there is $b\in \lceil y \rceil_Q$ such that $a\leq b$. 
    \end{itemize}
    In fact, we check it by following discussion. 
    Firstly, we check that the above condition implies (AL1). 
    It  suffices to show that 
    $\lceil y^{\downarrow} \rceil_Q \subseteq (\lceil y \rceil_Q)^{\downarrow}$
    holds for any $y\in Q$ 
    since the converse inclusion relation always holds. 
    Let $y\in Q$ and $a\in \lceil y^{\downarrow} \rceil_Q$. 
    Since $\lfloor a \rfloor_Q \leq y$, there exists $b\in \lceil y \rceil_Q$ such that $a\leq b$ by our assumption. 
    This means $a\in (\lceil y \rceil_Q)^{\downarrow}$. 
    Conversely, we assume (AL1). 
    To check the condition, we take $x,y\in Q$ with $x\leq y$ and $a\in \lceil x \rceil_Q$. 
    By $\lfloor a \rfloor_Q = x \leq y$, we have 
    $a \in \lceil y^{\downarrow} \rceil_Q = (\lceil y \rceil_Q)^{\downarrow}$, where the latter equality is given by our assumption (AL1). 
    This shows that there is $b\in \lceil y \rceil_Q$ satisfying $a\leq b$, as desired.

    From now on, we prove that (a) and (b) are equivalent. 
    On one hand, we assume that (a) is satisfied.
    To check the condition (AL2), let $a,b\in \lceil y\rceil_Q$. 
    By (a), there is an element $c\in \lceil y^{\downarrow}\rceil_Q$ such that $a,b\leq c$. 
    In this case, we have $y = \lfloor a\rfloor_Q \leq \lfloor c \rfloor_Q \leq y$ and hence $c \in \lceil y\rceil_Q$. 
    Thus, a pair $a,b$ have an upper bound in the set $\lceil y \rceil_Q$. 
    Next, we check the condition (AL1). 
    Let $x,y\in Q$ with $x \leq y$ and $a\in \lceil x \rceil_Q$. 
    We have $a\in \lceil y^{\downarrow} \rceil_Q$ by $\lfloor a \rfloor_Q = x \leq y$. 
    Since $\lceil y^{\downarrow} \rceil_Q$ is filtered by (a), 
    there is an element $b \in \lceil y^{\downarrow} \rceil_Q$ such that $a\leq b$ and $y\leq b$. 
    Notice that $\lfloor b \rfloor_Q \leq y$ holds by $b \in \lceil y^{\downarrow} \rceil_Q$. 
    Then, applying $\floorQ$, 
    we have relations 
    $y = \lfloor y \rfloor_Q \leq \lfloor b \rfloor_Q \leq y$, 
    showing $\lfloor b \rfloor_Q = y$ and $b \in \lceil y \rceil_Q$. 
    Thus, we get the claim.
    
    On the other hand, we assume that both (AL1) and (AL2) in (b) are satisfied. 
    Let $a,b\in \lceil y^{\downarrow} \rceil_Q$ with  $x_a := \lfloor a \rfloor_Q $ and $x_b := \lfloor b \rfloor_Q$. 
    By (AL1) with $x_a,x_b \leq y$, 
    there are $c_a,c_b\in \lceil y \rceil_Q$ such that $a\leq c_a$ and $b \leq c_b$. 
    Since $\lceil y \rceil_Q$ is filtered by (AL2), we have an element $d\in \lceil y \rceil_Q \subseteq \lceil y^{\downarrow} \rceil_Q$ 
    such that $c_a\leq d$ and $c_b\leq d$. 
    This gives an upper bound of $a$ and $b$ in $\lceil y^{\downarrow} \rceil_Q$ as desired. 
\end{proof}

We have the following description of contraction functors over aligned interior systems. 

\begin{proposition}\label{prop:Cont_another}
    Let $Q$ be an aligned interior system of $P$. 
    For any $M\in \Rep P$, we have 
    \begin{equation}
        \Cont_Q M(x) = \colim(M|_{\lceil x \rceil_Q}) 
        \qand 
        \Cont_Q M(x\leq y) = \left(
        \colim(M|_{\lceil x \rceil_Q}) \to 
        \colim(M|_{\lceil y \rceil_Q})
        \right) 
    \end{equation}
    up to isomorphism, where the right-most arrow is the $k$-linear map given by the universal property of colimits. 
    Besides, a similar description holds for morphisms. 
\end{proposition}

\begin{proof}
    In the statement, the map $\colim(M|_{\lceil x \rceil_Q})\to \colim(M|_{\lceil y \rceil_Q})$ exists because $\colim M|_{\lceil y \rceil_Q}$ is a co-cone of $M|_{\lceil x \rceil_Q}$ thanks to (AL1) and (AL2). 
    Thus, they give rise to a $Q$-persistence module.
    
    From \eqref{eq:contM}, 
    it suffices to show that the induced map $\colim (M|_{\lceil x \rceil_Q}) \to \colim(M|_{\lceil x^{\downarrow} \rceil_Q})$ is an isomorphism for every $x\in P$.   
    However, we can deduce it 
    by applying Lemma \ref{lem:Scolim} to $\lceil x \rceil_Q \subseteq 
    \lceil x^{\downarrow} \rceil_Q$. 
    Indeed, they satisfy 
    $\lceil x^{\downarrow} \rceil_Q \subseteq (\lceil x \rceil_Q)^{\downarrow}$ by (AL1), 
    and $\lceil x \rceil_Q$ is filtered by (AL2). 
\end{proof}

\begin{proposition} \label{prop:cont is exact}
    Let $Q$ be an aligned interior system of $P$. 
    Then, $\Cont_Q$ is exact. 
\end{proposition}

\begin{proof}
    The right exactness of $\Cont_Q$ is given in Proposition \ref{prop:basic Cont}. 
    On the other hand, the left exactness follows from the facts that filtered colimits commute with finite limits, and kernels are types of finite limits. 
\end{proof}

\begin{example}\label{ex:aligned posets}
    \begin{enumerate}
        \item Trivially, $P$ itself is always an aligned interior system of a poset $P$. In particular, $P$ admits at least one aligned interior systems. 

        \item The subset $\mathbb{Z}\subseteq \mathbb{R}$ is an aligned interior system of $\mathbb{R}$ with respect to the usual order. 
        Indeed, the corresponding floor function is given by mapping an real number $r$ to the greatest integer less than or equals to $r$. 
        Thus, we have $\lceil x \rceil_{\mathbb{Z}} = [x,x+1)$ for any integer $x\in \mathbb{Z}$. 
        
        \item Let $R=R_1\times \cdots \times R_m$ be a finite product of totally ordered sets equipped with the usual product order (e.g. $R=\mathbb{R}^d$). 
        We recall from \cite{BBH2} that a \emph{finite aligned subgrid} of $R$ is a subset $Q$ of the form $T_1\times \cdots \times T_m \subseteq R$ where $T_i\subseteq R_i$ is a finite subset for every $i$. 
        Following from \cite[Lemma 7.6]{BBH2}, $Q$ is an aligned interior system of $Q^{\uparrow}$. 
        Furthermore, by Proposition \ref{prop:Cont_another}, 
        our $\Cont_Q$ is exactly 
        their contraction functor defined in \cite[Definition 7.14]{BBH2}. 
        In this setting, Proposition \ref{prop:cont is exact} corresponds to \cite[Lemma 7.17]{BBH2}. 

        Similarly, there are situations where a given full subposet $Q\subseteq P$ is not an (aligned) interior system of $P$ but it is an (aligned) interior system of $Q^{\uparrow}$. 
        
        \item Consider a poset $P$ whose Hasse diagram is given as follows. 
        \begin{equation*}    
        \xymatrix@C17pt@R13pt{
        &  y&   & \\  
        x \ar[ur] &  & v \ar[ul] & \\  
        & u \ar[ur] \ar[ul] & & d \ar[ul] & \\ 
        &b \ar[u] \ar[urr] & & c \ar[ull] \ar[u] \\ 
        && a \ar[ul] \ar[ur] &  
        }   
        \end{equation*}
        We take $Q:=\{a,u,x\}$. 
        Then, it is an aligned interior system of $P$ with satisfying 
        $\lceil a \rceil_Q = \{a,b,c,d\}$,  $\lceil u \rceil_Q = \{u,v\}$, and
        $\lceil x \rceil_Q = \{x,y\}$. 

        On the other hand, if we take $Q'$ as a full subposet $Q':= P \setminus \{x\}$ of $P$, 
        then it is an interior system of $P$ with $\lfloor x \rfloor_{Q'} = u$. 
        However, it is not aligned because 
        $\lceil v^{\downarrow} \rceil_{Q'} = P\setminus \{y\}$ is not filtered. 
        We also have examples in Remark \ref{rem:preservation_intervals} for interior systems which are not aligned. 
    \end{enumerate}
\end{example}

\begin{remark}\label{rem:cont finite}
    Suppose that $P$ is a finite poset. 
    In this situation, the contraction functor $\Cont_Q$ with respect to an aligned interior system $Q$ can be understood as a certain restriction functor. 
    Indeed, by (AL2), every $\lceil y \rceil_Q$ with $y\in Q$ admits the maximum element, say $\nu(y)$. 
    Then, the correspondence $y\mapsto \nu(y)$ gives an isomorphism between $Q$ to its image $\nu(Q) = \{\nu(y) \mid y\in Q\}$ by (AL1). 
    By the definition of colimits, $\Cont_Q\simeq \nu^{*}\circ \Res_{\nu(Q)}$ holds as functors. 
    That is, $\Cont_QM$ for $M\in \Rep P$ is provided by 
    \begin{equation}
        \Cont_QM(y) = M(\nu(y))
        \qand 
        \Cont_QM(x\leq y) = M(\nu(x) \leq \nu(y))
    \end{equation}
    up to isomorphisms.
    
    The contraction functor appears to have interesting properties especially for infinite posets, 
    but it also brings consequences even when $P$ is finite, see Sections \ref{sec:stabilizing} and \ref{sec:computing intresdim}. 
\end{remark}

%% file: 4_Invariants.tex
\section{Induction and contraction for interval modules}

In this section, we consider to apply the induction and contraction functors to the class of interval modules. 
In particular, we prove that both functors preserve interval-decomposability of persistence modules when we consider aligned interior systems (Proposition \ref{prop:interval-Ind-Cont}). 
This contrasts with a known fact (see \cite[Example 4.9]{AET} for example) that the induction functor with respect to full subposet inclusions may send an interval module to a non-interval-decomposable module in general. 
After that, we use these functors over aligned interior systems to study interval covers and resolutions (Section \ref{sec:Ind-Cont resolusion}).

\subsection{Interval modules and aligned interior systems}

We begin with the following observations. 
Let $P$ be an arbitrary poset.

\begin{lemma}\label{lem:ceil interval}
    Let $Q$ be an interior system of $P$. 
    For a given subset $S\subseteq Q$, we have the following. 
    \begin{enumerate}[\rm (1)]
        \item $S$ is convex in $Q$ if and only if $\lceil S \rceil_Q$ is convex in $P$. 
        \item $S$ is connected in $Q$ if and only if $\lceil S \rceil_Q$ is connected in $P$. 
        \item $S$ is an interval in $Q$ if and only $\lceil S \rceil_Q$ is an interval in $P$. 
    \end{enumerate}
\end{lemma}

\begin{proof}
    (1) We recall a fact that the inverse image of a convex set under an order-preserving map is again convex. 
    In our situation, $\lceil S \rceil_Q$ (resp., $S$) is the inverse image of $S$ (resp., $\lceil S \rceil_Q$) by $\floorQ$ (resp., $\incl_Q$) by definition. 
    
    (2) We observe that the image of a connected set under an order-preserving map is again connected. 
    Thus, $S$ is connected whenever $\lceil S \rceil_Q$ is connected. 
    Conversely, if $S$ is connected, then for two elements $a,b\in \lceil S \rceil_Q$, 
    we have a sequence $a \geq \lfloor a \rfloor_Q = z_0,z_1,\ldots, z_m = \lfloor b \rfloor_Q \leq b$ with $z_i\in S\subseteq \lceil S \rceil_Q$ such that $z_{i-1}$ and $z_i$ are comparable for each $i\in \{1,\ldots,m\}$. 
    It implies that $\lceil S \rceil_Q$ is connected. 

    (3) It is immediate from (1) and (2).
\end{proof}

\begin{lemma}\label{lem:cont interval}
    Suppose that $Q$ is an aligned interior system of $P$. 
    For an interval $T\subseteq P$, let 
    
\begin{equation}\label{eq:tbar}
        \overline{T}^Q := \{x\in Q \mid 
        \text{$T\cap \lceil x \rceil_Q$ is non-empty and an upset of $\lceil x \rceil_Q$}
        \}.
    \end{equation}

    Then, $\overline{T}^Q$ is convex in $Q$. 
\end{lemma} 

\begin{proof}
    Let $T\subseteq P$ be an interval, and let $\overline{T}^Q$ be a subset of $Q$ as in the statement. 
    We first show that an element $x\in Q$ belongs to $\overline{T}^Q$ 
    if and only if $\lceil x \rceil_Q  \subseteq (T\cap\lceil x \rceil_Q)^{\downarrow}$ holds
    (i.e., any $r\in \lceil x \rceil_Q$ admits $p\in T\cap \lceil x \rceil_Q$ satisfying $r\leq p$) in the following discussion. 
    On one hand, let $x\in \overline{T}^Q$ and $t\in T\cap \lceil x \rceil_Q$. For any $a\in \lceil x \rceil_Q$, by (AL2), we have an element $q \in \lceil x \rceil_Q$ such that $a\leq q$ and $t\leq q$. 
    In this case, we have $q\in T$ since $T\cap \lceil x \rceil_Q$ is an upset of $\lceil x \rceil_Q$ by our assumption. 
    This shows that $a\in (T\cap \lceil x \rceil_Q)^{\downarrow}$. 
    On the other hand, let $x\in Q$ be an element such that $\lceil x \rceil_Q \subseteq (T\cap \lceil x \rceil_Q)^{\downarrow}$.
    In this case, $T\cap \lceil x \rceil_Q$ is non-empty. 
    Moreover, if we take $a\in T\cap \lceil x \rceil_Q$ and $b\in \lceil x \rceil_Q$ with $a \leq b$, then 
    there exists $c\in T\cap \lceil x \rceil_Q$ such that $b\leq c$ by our choice of $x$. 
    Thus, we have $a\leq b \leq c $ and $b\in T$ by the convexity of $T$, as desired. 

    Now, we prove the claim. 
    If $\overline{T}^Q$ is empty, then it is convex. 
    So we will consider the other case. 
    Let $x,y\in \overline{T}^Q$ and $z\in Q$ such that $x\leq z \leq y$. 
    We need to show that $z\in \overline{T}^Q$. 
    To do this, we first show that $T \cap \lceil z \rceil_Q$ is non-empty. 
    Take any element $t\in T\cap \lceil x \rceil_Q$. 
    By (AL1) with $x\leq z$, 
    there exists $s \in \lceil z \rceil_Q$ such that $t\leq s$. 
    Similarly, by $z\leq y$, there exists $r\in \lceil y \rceil_Q$ such that $s\leq r$. 
    Since $y\in \overline{T}^Q$, there is an element $p\in T\cap \lceil y \rceil_Q$ such that $r \leq p$ by definition. 
    Combining them, we obtain $t \leq s \leq r \leq p$ with $t,p\in T$. 
    This shows $s,r\in T$ by the convexity. 
    Consequently, we have $s \in T\cap \lceil z \rceil_Q\neq \emptyset$. 

    Next, we show that $T\cap \lceil z \rceil_Q$ is an upset of $\lceil z \rceil_Q$. 
    We take $s\in T\cap \lceil z \rceil_Q$ and $q\in \lceil z \rceil_Q$ such that $s\leq q$.   
    By (AL1) with $z\leq y$, 
    there is $r\in \lceil y \rceil_Q$ such that $q\leq r$.
    In addition, by $y\in \overline{T}^Q$, 
    there exists $p\in T\cap \lceil y \rceil_Q$ such that $r\leq p$. 
    Then, we obtain $s \leq q \leq r \leq p$ with $s,p\in T$. 
    By the convexity of $T$, we have $q\in T$ (and $r\in T$). 
    It finishes the proof.
\end{proof}

We notice that the above $\overline{T}^Q$ is not necessarily connected as in the next example. 
\begin{example}
We consider a poset $P$ whose Hasse diagram is given by  
    \begin{equation*}    
    \xymatrix@C10pt@R10pt{
    & z' & \\  
    & z \ar[u] & \\ 
    x \ar[ur] && y. \ar[ul]
    }
    \end{equation*}
In this case, $Q := \{x,y,z\}$ is an aligned interior system of $P$ with $\lfloor z'\rfloor_Q = z$. 
Letting $T = Q$, this is an interval of $P$ and such that $\overline{T}^Q = \{x,y\}$. 
In fact, we have $z\not\in \overline{T}^Q$ since $z'\not \in T$. 
Thus, $\overline{T}^Q$ is convex but not connected. 
\end{example}

Our result is the following. 

\begin{proposition}\label{prop:interval-Ind-Cont}
    Let $Q$ be an interior system of $P$. Then, the following hold. 
    \begin{enumerate}[\rm (1)]
        \item $\Ind_Q$ sends interval modules to interval modules. 
        In fact, for any interval $S\subseteq Q$, we have $\Ind_Q \mathbb{I}_S \cong \mathbb{I}_{\lceil S \rceil_Q}$.
        \item If moreover $Q$ is aligned, 
        then $\Cont_Q$ sends interval modules to interval-decomposable modules. 
        In fact, for any interval $T\subseteq P$, we have 
        $\Cont_Q \mathbb{I}_T \cong \mathbb{I}_{\overline{T}^Q}$.
    \end{enumerate}
    Therefore, if $Q$ is an aligned interior system of $P$,
    then the adjoint pair \eqref{eq:Ind-Cont} restricts to 
    an adjoint pair between interval-decomposable modules. 
\end{proposition}

\begin{proof}
(1) It follows from a natural isomorphism $\Ind_Q \simeq (\floorQ)^{*}$. 

(2) Let $T\subseteq P$ be an interval and $\overline{T}^Q$ the convex subset of $Q$ given in \eqref{eq:tbar}. 
For $x\in Q$, we have 
\begin{equation}
    \Cont_Q \mathbb{I}_T(x) \overset{\rm Prop. \ref{prop:Cont_another}}{=} \colim (\mathbb{I}_T|_{\lceil x \rceil_Q}) = \colim (\mathbb{I}_{T\cap \lceil x \rceil_Q}|_{\lceil x \rceil_Q}) .  
\end{equation}

We recall that $x\in \overline{T}^Q$ if and only if $T\cap \lceil x \rceil_Q$ is non-empty and an upset of $\lceil x \rceil_Q$. 
Following from Lemma \ref{lem:colim_convex}, we find that the induced map $\colim (\mathbb{I}_{T\cap \lceil x \rceil_Q}|_{T\cap \lceil x \rceil_Q}) \xrightarrow{\sim} \colim (\mathbb{I}_{T\cap \lceil x \rceil_Q}|_{\lceil x \rceil_Q})$ is an isomorphism between one-dimensional $k$-vector spaces if $x\in \overline{T}^Q$, and $\colim (\mathbb{I}_{T\cap \lceil x \rceil_Q}|_{\lceil x \rceil_Q}) = 0$ otherwise. This shows the desired isomorphism $\Cont_Q \mathbb{I}_T\cong \mathbb{I}_{\overline{T}^Q}$ of $Q$-persistence modules. 

The latter assertion follows from (1) and (2).
\end{proof}

In addition, we focus on the subclasses of upset and single-source interval modules.
Here, an interval module is said to be 
\emph{upset} (resp., \emph{single-source}) 
if its associated interval is an upset (resp., has the  minimum). 
For our convenience, we regard the zero object, which is an interval module associated to the empty set, as both upset and single-source interval module. 

The next result is a direct consequence of Proposition \ref{prop:interval-Ind-Cont}(1). 

\begin{proposition}
    Suppose that $Q$ is an interior system of $P$.
    For any interval $S$ of $Q$, we have the following.
    \begin{enumerate}[\rm (1)]
        \item $S$ is an upset of $Q$ if and only if $\lceil S \rceil_Q$ is an upset of $P$.
        \item If $S$ has the minimum $a$, then so does $\lceil S \rceil_Q$, and vice versa. 
    \end{enumerate}
    In particular, $\Ind_Q$ sends an upset (resp., single-source) interval modules to upset (resp., single-source) interval modules. 
\end{proposition}

\begin{proof}
    The claims (1) and (2) are clear. The latter assertion follows from Proposition \ref{prop:interval-Ind-Cont}(1).
\end{proof}

If moreover $Q$ is aligned, 
then we prove that a similar result holds for contraction functors. 
This enables us to restrict our adjoint pair to 
the subclasses of upset and single-source interval modules too (Corollary \ref{cor:adj_upset/ss}).

\begin{proposition}\label{prop:Tbar_upset}
    Let $Q$ be an aligned interior system of $P$.
    For an interval $T$ of $P$, the following hold.
    \begin{enumerate}[\rm (1)]
        \item If $T$ is an upset of $P$, 
        then $\overline{T}^Q = \lfloor T \rfloor_Q$ and it is an upset of $Q$. 
        \item If $T$ has the minimum $a$, 
        then $\overline{T}^Q$ has the minimum $\lfloor a \rfloor_Q$ whenever $\overline{T}^Q$ is non-empty.
    \end{enumerate}
    Therefore, $\Cont_Q$ sends an upset (resp., single-source) interval modules to upset (resp., single-source) interval modules.     
\end{proposition}

\begin{proof}
    (1) Suppose that $T$ is an upset of $P$. 
    Firstly, we check that $\lfloor T \rfloor_Q$ is an upset of $Q$. 
    Let $x = \lfloor a \rfloor_Q \in \lfloor T \rfloor_Q$ with $a\in T$. 
    For $y\in Q$ with $x\leq y$, by (AL1), 
    there is $b\in \lceil y \rceil_Q$ such that $a\leq b$. Since $T$ is an upset, we have $b\in T$. 
    By $y = \lfloor b \rfloor_Q$, we have $y \in \lfloor T \rfloor_Q$ as desired. 
    Secondly, we have $\overline{T}^Q \subseteq \lfloor T \rfloor_Q$ since $x\in \overline{T}^Q$ implies that $T\cap \lceil x \rceil_Q$ is non-empty. 
    Finally, for any $x\in \lfloor T \rfloor_Q$, we note that $T\cap \lceil x \rceil_Q$ is non-empty. 
    Then, it is an upset of $\lceil x \rceil_Q$ since $T$ is an upset. This shows $x\in \overline{T}^Q$ by \eqref{eq:tbar}.

    
    (2) Let $T$ be an interval of $P$ having the minimum $a$. 
    For simplicity, let $y := \lfloor a \rfloor_Q\in Q$. 
    Suppose that $\overline{T}^Q$ is non-empty and take an arbitrary $z\in \overline{T}^Q$.  
    In this case, there is an element $p\in T\cap \lceil z \rceil_Q$. 
    Since $a$ is the minimum of $T$, it satisfies $a \leq p$. 
    Applying $\floorQ$, we obtain $y \leq z$. 

    Thus, it remains to show that $y \in \overline{T}_Q$, namely, $T\cap \lceil y \rceil_Q$ is an upset of $\lceil y \rceil_Q$. 
    Since $a$ is the minimum of $T$, it is sufficient to check that any $b\in \lceil y \rceil_Q$ with $a\leq b$ is contained in $T$. 
    Take any $z\in \overline{T}^Q$. 
    Then, $y\leq z$ holds in $Q$ as in the previous paragraph. 
    By (AL1) with $y\leq z$,
    there exists $c \in \lceil z \rceil_Q$ such that $b\leq c$. 
    Since $z\in \overline{T}^Q$, we have an element $p \in T \cap \lceil z \rceil_Q$ such that $c\leq p$. 
    They satisfy relations $a\leq b \leq c \leq p$ with $b,p\in T$. 
    Since $T$ is convex, we have $b\in T$ (and $c\in T$) as desired. 
\end{proof}

\begin{corollary}\label{cor:adj_upset/ss}
    If $Q$ is an aligned interior system of $P$, 
    then the adjoint pair \eqref{eq:Ind-Cont} restricts to the adjoint pairs between upset/single-source interval-decomposable modules. 
\end{corollary}

We put a remark on preservation of interval-decomposability of modules.

\begin{remark} \label{rem:preservation_intervals}
    On Proposition \ref{prop:interval-Ind-Cont}(2), 
    if $Q$ is an interior system which is not aligned, 
    then the contraction functor may send an interval module to a non-interval-decomposable module in general. 
  Indeed, the following example illustrates this.
    Let $P$ and $Q$ be posets given by the following  Hasse diagrams:   
    \begin{equation}
        \begin{tikzpicture}   
        \node at (-2,1.7) {$Q$};
        \node at (-2,0) {$\xymatrix@C10pt@R15pt{
         &w&  \\ 
        & z \ar[u] & \\ 
        x \ar[ur] & & \ar[ul] y  
        }$}; 
        \node at (0,0) {$\subseteq$}; 
        \node at (3,1.7) {$P$};
        \node at (3,0) {$\xymatrix@C10pt@R15pt{
        && w &&  \\ 
        x' \ar[urr] &&  z \ar[u] &  \\ 
        &\ar[ul] x \ar[ur] & & \ar[ul] y  & 
        }$}; 
        \end{tikzpicture}    
    \end{equation}
    Then, $Q=P\setminus \{x'\}$ is an interior system of $P$ with satisfying $\lfloor x'\rfloor_Q = x$. 
    However, it is not aligned because $\lceil
    z^{\downarrow} \rceil_Q = \{x,x',y,z\}$ is not filtered. 
    Consider an interval $S = \{y,z,w,x'\}$ of $P$. 
    Recalling the definition \eqref{eq:contM}, 
    we find that $\Cont_Q \mathbb{I}_S$ is isomorphic to 
    \begin{equation}
    \begin{tikzpicture}   
    \node at (-1,0) {$\Cont_Q \mathbb{I}_S \ \cong $};
    \node at (1.8,0) {$\xymatrix@C10pt@R15pt{
    &k & \\ 
    & k^2 \ar[u]_-{
    \left[\begin{smallmatrix}
    1 & 1
    \end{smallmatrix}
    \right]
    } && \\ 
    k \ar[ur]^-{
    \left[\begin{smallmatrix}
    1 \\ 0
    \end{smallmatrix}
    \right]
    } & & \ar[ul]_-{
    \left[\begin{smallmatrix}
    0 \\ 1
    \end{smallmatrix}
    \right]
    } k. & 
    }$}; 
    \end{tikzpicture}  
    \end{equation}
    Then, it is indecomposable which is not an interval. 
    Consequently, it is not interval-decomposable.

    We also note that contraction functors may preserve interval-decomposability of modules even when $Q$ is not aligned. 
    Indeed, it holds if $Q$ is an $A_n$-type interior system of a finite poset. 
\end{remark}

\subsection{Application to interval resolutions}
\label{sec:Ind-Cont resolusion}

Let $P$ be a poset. 
We recall that $\mathscr{I}_P$ is the full subcategory consisting of interval-decomposable $P$-persistence modules that are finitely presentable. 

For $M \in \fprep P$, we refer to a right $\mathscr{I}_P$-approximation (resp., a minimal right $\mathscr{I}_P$-approximation) of $M$ by an \emph{interval approximation} (resp., \emph{interval cover}) of $M$.  
Namely, an interval approximation of $M$ is a morphism $\varphi \colon V\to M$ with $V\in \mathscr{I}_P$ such that the induced map $\Hom(U,V) \to \Hom(U,M)$ is surjective for any $U\in \mathscr{I}_P$. 
In addition, it is called an interval cover if moreover it is right minimal, i.e., every endomorphism $\kappa \colon V\to V$ satisfying $\varphi = \varphi \circ \kappa$ is an automorphism.

If $M$ admits an interval approximation, 
then it gives rise to an interval cover by taking its right minimal version. In particular, interval covers are determined uniquely up to isomorphism. 
Moreover, it is always surjective because $\mathscr{I}_P$ contains all finitely generated projectives by Proposition~\ref{prop:proj as interval}.

\begin{definition}
    For $0\neq M\in \fprep M$, 
    we say that $M$ admits an \emph{interval resolution}
    if there is an exact sequence
    \begin{equation}\label{eq:resol_M}
        \xymatrix@C10pt@R18pt{
        & \cdots \cdots \ar[rr]^{\varphi_3} \ar[rd]_{\bar{\varphi}_3} && V_2 \ar[dr]_{\bar{\varphi}_2} \ar[rr]^{\varphi_2} && V_1 \ar[rr]^-{\varphi_1} \ar[dr]_{\bar{\varphi}_1} && V_0 \ar[rr]^-{\varphi_0} && M \ar[rr] && 0 \\ 
         && \ker \varphi_2 \ar@{}[u]|{\circlearrowleft} \ar[ur]_{\iota_2} && \ker \varphi_1 
         \ar@{}[u]|{\circlearrowleft}
         \ar[ur]_{\iota_1} &&\ker \varphi_0 
         \ar@{}[u]|{\circlearrowleft} \ar[ur]_{\iota_0} &&&&
        }
    \end{equation}
    in $\fprep P$ such that 
    \begin{itemize}
        \item $V_i\in \mathscr{I}_P$ for all $i\geq 0$, 
        \item $\varphi_0$ is an interval cover of $M$, 
        \item $\bar{\varphi}_i$ is an interval cover of $\ker \varphi_{i-1}$, and $\iota_{i-1}$ is a canonical inclusion for all $i>0$. 
    \end{itemize}
    In this case, we call \eqref{eq:resol_M} \emph{interval resolution} of $M$ and simply write $\varphi_{\bullet}\colon V_{\bullet}\to M$ for it.
    Notice that it is uniquely determined up to isomorphism if exists. 
    Then, the \emph{interval resolution dimension} of $M$ is defined by the supremum 
    \begin{equation}
        \intresdim M := \sup \{m\in \mathbb{Z}_{\geq 0}\mid V_m \neq 0\}. 
    \end{equation} 
    In addition, the interval resolution dimension of the zero object is defined to be $0$. 
    Finally, if every object of $\fprep P$ admits an interval resolution (For example, this condition is satisfied when $P$ is finite), 
    then we define the \emph{interval resolution global dimension} of $P$ by   
    \begin{equation}
        \intresgldim P := \sup_{M\in \fprep P} \intresdim M. 
    \end{equation}
    We note that these supremum can be infinite.
\end{definition}

Now, we give one of our main results in this paper, 
which asserts that the induction functor over an aligned interior system preserves interval covers and interval resolutions. 

\begin{theorem}\label{thm:Ind-intervalcover}
    Let $Q\subseteq P$ be an aligned interior system. 
    Then, we have the following statements. 
    \begin{enumerate}[\rm (a)]
        \item For $M\in \fprep Q$, if it admits an interval cover $\varphi \colon V\to M$,
        then $\Ind_QM$ also admit an interval cover and given by $\Ind_Q\varphi \colon \Ind_QV\to \Ind_Q M$. 
        \item For $M\in \fprep Q$, 
        if it admits an interval resolution  $\varphi_{\bullet}\colon V_{\bullet}\to M$, 
        then $\Ind_QM$ also admits an interval resolution and given by 
        $\Ind_Q \varphi_{\bullet}\colon \Ind_Q V_{\bullet}\to \Ind_QM$.  
        In particular, we have 
        \begin{equation} \label{eq:intresdimM}
            \intresdim M = \intresdim \Ind_QM. 
        \end{equation}
        \item If every object of $\fprep Q$ admits an interval resolution, 
        then any object of the essential image $\mathcal{L}_Q$ (defined in Subsection \ref{subsec:Kan extension}) admits an interval resolution, and therefore
        \begin{equation}
            \intresgldim Q = \sup_{X\in \mathcal{L}_Q} \intresdim X.  
        \end{equation}
    \end{enumerate}
\end{theorem}

\begin{proof}
    For simplicity, we write $F := \Cont_Q$ and $G:=\Ind_Q$ with $F\dashv G$. 
    Then, Proposition \ref{prop:interval-Ind-Cont} asserts that they restrict to an adjoint pair 
    $F|_{\mathscr{I}_P} \dashv G|_{\mathscr{I}_Q} \colon \mathscr{I}_P \rightleftarrows \mathscr{I}_Q$.
    
    (a) Suppose that $M\in \fprep Q$ admits an interval approximation $\varphi\colon V\to M$ with $V\in \mathscr{I}_Q$. 
    Applying the functor $G$, we obtain a morphism $G(\varphi) \colon G(V)\to G(M)$. 
    Then, we have $G(V)\in \mathscr{I}_{P}$ by $V\in \mathscr{I}_Q$. 
    We claim that $G(\varphi)$ is an interval approximation of $G(V)$. 
    Consider any morphism $\psi\colon U\to G(M)$ with $U\in \mathscr{I}_P$. 
    From an isomorphism 
    \begin{equation}\label{eq:HomFUtoM}
        \Hom_Q(F(U), M)\cong \Hom_P(U,G(M)) 
    \end{equation}
    given by the adjoint pair $F\dashv G$, 
    we have the map $\psi'\colon F(U) \to M$ corresponding to $\psi$. 
    Since $\varphi$ is an interval approximation of $M$, 
    there exists $\rho\colon F(U) \to V$ such that $\psi' = \varphi\circ \rho$. 
    By an isomorphism 
    \begin{equation}
        \Hom_Q(F(U),V)\cong \Hom_P(U,G(V)), 
    \end{equation} 
    we obtain the map $\rho'\colon U\to G(M)$ corresponding to $\rho$. 
    By their construction, they satisfy $\psi = \rho' \circ G(\varphi)$. 
    Therefore, $G(\varphi)$ is an interval approximation. 
    \begin{equation*}
        \xymatrix@C28pt@R28pt{
        V \ar[r]^{\varphi} & M \\
        \ar@{}[ru]|(0.68){\circlearrowleft} & F(U) \ar[u]_{\psi'} \ar@{.>}[ul]^{\rho} 
        } 
        \quad \quad 
        \xymatrix@C28pt@R28pt{
        G(V) \ar[r]^{G(\varphi)} & G(M) \\
        \ar@{}[ru]|(0.68){\circlearrowleft} & U \ar[u]_{\psi} \ar@{.>}[ul]^{\rho'} 
        } 
    \end{equation*}

    Next, we assume that $\varphi$ is an interval cover. 
    We consider an endomorphism $\kappa \colon G(V)\to G(V)$ satisfying $G(\varphi) = G(\varphi) \circ \kappa$. 
    Applying $F$, we obtain $FG(\varphi) = FG(\varphi)\circ F(\kappa)$. 
    Since $FG(\varphi) \colon FG(V)\to FG(M)$ is also an interval cover by $FG \simeq \id_{\Rep Q}$, 
    it implies that $F(\kappa)$ is an automorphism of $FG(V)$. 
    Since there is an isomorphism 
    \begin{equation}
        \Hom_P(G(V),G(V))\cong \Hom_Q(FG(V),FG(V)) \quad
        (\alpha \mapsto F(\alpha)), 
    \end{equation} 
    $\kappa$ should be an automorphism of $G(V)$. 
    Thus, $G(\varphi)$ is an interval cover. 
    
    (b) By Propositions \ref{prop:idemp_embed}(a) and \ref{prop:basic Cont}(3), the functor $G$ is fully faithful and exact. 
    Then, we clearly have the assertion. 
    
    (c) This is immediate from (b). 
\end{proof}

The following corollary will be used later. 

\begin{corollary}\label{cor:reduce_to_aligned}
    Let $P$ be a poset such that every object in $\fprep P$ admits an interval resolution. 
    Suppose that $\mathcal{X}\subseteq \fprep P$ is a collection of finitely presentable $P$-persistence module satisfying the following two conditions:  
    \begin{itemize}
        \item $\intresgldim P = \sup_{X\in \mathcal{X}} \intresdim X$. 
        \item For any $X\in \mathcal{X}$, we have an aligned interior system $Q_X\subseteq P$ such that $X\in \mathcal{L}_{Q_X}$. 
    \end{itemize}
    Then, we have 
    \begin{equation}
        \intresgldim P = \sup_{X \in \mathcal{X}} \intresgldim Q_X. 
    \end{equation}
\end{corollary}

\begin{proof}
    This is immediate from Theorem \ref{thm:Ind-intervalcover}(c). 
\end{proof}

In practice, one can apply Theorem \ref{thm:Ind-intervalcover} to compute an interval cover/resolution of a given persistence module by the following strategy. 

\begin{remark}\label{rem:computational interval cover}
    Let $M$ be a finitely presentable $P$-persistence module. 
    By definition, $M$ is obtained as the cokernel of the form 
    $\coker\big(\bigoplus_{b\in \mathcal{J}} \mathbb{I}_{b^{\uparrow}} \to \bigoplus_{a\in \mathcal{I}} \mathbb{I}_{a^{\uparrow}}\big)$, 
    where $\mathcal{I}$ and $\mathcal{J}$ are finite multiset of elements of $P$. 
    We choose a subset $Q\subseteq P$ such that $Q$ contains both $\mathcal{I}, \mathcal{J}$ (as sets) and $Q$ is an aligned interior system of $Q^{\uparrow}$, as small as possible. 
    Indeed, such a poset always exists by Example \ref{ex:aligned posets}(1) but not necessarily finite. 
    In this situation, $M$ belongs to the full subcategory $\mathcal{L}_Q$ and satisfies $M\cong \Ind_Q\Cont_Q M$. 
    Applying Theorem \ref{thm:Ind-intervalcover}, we find that the the interval cover/resolution of $M$ can be directly obtained from those of $\Cont_QM$ (if exist). 
\end{remark}

%% file: 5_StabilizingProperty.tex
\section{A stabilizing property of interval resolution global dimensions}
\label{sec:stabilizing}

In this section, we restrict our attention to \emph{finite} posets, over which pfd persistence modules coincide with finitely presentable persistence modules (see Section \ref{sec:prelim}). 
For our purpose, we assume that finite posets are connected if otherwise specified throughout this section.

Let $P$ be a finite poset and $\Hasse(P)$ the Hasse diagram of $P$ (Definition \ref{def:Hasse}). 
For an element $x\in P$, we say that the \emph{degree} $\deg(x)$ of $x$ is a sum $\deg(x) := \outdeg(x) + \indeg(x)$, 
where $\outdeg(x)$ (resp., $\indeg(x)$) is the number of arrows in $\Hasse(P)$ starting (resp., ending) at $x$. 
In addition, we say that $x$ is a \emph{source} (resp., \emph{sink}) if $\indeg(x)=0$ (resp., $\outdeg(x)=0$), in other words, it is a minimum (resp., maximum) element of $P$.  
In addition, we say that $x$ is a \emph{leaf} if $\deg(x) = 1$. By definition, every leaf is either a source or a sink of $P$.

\begin{definition}\label{def:An-type segment}
    For $n\geq 2$, we say that an \emph{$A_n$-type segment} of $P$ is an interval $L\subseteq P$ 
    whose underlying graph 
    is an $A_n$-type graph of the form  
    \begin{equation}\label{eq:A-type segment}
        \xymatrix{
        \ell_1 \ar@{-}[r] & \ell_2 \ar@{-}[r]
        & \cdots \cdots \ar@{-}[r] & \ell_{n-1} \ar@{-}[r] & \ell_{n}}
    \end{equation} 
    with satisfying 
    $\deg(\ell_{i}) = 2$ for all $i \in \{2,\ldots,n-1\}$ and 
    $\outdeg(\ell_{i_1}) = 1$ (resp., $\indeg(\ell_{i_1}) = 1$) if $\ell_{i_1} \to \ell_{i_2}$ (resp., $\ell_{i_1} \leftarrow \ell_{i_2}$) 
    for $(i_1,i_2)\in \{(1,2), (n,n-1)\}$. 
    In addition, it is called \emph{equioriented $A_n$-type segment} if moreover
    $L$ is given by 
    \begin{equation}\label{eq:equiA-type segment}
        \xymatrix{\ell_1 \ar@{->}[r] & \ell_2 \ar@{->}[r]
        & \cdots \cdots \ar@{->}[r] & \ell_{n-1} \ar@{->}[r] & \ell_{n},} 
    \end{equation} 
    in which case $\outdeg(\ell_1) = \indeg(\ell_n)=1$ by definition. 
\end{definition}

The next result gives a way to transform a given finite poset into a smaller one with preserving interval resolution global dimensions. 

\begin{theorem}\label{thm:cont chains}
Let $P$ be a finite poset. Suppose that $P$ has an $A_n$-type segment $L$ as in Definition~\ref{def:An-type segment} with $n\geq 4$. 
Let $P'$ be a full subposet obtained from $P$ by removing elements $\ell_4,\ldots,\ell_n$. 
\begin{enumerate}[\rm (a)]
    \item If $L$ is equioriented, then we have 
\begin{equation}
    \intresgldim P = \intresgldim P'. 
\end{equation}
    \item If $\ell_n$ is a leaf of $P$, then we have 
\begin{equation}
    \intresgldim P = \intresgldim P'. 
\end{equation}
\end{enumerate}
\end{theorem}

The operations in Theorem \ref{thm:cont chains} can be illustrated in the following figure, where $P$ and $P'$ have the same interval resolution global dimensions. 
\begin{equation}
\begin{tabular}{ccccccc}
    &$P$& \quad \quad \quad \quad &$P'$ \\
    (a) &$\xymatrix@C=15pt@R=12pt{
    \ar[dr] &&&&& \ar@{<-}[dl] \\ 
    \ar[r] &\ell_1 \ar@{->}[r] & \ell_2 \ar@{->}[r]
    & \cdots \cdots \ar@{->}[r]  & \ell_{n} \ar[r] &\\
    \ar[ur] &&&&&\ar@{<-}[ul]}$   
    &  &
    $\xymatrix@C=15pt@R=12pt{
    \ar[dr] &&&& \ar@{<-}[dl] \\ 
    \ar[r] &\ell_1 \ar@{->}[r] & \ell_2 \ar@{->}[r]
     & \ell_{3} \ar[r] &\\ 
    \ar[ur] &&&&\ar@{<-}[ul]}$ 
    \\ \\ 
    (b)& 
    $\xymatrix@C=15pt@R=12pt{
    \ar[dr] &&&&&  \\ 
    \ar[r] &\ell_1 \ar@{-}[r] & \ell_2 \ar@{-}[r]
    & \cdots \cdots \ar@{-}[r]  & \ell_{n} &\\
    \ar[ur] &&&&&}$   
    &  &
    $\xymatrix@C=15pt@R=12pt{
    \ar[dr] &&&&  \\ 
    \ar[r] &\ell_1 \ar@{-}[r] & \ell_2 \ar@{-}[r]
     & \ell_{3}  &\\ 
    \ar[ur] &&&&}$    
\end{tabular}
\end{equation}

We will prove claims (a) and (b) of Theorem \ref{thm:cont chains} in Subsections \ref{subsec:Thm_chain(a)} and \ref{subsec:Thm_chain(b)} respectively. 
To do this, we need knowledge of representation theory of finite dimensional algebras, especially the Auslander-Reiten theory and tilting theory. 
We refer to \cite[Chapters IV,VI]{ASS} for these materials. 
Indeed, it is applicable to the study of persistence modules through incidence algebras as follows. 

Let $P$ be a finite poset and $\Hasse(P)$ the Hasse diagram of $P$. 
The \emph{incidence algebra} $k[P]$ of $P$ is defined by a quotient of a path algebra of $\Hasse(P)$ over $k$ modulo the two-sided ideal generated by all the commutative relations. 
It is well-known that there is an equivalence of abelian categories between 
the category $\rep P$ ($=\fprep P$) of pfd $P$-persistence modules and the category $\mod k[P]$ of finitely generated right $k[P]$-modules. 
Thus, we will identify the two categories. 
Under this identification, the classes of projectives, injectives, simples, indecomposables, etc. coincide in both categories. 

We correct some notations. 
For an element $x\in P$, 
let $\rS_x := \mathbb{I}_{\{x\}}$ be the simple module corresponding to $x$. 
In addition, let $\rP_x := \mathbb{I}_{x^{\uparrow}}$ and $\rI_x := \mathbb{I}_{x^{\downarrow}}$ be the indecomposable projective and indecomposable injective modules corresponding to $x$ respectively. 
Next, for a given pfd $P$-persistence module $M$, we recall that the projective dimension $\projdim M$ of $M$ is defined to be the length of a minimal projective resolution of $M$. 
The injective dimension $\injdim M$ is defined dually. 
Finally, we denote by $\tau$, $\tau^{-1}$ the Auslander-Reiten translations of $\rep P$ (through the equivalence with $\mod k[P]$). 
They give a correspondence 
mapping $M$ to $\tau M$ and $\tau^{-1}M$ respectively in $\rep P$, and $\tau M = 0$ (resp., $\tau^{-1}M =0)$ if and only if $M$ is projective (resp., injective).

\subsection{Proof of Theorem \ref{thm:cont chains}(a)} \label{subsec:Thm_chain(a)}
Firstly, we recall some formulas for computing interval resolution global dimensions over finite posets. 
In particular, we use a description of irreducible morphisms relative to intervals.

Let $P$ be a finite poset. We describe intervals of $P$ by using the following notations due to \cite{BBH1}. 

\begin{definition}\label{def:finite intervals}
For two subsets $A,B\subseteq P$, we define convex subsets
\begin{enumerate}[\rm (1)]
    \item ${\langle} A,B {\rangle} := A^{\uparrow}\cap B^{\downarrow}$. 
    \item ${\langle} A,B{\langle} := 
    A^{\uparrow}\setminus B^{\uparrow} 
    $.
    \item ${\rangle} A,B{\rangle} :=  
    B^{\downarrow}\setminus A^{\downarrow} 
   $.     
\end{enumerate}
\end{definition}

It is clear that every interval $S\in \Int(P)$ can be written as $S= {\langle} \min(S),\max(S) {\rangle}$ since $P$ is finite. 

Next, we discuss irreducible morphisms between interval modules. 
For two intervals $S,T \in \Int(P)$, a morphism $\alpha \colon \mathbb{I}_T \to \mathbb{I}_S$ is said to be \emph{irreducible} (relative to interval modules) if it is neither a section nor a retraction 
and any factorization of $\alpha$ of the form $\mathbb{I}_T \xrightarrow{\beta} M \xrightarrow{\gamma} \mathbb{I}_S$ with interval-decomposable $M$ implies that $\beta$ is a section or $\gamma$ is a retraction. 
In this situation, we clearly have $S\neq T$. 

\begin{proposition} \label{prop:irred_interval}
\cite[Proposition 6.8]{BBH2}
    Let $T$ and $S$ be intervals of $P$.
    Then, we have an irreducible morphism $\alpha\colon \mathbb{I}_{T} \to \mathbb{I}_{S}$ if and only if one of the following holds. 
    \begin{enumerate}[\rm (a)]
        \item $\alpha$ is surjective and there exists $x\in P$ such that $S \lessdot x$ (Definition \ref{def:Hasse}) and $T = S \cup {\rangle} S, x {\rangle}$. 
        \item $\alpha$ is injective and there exists $z\in \min(S)$ such that $S = T \cup {\langle} z, T {\langle}$ and $T$ is a connected component of 
        $S\setminus \{z\}$. 
    \end{enumerate}
    Moreover, such an irreducible morphism $\alpha$ is given by (up to scalar multiplication)
    \begin{equation*}
        \alpha(x) = \begin{cases}
            \id_k & \text{if $x\in S\cap T$,} \\ 
            0 & \text{otherwise.}            
        \end{cases}
    \end{equation*}

\end{proposition}

For an interval $S\subseteq P$, let $\mathcal{A}_S$ be the set of all irreducible morphisms $\alpha\colon \mathbb{I}_{S_{\alpha}} \to \mathbb{I}_S$ with intervals $S_{\alpha}\in \Int(P)$. 
Notice that it is a finite set since $P$ is finite. 
Then, we define $\Gamma_S$ in $\rep P$ by an exact sequence 
\begin{equation}\label{eq:GammaS_exact}
    0 \to \Gamma_S \to \bigoplus_{\alpha\in \mathcal{A}_S} \mathbb{I}_{S_{\alpha}} \xrightarrow{(\alpha)_{\alpha\in \mathcal{A}_S}} \mathbb{I}_S. 
\end{equation}

We summarize below known formulas for computing the interval resolution global dimension of a given finite poset $P$. 
In fact, it is a direct consequence of \cite[Proposition 3.15]{AENY}. 

\begin{proposition}\label{prop:formulas}
    For a finite poset $P$, we have the equalities among all the following values. 
\begin{enumerate}[\rm (i)]
    \item $\intresgldim P$.
    \item $\sup\{\intresdim \tau \mathbb{I}_S\mid S\in \Int(P)\}$.
    \item $\sup\{\intresdim \Gamma_S \mid S\in \Int(P)\}$.
    \item $\intresgldim P^{\op}$.
\end{enumerate}
    Moreover, all the above values are always finite. 
\end{proposition}

In particular, we obtain the following formula. 

\begin{lemma}\label{lem:supQS}
    If every interval $S\in \Int(P)$ admits an aligned interior system $Q_S$ of $P$ such that $\Gamma_S \in \mathcal{L}_{Q_S}$, then we have $\intresgldim P = \sup_{S\in \Int(P)}\intresgldim Q_S$.  
\end{lemma}

\begin{proof}
    This is straightforward from Corollary \ref{cor:reduce_to_aligned} and Proposition \ref{prop:formulas}. 
\end{proof}

Under these notation, we prove Theorem \ref{thm:cont chains}(a). 

\begin{proof}[Proof of Theorem \ref{thm:cont chains}(a)]
    By induction on $n$, it suffices to show the case $n = 4$.
    Thus, let $L$ be an equioriented $A_4$-type segment 
    $L=(\ell_1\to \ell_2 \to \ell_3\to \ell_4)$ as in Definition \ref{def:An-type segment}. 
    For $i\in \{2,3,4\}$, let $Q_i:=P\setminus \{\ell_i\}$ be a full subposet of $P$. 
    Then, $Q_i$ is an aligned interior system of $P$. 
    In fact, an inclusion map $Q_i\hookrightarrow P$ has a right adjoint 
    \begin{equation}  
        \lfloor a \rfloor_{Q_i} := 
        \begin{cases} 
        \ell_{i-1} & \text{if $a = \ell_i$,} \\  
        a & \text{otherwise},
    \end{cases}
    \quad \text{with} \quad 
    \lceil y \rceil_{Q_i} = \begin{cases}
        \{\ell_{i-1},\ell_i\} & \text{if $y=\ell_{i-1}$,} \\
        \{y\} & \text{otherwise,}
    \end{cases}
    \end{equation}
    for any $a\in P$ and $y\in Q_i$. 
    In addition, $Q_i$ clearly satisfies the conditions (AL1) and (AL2) in Definition \ref{def:aligned intersys}. 
    Therefore, we have an adjoint pair $\Cont_{Q_i} \dashv \Ind_{Q_i} \colon \rep P \rightleftarrows \rep Q_i$. 
    We write $\mathcal{L}_i := \mathcal{L}_{Q_i}$ for the essential image of $\Ind_{Q_i}$ on $\rep Q_i=\fprep Q_i$. 
    By \eqref{eq:Ind as floor}, a given pfd $P$-persistence module $M$ belongs to $\mathcal{L}_{i}$ if and only if 
    $M(\ell_{i-1}\leq \ell_i)$ is an isomorphism.

    The poset $P'$ in the statement is naturally isomorphic to $Q_i$'s as posets by our construction. 
    So, we write $q$ for their interval resolution global dimensions. 
    
    We claim that, for every $S\in \Int(P)$, 
    we have $\Gamma_S\in \mathcal{L}_{i}$ for some $i\in \{2,3,4\}$.
    After that, we obtain the desired equality by
    \begin{equation}
        \intresgldim P \overset{\rm Lem.\,\ref{lem:supQS}}{=} 
        \sup_{i\in \{2,3,4\}} \intresgldim Q_i = q.  
    \end{equation}

    We use the following observation. 
    Now, we recall that 
    $\mathcal{A}_S$ is the set of irreducible morphisms terminated at $\mathbb{I}_S$ for $S\in \Int(S)$ and of the form $\alpha \colon \mathbb{I}_{S_{\alpha}} \to \mathbb{I}_S$ with $S_{\alpha}\in \Int(P)$.
    
\begin{lemma}\label{lem:Gamma_S_where}
    Let $S\in \Int(P)$ be an interval and $i\in \{2,3,4\}$. 
    Then, $\Gamma_S$ belongs to $\mathcal{L}_i$ if the following condition is satisfied. 
    \begin{itemize}
        \item For any $T\in \{S\}\cup \{S_{\alpha}\}_{\alpha\in \mathcal{A}_S}$, 
        we have $\ell_{i-1}\in T$ if and only if $\ell_i\in T$. 
    \end{itemize}     
\end{lemma}

\begin{proof}
    As we mentioned before, $\Gamma_S \in \mathcal{L}_i$ if and only if 
    $\Gamma_S(\ell_{i-1}\leq \ell_i)\colon \Gamma_S(\ell_{i-1})\to \Gamma_S(\ell_i)$ is an isomorphism. 
    By the exact sequence \eqref{eq:GammaS_exact}, we obtain a commutative diagram 
    \begin{equation*}
        \xymatrix@C68pt{
        0 \ar[r] 
        &
        \Gamma_S(\ell_{i-1}) \ar[r] 
        &
        \displaystyle
        \bigoplus_{\alpha\in \mathcal{A}_S} \mathbb{I}_{S_{\alpha}}(\ell_{i-1}) \ar[r]^-{(\alpha(\ell_{i-1}))_{\alpha\in \mathcal{A}_S}} 
        & 
        \mathbb{I}_S(\ell_{i-1}) \\ 
        0 \ar[r] 
        &
        \Gamma_S(\ell_i) \ar[r] \ar@{<-}[u]^{\Gamma_S(\ell_{i-1}\leq \ell_i)} \ar@{}[ur]|{\circlearrowleft}
        & 
        \displaystyle
        \bigoplus_{\alpha\in \mathcal{A}_S} \mathbb{I}_{S_{\alpha}}(\ell_i) \ar[r]^-{(\alpha(\ell_i))_{\alpha\in \mathcal{A}_S}} \ar@{<-}[u]^-{h} \ar@{}[ur]|{\circlearrowleft} 
        & 
        \mathbb{I}_S(\ell_i). \ar@{<-}[u]^{\mathbb{I}_S(\ell_{i-1}\leq \ell_i)}
        }
    \end{equation*}
    where $h$ is given by a diagonal matrix whose diagonal entries are $\mathbb{I}_{S_{\alpha}}(\ell_{i-1}\leq \ell_i)$ for all $\alpha\in \mathcal{A}_S$. 
    
    For $T\in \{S\}\cup \{S_{\alpha}\}_{\alpha\in \mathcal{A}_S}$, 
    we have either $\ell_{i-1},\ell_i\in T$ or 
    $\ell_{i-1},\ell_i\not\in T$ by our assumption. 
    Thus, we get $\mathbb{I}_T(\ell_{i-1}) = \mathbb{I}_T(\ell_i) = k$ and $\mathbb{I}_T(\ell_{i-1}\leq \ell_i) = \id_k$ in the former case, and 
    $\mathbb{I}_T(\ell_{i-1}) = \mathbb{I}_T(\ell_{i})=0$ in the latter case. 
    In particular, $h$ is an isomorphism. 
    Therefore, $\Gamma_S(\ell_{i-1} \leq \ell_i)$ is an isomorphism by the universality of kernels. 
\end{proof}

We continue our proof of (a). 
We divide cases into (I)--(IV) as follows. 

\bigskip \noindent 
(I) We assume that $\ell_2,\ell_3\in S$. 

\begin{itemize}
    \item [(I-1)] We additionally assume that $\ell_1\in S$. 
For any $\alpha\in \mathcal{A}_S$, we have the following. 
    \begin{itemize}
        \item If $\alpha$ is surjective, then $\ell_2,\ell_3\in S_{\alpha}$. 
        This is because $S\subseteq S_{\alpha}$.  
        \item If $\alpha$ is injective, then $\ell_2 \in S_{\alpha}$ if and only if $\ell_3\in S_{\alpha}$. 
        Indeed, since 
        $\ell_2,\ell_3\not\in \min(S)$
        by $\ell_1\in S$, by the shape of $L$, 
        we observe that elements $\ell_2$ and $\ell_3$ always lie in the same connected component of $S\setminus \{z\}$ for any $z\in \min(S)$. 
        Then, we get the claim since $S_{\alpha}$ is one of connected components of this form by Proposition \ref{prop:irred_interval}. 
    \end{itemize}
Thus, the assumption of Lemma \ref{lem:Gamma_S_where} is satisfied with $\Gamma_S\in \mathcal{L}_3$. 

\item [(I-2)]  If $\ell_4\in S$, then we get $\Gamma_S\in \mathcal{L}_4$ by a similar discussion to the case (I-1). 

\item [(I-3)] If $\ell_1,\ell_4\not\in S$, then $S = \{\ell_2,\ell_3\}$ holds. 
In this case, $\mathcal{A}_S$ consists of two morphisms 
$\varphi\colon \mathbb{I}_{\{\ell_3\}} \to \mathbb{I}_S$ and 
$\psi\colon \mathbb{I}_{\{\ell_2,\ell_3,\ell_4\}} \to \mathbb{I}_S$. 
Computing the kernel of $(\varphi,\psi)\colon \mathbb{I}_{\{\ell_3\}} \oplus \mathbb{I}_{\{\ell_2,\ell_3,\ell_4\}} \to \mathbb{I}_S$, 
we get $\Gamma_S\cong \mathbb{I}_{\{\ell_3,\ell_4\}}\in \mathcal{L}_4$. 
\end{itemize}

\bigskip \noindent
(II) We assume that $\ell_2\not\in S$ and $\ell_3\in S$. 
In this case, we have $\ell_1\not\in S$ by the convexity of $S$. 
For any $\alpha\in \mathcal{A}_S$, we have the following. 
\begin{itemize}
    \item [\textbf{--}] If $\alpha$ is surjective, then $\ell_1,\ell_2\not\in S_{\alpha}$. 
    Indeed, by Proposition \ref{prop:irred_interval}, we have 
    $S_{\alpha} = S\cup {\rangle}S, x {\rangle}$
    for some element $x\in P$ such that $S\lessdot x$. 
    Then, the claim follows from $\ell_1,\ell_2 \in S^{\downarrow}$ by $\ell_3\in S$. 
    \item [\textbf{--}] If $\alpha$ is injective, then $\ell_1,\ell_2\not\in S_{\alpha}$ since $S_{\alpha} \subseteq S$. 
\end{itemize}
Therefore, we have $\Gamma_S(\ell_1) = \Gamma_S(\ell_2) = 0 $ and $\Gamma_S \in \mathcal{L}_2$. 

\bigskip \noindent
(III) We assume that $\ell_2\in S$ and $\ell_3\not\in S$. 
In this case, we have $\ell_2\in \max(S)$ and $\ell_4\not\in S$ by the convexity of $S$.  
We consider a natural injection  
$\varphi\colon \mathbb{I}_{\{\ell_2\}} \to \mathbb{I}_{S}$ and 
a natural surjection $\psi \colon \mathbb{I}_{S\cup \{\ell_3\}} \to \mathbb{I}_{S}$. 
By Proposition \ref{prop:irred_interval}, we clearly have $\psi\in \mathcal{A}_S$. 

\begin{itemize}
    \item [(III-1)] If $\ell_1\not\in S$, then $S=\{\ell_2\}$ and $\mathcal{A}_S = \{\psi\}$. 
In this case, we have $\Gamma_S \cong \mathbb{I}_{\{\ell_3\}}\in \mathcal{L}_2$. 
\item [(III-2)] 
We assume that $\ell_1\in S$ and $\varphi\in \mathcal{A}_S$. 
In this case, the interval $S$ must be $S = \{\ell_1,\ell_2\}$ 
by a characterization of injective irreducible morphisms in Proposition \ref{prop:irred_interval}. 
In addition, we have $\mathcal{A}_S = \{\psi,\varphi\}$ by Proposition \ref{prop:irred_interval}.
Thus, we have  
$\Gamma_S \cong \mathbb{I}_{\{\ell_2,\ell_3\}}\in \mathcal{L}_3$.

\item [(III-3)]  We assume that $\ell_1\in S$ and $\varphi \not\in \mathcal{A}_S$.
For any $\alpha\in \mathcal{A}_S$, we have the following. 
\begin{itemize}
    \item If $\alpha$ is surjective, then $\ell_1,\ell_2\in S_{\alpha}$ since $S\subseteq S_{\alpha}$. 
    \item If $\alpha$ is injective, then $\ell_1\in S_{\alpha}$ if and only if $\ell_2\in S_{\alpha}$. 
    In fact, $\varphi\not\in \mathcal{A}_S$ implies $\ell_1\not\in \min(S)$, and therefore 
    elements $\ell_1$ and $\ell_2$ lie in any connected component of $S\setminus \{z\}$ with $z\in \min(S)$ simultaneously. 
    Thus, we get the claim by Proposition \ref{prop:irred_interval}.
\end{itemize}
Therefore, we have $\Gamma_S \in \mathcal{L}_2$ by Lemma \ref{lem:Gamma_S_where}.
\end{itemize}

\bigskip \noindent
(IV) Finally, we assume that $\ell_2,\ell_3\not\in S$. 
In this case, we have either $\ell_1\not\in S$ or $\ell_4\not\in S$ by the convexity of $S$. 

\begin{itemize}
    \item [(IV-1)]  If $\ell_1\in S$, then we have $\ell_4\not\in S$. 
For any $\alpha\in \mathcal{A}_S$, we have the following. 
\begin{itemize}
    \item If $\alpha$ is surjective, then $\ell_3,\ell_4\not\in S_{\alpha}$ from a description of $S_{\alpha}$ in Proposition \ref{prop:irred_interval}.
    \item If $\alpha$ is injective, then $\ell_3,\ell_4\not\in S_{\alpha}$ since $S_{\alpha} \subseteq S$. 
\end{itemize}
Therefore, $\Gamma_S(\ell_3) = \Gamma_S(\ell_4) = 0$ and $\Gamma_S\in \mathcal{L}_4$. 

    \item [(IV-2)] If $\ell_1\not\in S$, then we have the following. 
\begin{itemize}
    \item If $\alpha$ is surjective, then $\ell_2 \in S_{\alpha}$ if and only if $\ell_3\in S_{\alpha}$. 
    Indeed, by $\ell_1\not\in S$, we find that elements $\ell_2$ and $\ell_3$ lie in ${\rangle} S, x {\rangle}$ for $x\in P$ such that $S\lessdot x$ simultaneously. 
    Thus, we get the claim by a description of $S_{\alpha}$ in Proposition \ref{prop:irred_interval}. 
    \item If $\alpha$ is injective, then $\ell_2,\ell_3\not\in S_{\alpha}$ since $S_{\alpha} \subseteq S$. 
\end{itemize}
Therefore, we have $\Gamma_S\in \mathcal{L}_3$ by Lemma \ref{lem:Gamma_S_where}.
\end{itemize}

\bigskip 
All the cases have been considered in the above discussion (I)--(IV). 
We finish the proof of Theorem \ref{thm:cont chains}(a). 
\end{proof}

\subsection{Proof of Theorem \ref{thm:cont chains}(b)} \label{subsec:Thm_chain(b)}

In this section, we prove Theorem \ref{thm:cont chains}(b). 
Let $P$ be a finite poset.
For an element $a$ that is a sink or a source of $\Hasse(P)$, 
we define a new poset $\mu_a(P)$ whose Hasse diagram is obtained from $\Hasse(P)$ by reversing all arrows incident to $a$ and call it \emph{reflection} of $P$ at $a$. 

In the following result (Proposition \ref{prop:mutation_poset} and Lemma \ref{lem:endo_U}), we just give claims for sinks because the claims for sources can be obtained dually. 

\begin{proposition}\label{prop:mutation_poset}
    Let $a$ be a sink of $\Hasse(P)$ with $m:=\indeg(a)>0$. 
    If $m \in \{1,2\}$ and the graph $\Hasse(P)\setminus \{a\}$ consists of exactly $m$ connected components, 
    then we have 
    \begin{equation}
        \intresgldim P = \intresgldim \mu_a(P). 
    \end{equation}
\end{proposition}

To prove this, we need some preparations. 
For a sink $a$ in $\Hasse(P)$, we consider a pfd $P$-persistence module 

\begin{equation}
    U_a := \tau^{-1} \rS_a \oplus \bigoplus_{x\neq a} \rP_x 
\end{equation}
known as the APR-tilting module at $a$ \cite[Example 6.8(c)]{ASS}.

\begin{lemma} \label{lem:endo_U}
    Let $a$ be a sink of $\Hasse(P)$ with $m:=\indeg(a)>0$. 
    If $\Hasse(P) \setminus \{a\}$ consists of exactly $m$ connected components, then the following statements hold. 
    \begin{enumerate}[\rm (1)]
        \item We have ${\rm inj.dim} \rS_a =1$.
        \item $\End_P(U_a)^{\rm op} \cong k[\mu_a(P)]$ holds as $k$-algebras.
    \end{enumerate}
\end{lemma}

\begin{proof}
    (1) By our assumption, 
    we have elements $b_1,\ldots,b_m\in P$ such that 
    there is an arrow $b_i \rightarrow a$ in the Hasse diagram of $P$.  
    In addition, $\Hasse(P) \setminus \{a\}$ is a disjoint union of $m$ connected components containing $b_1,\ldots,b_m$ respectively. 
    In this situation, we have an exact sequence 
    \begin{equation}
        0\to \rS_a \to \rI_a \to \bigoplus_{i=1}^m \rI_{b_i} \to 0  
    \end{equation}
    that provides a minimal injective resolution of $\rS_a$. 
    This shows that $\injdim \rS_a = 1$.
    
    (2) Let $b_1,\ldots,b_m$ be elements as above. 
    We compute the endomorphism algebra $\End_P(U_a)$. 
    
    Firstly, we have 
    \begin{equation}
        \Hom_P(\rP_x, \rP_y) \cong \begin{cases}
            k & \text{if $x\geq y$,} \\
            0 & \text{otherwise}
        \end{cases}
    \end{equation}
    for two elements $x,y$ in $P$. 
    In particular, since $a$ is a sink of $P$, 
    we have that $\rS_a = \rP_a$ and $\Hom_P(\rP_x,\rS_a) = 0$ for all $x\neq a$. 
    In addition, we have $\Hom_P(\rP_x,\rP_y) = 0$ whenever they lie in different components of $\Hasse(P) \setminus \{a\}$. 
    
    Secondly, for any $x\in P$, 
    we have 
    \begin{equation}
        \Hom_P(\tau^{-1}\rS_a, \rP_x)\cong \Hom_k(\Ext^1_P(\rP_x, \rS_a),k) = 0,
    \end{equation} 
    by the Auslander-Reiten formula \cite[Chapter IV, Theorem 2.13]{ASS} with $\injdim \rS_a = 1$ by (1). 
    
    Finally, we have an almost split sequence 
    \begin{equation}\label{eq:almost_seq}
        0 \to \rS_a \to \bigoplus_{i=1}^m \rP_{b_i} 
        \to
        \tau^{-1}\rS_a \to 0. 
    \end{equation}
    Then, we apply $\Hom_P(\rP_x,-)$ to this sequence and obtain an isomorphism
    \begin{equation}
        \Hom_P(\rP_x, \tau^{-1}\rS_a) 
        \cong 
        \begin{cases}
            \Hom_P(\rP_x, \rP_{b_i}) & \text{if $x \geq b_i$},  \\ 
            0 & \text{otherwise} 
        \end{cases}
    \end{equation}
    for every $x\neq a$. 

    By the above argument, 
    we find that the Gabriel quiver of $\End_P(U_a)^{\rm op}$ coincides with the Hasse diagram of $\mu_a(P)$, and relations for this algebra are generated by commutative relations. 
    Consequently, $\End_P(U_a)^{\rm op}$ is isomorphic to the incidence algebra of $\mu_a(P)$. 
\end{proof}

Next, we recall the notion of torsion pairs with the following notation. 
For a full subcategory $\mathcal{C}$ of $\rep P$, let 
\begin{eqnarray}
    \mathcal{C}^{\perp} := \{M\in \rep P \mid \Hom_P(\mathcal{C}, M) = 0\} \qand
    {}^{\perp}\mathcal{C} := \{M\in \rep P \mid \Hom_P(M,\mathcal{C}) = 0\} 
\end{eqnarray}
be full subcategories of $\rep P$. 
We say that a pair $(\mathcal{T},\mathcal{F})$ of full subcategories is a \emph{torsion pair} if $\mathcal{T} = {}^{\perp} \mathcal{F}$ and $\mathcal{F}=\mathcal{T}^{\perp}$. 
In this case, $\mathcal{T}$ and $\mathcal{F}$ are called a torsion class and a torsion-free class respectively. 
A torsion pair $(\mathcal{T},\mathcal{F})$ is called \emph{splitting} if every indecomposable module $M\in \rep P$ lies in either $\mathcal{T}$ or $\mathcal{F}$.

Under these notations, we prove Proposition \ref{prop:mutation_poset}.

\begin{proof}[Proof of Proposition \ref{prop:mutation_poset}]
    We will give a proof for the case $m = 2$ because the case $m=1$ can be shown similarly.
    
    Suppose that $b,c\in P$ such that $(b \rightarrow a \leftarrow c)$ in the Hasse diagram of $P$. 
    By our assumption, the graph $\Hasse(P)\setminus\{a\}$ is a disjoint union of two connected components, one contains $b$ and the other contains $c$. 
    
    Now, let $P' := \mu_a(P)$. 
    Using Lemma \ref{lem:endo_U}(2), 
    we have equivalences 
    \begin{equation}
        \mod \End_P(U_a)^{\rm op} \simeq \mod k[P'] \simeq \rep P'
    \end{equation}
    of abelian categories. We will identify these categories here. 
    
    In the category $\rep P'$, 
    we denote by $\rS'_x$ (resp., $\rP'_x$ and $\rI'_x$) 
    the simple (resp., indecomposable projective and indecomposable injective) 
    $P'$-persistence module corresponding to $x\in P'$. 
    In addition, for a given interval $T$ of $P'$, 
    we write the corresponding interval module in $\rep P'$ by $\mathbb{I}'_{T}$. 

    In this situation, the APR-tilting module $U_a$ induces \emph{splitting} torsion pairs 
    $(\mathcal{T}_a, \mathcal{F}_a)$ in $\rep P$ and $(\mathcal{X}_a, \mathcal{Y}_a)$ in $\rep P'$, 
    where $\mathcal{F}_a$ (resp., $\mathcal{T}_a$) is the additive full subcategory generated by the simple-projective module $\rS_a=\mathbb{I}_{\{a\}}$ 
    (resp., all the remaining indecomposable modules) in $\rep P$, 
    whereas $\mathcal{X}_a$ (resp., $\mathcal{Y}_a$) is the additive full subcategory generated by the simple-injective module $\rS_a'=\mathbb{I}'_{\{a\}}$ 
    (resp., all the remaining indecomposable modules) in $\rep P'$. 
    In fact, $(\mathcal{T}_a,\mathcal{F}_a)$ is always splitting, 
    and $(\mathcal{X}_a,\mathcal{Y}_a)$ is also splitting by $\injdim \rS_a = 1$ (see \cite[Chapter~VI, Theorem~5.6]{ASS}). 
    By Brenner-Butler theorem \cite[Chapter VI, Theorem 3.8]{ASS}, they are related to each other by equivalences 
    \begin{equation}\label{eq:BBthm}
        E := \Hom_P(U_a,-)\colon \mathcal{T}_a \xrightarrow{\sim} \mathcal{Y}_a 
        \qand 
        F := \Ext^1_P(U_a,-) \colon \mathcal{F}_a \xrightarrow{\sim} \mathcal{X}_a
    \end{equation}
    with quasi-inverses $G := -\otimes_{P'} U_a$ and $H := \Tor_1^{P'}(-,U_a)$ respectively (These functors are naturally defined between $\rep P$ and $\rep P'$). 
    In this situation, we have that $F(M) = 0$ for all $M\in \mathcal{T}_a$ and $E(N)=0$ for all $N\in \mathcal{F}_a$. 
    Therefore, the correspondence
    \begin{equation*}
        M\mapsto M^{\vee} :=
        \begin{cases}
            E(M)\in \mathcal{Y}_a & \text{if $M\in \mathcal{T}_a$,} \\ 
            F(M)\in \mathcal{X}_a & \text{if $M\in \mathcal{F}_a$}
        \end{cases}
    \end{equation*}
    provides a bijection between the set of isomorphism classes of indecomposables in $\rep P$ and in $\rep P'$. 
    
    Furthermore, the above correspondence restricts to 
    a bijection between interval modules as follows: 
    For an interval $S\in \Int(P)$, we set 
    \begin{equation}\label{eq:BB_bij}
        S^{\vee} := 
        \begin{cases}
            \{a\} & \text{if $S = \{a\}$}, \\
            S\setminus\{a\} & \text{if $a,b \in S$, $c\notin S$ or $a,c \in S$, $b\notin S$}, \\
            S\cup \{a\} & \text{if $b \in S$, $a,c\notin S$ or $c \in S$, $a,b\notin S$}, \\
            S & \text{otherwise},  
        \end{cases}
    \end{equation}
    which is clearly an interval of $P'$. 
    Then, a direct calculation shows that 
    $(\mathbb{I}_{S})^{\vee}\cong \mathbb{I}'_{S^{\vee}}$ holds in $\rep P'$.
    In particular, $E$, $F$ and $(-)^{\vee}$ preserve interval-decomposability of modules. 
    
    Now, we prove the claim by showing that 
    $M$ and $M^{\vee}$ have the same interval resolution dimension for any indecomposable module $M\in \rep P$. 
    Notice that $M$ lies in either $\mathcal{T}_a$ or $\mathcal{F}_a$ since $(\mathcal{T}_a,\mathcal{F}_a)$ is splitting. 
    
    If $M\in \mathcal{F}_a$, then we have $M \cong \rS_{a} = \mathbb{I}_{\{a\}}$ and $M^{\vee} \cong \mathbb{I}_{\{a\}}'$ by \eqref{eq:BB_bij} with $\{a\}^{\vee} = \{a\}$. 
    In this case, both are interval modules and have interval resolution dimension zero. 
    
    Next, we assume that $M\in \mathcal{T}_a$. 
    We consider an interval cover $\varphi \colon V\to M$ with $V\in \mathscr{I}_P$ and an exact sequence 
    \begin{equation} \label{eq:shortM}
        0\to \ker \varphi \xrightarrow{\iota} V \xrightarrow{\varphi} M \to 0. 
    \end{equation} 
    Applying the functor $E=\Hom_P(U_a,-)$ to this sequence, 
    we obtain a long exact sequence 
    \begin{equation} \label{eq:long_exact}
        0 \to E(\ker \varphi) \xrightarrow{E(\iota)} E(V) \xrightarrow{E(\varphi)} E(M) \to F(\ker \varphi) \to F(V) \to F(M) = 0. 
    \end{equation}
    If $E(\varphi)=0$, then we have $E(M)\subseteq F(\ker \varphi)$. 
    It implies that $E(M)$ is a direct sum of $\rS'_a$ and $E(M)\in \mathcal{X}_a$, a contradiction. 
    Thus, we have $E(\varphi)\neq 0$. 
    
    We claim that $E(\varphi)$ is an interval cover of $E(M)$. 
    Firstly, $E(V)\neq 0$ is interval-decomposable since so is $V$. 
    Secondly, we take an interval $T\in \Int(P')$ and a non-zero morphism $\psi\colon \mathbb{I}'_T \to E(M)$. 
    In this case, we have $\mathbb{I}'_T\in \mathcal{Y}_a$ by $E(M)\in \mathcal{Y}_a$. 
    Applying $G$, we get a morphism $G(\psi) \colon G(\mathbb{I}'_T) \to GF(M)= M$. 
    Since $G(\mathbb{I}'_T)$ is an interval module and $\varphi$ is an interval cover, there is a morphism $\rho\colon G(\mathbb{I}'_T)\to V$ such that $G(\psi) = \varphi\circ\rho$. 
    Applying $E$ again, it gives $\psi= EG(\psi) = E(\varphi)\circ E(\rho)$. This shows that $E(\varphi)$ is an interval approximation of $E(M)$. 
    In particular, it is surjective. 
    Then, \eqref{eq:long_exact} gives $F(\ker \varphi)\cong F(V)$. 
    If they are non-zero, then $\ker\varphi$ and $V$ must have a common direct summand $K$ 
    such that the restriction of $\iota$ to $K$ is the identity map. 
    It contradicts the right minimality of $\varphi$. 
    Thus, we have $F(\ker \varphi)= F(V)=0$ and $\ker\varphi$, $V\in \mathcal{T}_a$. 
    In particular, the sequence \eqref{eq:shortM} is a short exact sequence in $\mathcal{T}_a$. 
    Since $E$ gives an equivalence between the full subcategories $\mathcal{T}_a$ and $\mathcal{Y}_a$, the right minimality of $E(\varphi)$ is immediate from that of $\varphi$. 
    Therefore, $E(\varphi)$ is an interval cover. 
    
    From the above discussion, if $\varphi^{\bullet} \colon V^{\bullet}\to M$ is an interval resolution of $M$, then $V^j\in \mathcal{T}_a$ holds for all $j$. 
    Then, $E(\varphi)^{\bullet}\colon E(V)^{\bullet} \to F(M)$ is an interval resolution of $E(M)$ such that $E(V^{j})\in \mathcal{Y}_a$ for all $j$. Moreover, $V^j=0$ if and only if $E(V^j)=0$. 
    Therefore, the interval resolution dimension of $M$ and $E(M)=M^{\vee}$ coincide. 

    It finishes the proof. 
\end{proof}

Now, we are ready to prove Theorem \ref{thm:cont chains}(b).

\begin{proof}[Proof of Theorem \ref{thm:cont chains}(b)]
Suppose that $L$ is an $A_n$-type segment with $n\geq 4$ such that $\ell_n$ is a leaf of $P$. 
We can consider a reflection at an arbitrary sink or source in $\{\ell_2,\ldots,\ell_n\}$. 
Since it satisfies the assumption of Proposition \ref{prop:mutation_poset} (or its dual), its interval resolution global dimension is the same as that of $P$. 
By applying such reflections iteratively, we can transform $L$ into the equioriented $A_n$-type segment as in \eqref{eq:equiA-type segment} with preserving interval resolution global dimensions. 
Finally, we apply Theorem \ref{thm:cont chains}(a) and obtain the desired equality. 
\end{proof}

%% file: 6_FinitePosets.tex
\section{Computing interval resolution global dimensions for certain classes of posets}
\label{sec:computing intresdim}

We end this paper with studying interval resolution global dimensions for certain classes of finite posets.
Throughout, we mainly consider finite connected posets having at least two elements.  
For our convenience, we use the following notation which makes us simple for describing a given family of posets.

\begin{notation} \label{notation:double}
We consider a finite simple graph $D$ such that 
the set of edges of $D$ is partitioned into $4$ classes of edges: 
\emph{ordinary lines}, \emph{ordinary arrows}, \emph{double lines} and \emph{double arrows}.  
Here, we visualize ordinary lines by  
$\begin{tikzpicture}[baseline=-2pt]
    \coordinate (x) at (1, 0); 
    \coordinate (y) at (0, 1); 
    \node (v) at ($0*(x) + 0*(y)$) {$\bullet$}; 
    \node (u) at ($1*(x) + 0*(y)$){$\bullet$}; 
    \draw[-] (v)--(u);
\end{tikzpicture}$, 
ordinary arrows by 
$\begin{tikzpicture}[baseline=-2pt]
    \coordinate (x) at (1, 0); 
    \coordinate (y) at (0, 1); 
    \node (v) at ($0*(x) + 0*(y)$) {$\bullet$}; 
    \node (u) at ($1*(x) + 0*(y)$){$\bullet$}; 
    \draw[->] (v)--(u);
\end{tikzpicture}$, 
double lines by 
$\begin{tikzpicture}[baseline=-2pt]
    \coordinate (x) at (1, 0); 
    \coordinate (y) at (0, 1); 
    \node (v) at ($0*(x) + 0*(y)$) {$\bullet$}; 
    \node (u) at ($1*(x) + 0*(y)$){$\bullet$}; 
    \draw[noArrow] (v)--(u);
\end{tikzpicture}$,
and double arrows by 
$\begin{tikzpicture}[baseline=-2pt]
    \coordinate (x) at (1, 0); 
    \coordinate (y) at (0, 1); 
    \node (v) at ($0*(x) + 0*(y)$) {$\bullet$}; 
    \node (u) at ($1*(x) + 0*(y)$){$\bullet$}; 
    \draw[vecArrow] (v)--(u);
\end{tikzpicture}$.

For such a diagram $D$, we say that a poset $P$ has a diagram $D$ if its Hasse diagram $\Hasse(P)$ can be obtained from $D$ by replacing 
\begin{itemize}
    \item each ordinary line 
    $\begin{tikzpicture}[baseline=-2pt]
    \coordinate (x) at (1, 0); 
    \coordinate (y) at (0, 1); 
    \node (v) at ($0*(x) + 0*(y)$) {$\bullet$}; 
    \node (u) at ($1*(x) + 0*(y)$){$\bullet$}; 
    \draw[-] (v)--(u);
    \end{tikzpicture}$
    with an arrow either $\bullet \leftarrow\bullet $ or $\bullet\rightarrow\bullet$; 
    \item each ordinary arrow 
    $\begin{tikzpicture}[baseline=-2pt]
    \coordinate (x) at (1, 0); 
    \coordinate (y) at (0, 1); 
    \node (v) at ($0*(x) + 0*(y)$) {$\bullet$}; 
    \node (u) at ($1*(x) + 0*(y)$){$\bullet$}; 
    \draw[->] (v)--(u);
\end{tikzpicture}$
    with an arrow $\bullet \rightarrow\bullet $;    
    \item each double line 
    $\begin{tikzpicture}[baseline=-2pt]
    \coordinate (x) at (1, 0); 
    \coordinate (y) at (0, 1); 
    \node (v) at ($0*(x) + 0*(y)$) {$\bullet$}; 
    \node (u) at ($1*(x) + 0*(y)$){$\bullet$}; 
    \draw[noArrow] (v)--(u);
\end{tikzpicture}$
    with an arbitrary $A_m$-type segment 
    $\begin{tikzpicture}[baseline=-2pt]
    \coordinate (x) at (1, 0); 
    \coordinate (y) at (0, 1); 
    \node (v) at ($0*(x) + 0*(y)$) {$\bullet$}; 
    \node (c) at ($1*(x) + 0*(y)$) {$\cdots$}; 
    \node (u) at ($2*(x) + 0*(y)$) {$\bullet$}; 
    \draw[-] (v)--(c)--(u);
    \end{tikzpicture}$ 
    with $m\geq2$; 
    \item each double arrow 
    $\begin{tikzpicture}[baseline=-2pt]
    \coordinate (x) at (1, 0); 
    \coordinate (y) at (0, 1); 
    \node (v) at ($0*(x) + 0*(y)$) {$\bullet$}; 
    \node (u) at ($1*(x) + 0*(y)$){$\bullet$}; 
    \draw[vecArrow] (v)--(u);
\end{tikzpicture}$
    with an equioriented $A_m$-type segment 
    $\begin{tikzpicture}[baseline=-2pt]
    \coordinate (x) at (1, 0); 
    \coordinate (y) at (0, 1); 
    \node (v) at ($0*(x) + 0*(y)$) {$\bullet$}; 
    \node (c) at ($1*(x) + 0*(y)$) {$\cdots$}; 
    \node (u) at ($2*(x) + 0*(y)$){$\bullet$}; 
    \draw[->] (v)--(c);
    \draw[->] (c)--(u);
    \end{tikzpicture}$ 
    with $m\geq 2$. 
\end{itemize}
We notice that $D$ is exactly the Hasse diagram if it consists of only ordinary arrows. 

For example, let $D$ be a diagram visualized by
\begin{equation}
\begin{tikzpicture}
        \coordinate (x) at (1.5, 0); 
        \coordinate (y) at (0, 1); 
        \node (v) at (0,0) {$\bullet_v$}; 
        \node (u) at ($(v) + 1.5*(x)$) {$\bullet_u$}; 
        \node (1) at ($(v) + 0*(x) + 1.5*(y)$) {$\bullet$}; 
        \node (3) at ($(u) + 1.5*(x) + 0*(y)$) {$\bullet$}; 
        \node (4) at ($(u) + 0*(x) + 1.5*(y)$) {$\bullet$}; 
        \node (w) at ($(v) + 0.75*(x)$) {$\bullet_w$}; 
        \node (5) at ($(u) + 1.5*(x) + 1.5*(y)$) {$\bullet$};
        
        \draw[vecArrow] (1)--(v);
        \draw[vecArrow] (4)--(1);
        \draw[vecArrow] (u)--(w);
        \draw[vecArrow] (w)--(v);
        \draw[noArrow] (u)--(3);
        \draw[->] (u)--(4);
        \draw[->] (4)--(5);
\end{tikzpicture}.
\end{equation}
Then, a given finite poset has a diagram $D$ if and only if its Hasse diagram is of the following form: 
\begin{equation}
\begin{tikzpicture}
        \coordinate (x) at (2, 0); 
        \coordinate (y) at (0, 2); 
        \node (v) at (0,0) {$\bullet_v$}; 
        \node (u) at ($2*(x)$) {$\bullet_u$}; 
        \node (1) at ($0*(x) + 1*(y)$) {$\bullet$}; 
        \node (1c) at ($(v) + 0*(x) + 0.5*(y)$) {\rotatebox{90}{$\cdots$}}; 
        \node (2c) at ($1*(x) + 1*(y)$) {\rotatebox{0}{$\cdots$}}; 
        \node (41) at ($0.5*(x) + 0*(y)$) {$\cdots$}; 
        \node (4c) at ($1*(x) + 0*(y)$) {\rotatebox{0}{$\bullet_w$}}; 
        \node (42) at ($1.5*(x) + 0*(y)$) {$\cdots$};
        \node (5c) at ($2.5*(x) + 0*(y)$) {\rotatebox{0}{$\cdots$}}; 
        \node (3) at ($3*(x) + 0*(y)$) {$\bullet$}; 
        \node (4) at ($2*(x) + 1*(y)$) {$\bullet$}; 
        \node (5) at ($3*(x) + 1*(y)$) {$\bullet$}; 
        
        \draw[->] (1)--(1c);
        \draw[->] (1c)--(v);
        \draw[->] (4)--(2c);
        \draw[->] (2c)--(1);
        \draw[->] (u)--(4);
        \draw[->] (u)--(42);
        \draw[->] (41)--(v);
        \draw[-] (u)--(5c);
        \draw[-] (5c)--(3);
        \draw[->] (4)--(5);
        \draw[->] (42)--(4c);
        \draw[->] (4c)--(41);
\end{tikzpicture}  
\end{equation}
Especially, the smallest cases are given by 
\begin{equation}
\begin{tikzpicture}[baseline = 5mm]
        \coordinate (x) at (1.2, 0); 
        \coordinate (y) at (0, 1); 
        \node (v) at (0,0) {$\bullet_v$}; 
        \node (u) at ($(v) + 1.5*(x)$) {$\bullet_u$}; 
        \node (1) at ($(v) + 0*(x) + 1.5*(y)$) {$\bullet$}; 
        \node (3) at ($(u) + 1.5*(x) + 0*(y)$) {$\bullet$}; 
        \node (3c) at ($(u) + -0.75*(x) + 0*(y)$) {$\bullet_w$}; 
        \node (4) at ($(u) + 0*(x) + 1.5*(y)$) {$\bullet$};
        \node (5) at ($(u) + 1.5*(x) + 1.5*(y)$) {$\bullet$};
        
        \draw[->] (1)--(v);
        \draw[->] (4)--(1);
        \draw[->] (u)--(3c);
        \draw[->] (3c)--(v);
        \draw[->] (u)--(4);
        \draw[->] (3)--(u);
        \draw[->] (4)--(5);
\end{tikzpicture}  
\qand
\begin{tikzpicture}[baseline = 5mm]
        \coordinate (x) at (1.2, 0); 
        \coordinate (y) at (0, 1); 
        \node (v) at (0,0) {$\bullet_v$}; 
        \node (u) at ($(v) + 1.5*(x)$) {$\bullet_u$}; 
        \node (1) at ($(v) + 0*(x) + 1.5*(y)$) {$\bullet$}; 
        \node (3) at ($(u) + 1.5*(x) + 0*(y)$) {$\bullet$}; 
        \node (3c) at ($(u) + -0.75*(x) + 0*(y)$) {$\bullet_w$}; 
        \node (4) at ($(u) + 0*(x) + 1.5*(y)$) {$\bullet$}; 
        \node (5) at ($(u) + 1.5*(x) + 1.5*(y)$) {$\bullet$};
        
        \draw[->] (1)--(v);
        \draw[->] (4)--(1);
        \draw[->] (u)--(3c);
        \draw[->] (3c)--(v);
        \draw[->] (u)--(4);
        \draw[<-] (3)--(u);
        \draw[->] (4)--(5);
\end{tikzpicture}.
\end{equation}
\end{notation}

With this notation, we can rewrite Theorem \ref{thm:cont chains}(a) as follows. 

\begin{corollary} \label{cor:cont chains}
    Let $D$ be a finite simple graph in Notation \ref{notation:double} which is consisting of ordinary arrows and double arrows. 
    Let $P$ be a finite poset obtained from $D$ by replacing each double arrows of $D$ with equioriented $A_3$-type segment. 
    Then, any finite poset having diagram $D$ and containing $P$ as its full subposet has the same interval resolution global dimension as that of $P$. 
\end{corollary}

\begin{proof}
    It is straightforward from Theorem \ref{thm:cont chains}(a). 
\end{proof}

This will be used in Example \ref{ex:int_res_dim_one} to provide a family of posets having interval resolution global dimension $1$.

We recall a classification of finite posets having interval resolution global dimension $0$ due to \cite[Theorem 5.1]{AET}.

\begin{proposition}  \label{prop:dim0}
    Let $P$ be a finite connected poset having at least two elements. 
    Then, the following conditions are equivalent. 
    \begin{enumerate}[\rm (a)]
        \item Every $P$-persistence module is interval-decomposable. 
        \item $\intresgldim P =0$.  
        \item $P$ is either an $A$-type poset or a bipath poset, that is, it has the diagram of the following form respectively. 
        \begin{equation}
        \begin{tikzpicture}[baseline = 0mm]
        \coordinate (x) at (1.5, 0); 
        \coordinate (y) at (0, 1); 
        \node (v) at (0,0) {$\bullet$}; 
        \node (u) at ($(v) + 1*(x)$) {$\bullet$}; 
        \draw[vecArrow] (v)--(u);
        \end{tikzpicture}
        \qand 
        \begin{tikzpicture}[baseline = 0mm]
        \coordinate (x) at (1.3, 0); 
        \coordinate (y) at (0, 1.3); 
        \node (v) at ($0*(x) + 0.5*(y)$) {$\bullet$}; 
        \node (vv) at ($0*(x) + -0.5*(y)$) {$\bullet$}; 
        \node (u) at ($1*(x) + 0.5*(y)$) {$\bullet$}; 
        \node (uu) at ($1*(x) + -0.5*(y)$) {$\bullet$.}; 
        \draw[vecArrow] (v)--(vv);
        \draw[vecArrow] (v)--(u);
        \draw[vecArrow] (vv)--(uu);
        \draw[vecArrow] (u)--(uu);
        \end{tikzpicture}
        \end{equation}
        
    \end{enumerate}
\end{proposition}

\subsection{Tree-type posets}

We say that a given finite poset $P$ is a \emph{tree-type poset} if the underlying graph of its Hasse diagram is a tree (i.e., a connected simple graph without cycles). 
We recall that a leaf of $P$ is an element $x\in P$ such that $\deg(x)=1$. 
We denote by $\leaf(P)$ the set of leaves of $P$. 

Our result is the following. 

\begin{proposition} \label{prop:dimTree}
    Let $P$ be a tree-type poset having at least two elements. Then we have 
    \[
    \intresgldim P = \#\leaf(P) - 2. 
    \]
    In particular, it is determined by the underlying graph of its Hasse diagram.  
\end{proposition}

\begin{proof}
    Let $S$ be an interval of $P$. 
    We regard $S$ as a full subposet of $P$. 
    From Proposition \ref{prop:formulas}, we will compute the interval resolution dimension of $\Gamma_S$. 
    If $S=\{v\}$ for an element $v\in \max(P)\cap \leaf(P)$, there are no irreducible morphisms ending at $\mathbb{I}_S$. 
    Thus, we have $\Gamma_S=0$. 
    Otherwise, we will describe an interval resolution of $\Gamma_S$ as follows.
    Let $\mathcal{I}:=\min(S)\cap \leaf(S)$ and $\mathcal{J}:=\{z\in P\mid S \lessdot z\}$. 
    In addition, let $\mathcal{K}:=\mathcal{I}\sqcup \mathcal{J}$ with $m:=\#\mathcal{K} \leq \#\leaf(P)$. 
    For any $\sigma \subseteq \mathcal{K}$, the set 
    \[
    S^{\sigma} := S \setminus (\sigma\cap \mathcal{I}) \cup \bigcup_{z\in \sigma\cap \mathcal{J}} \rangle S,z\rangle 
    \]
    is an interval of $P$ since $P$ is of the tree. 
    Thus, one obtains an interval module $\mathbb{I}_{S^{\sigma}}$. 
    By Proposition~\ref{prop:irred_interval}, 
    every $s\in \sigma$ gives an irreducible morphism $\alpha_{\sigma,s}\colon \mathbb{I}_{S^{\sigma}} \to \mathbb{I}_{S^{\sigma\setminus \{s\}}}$, 
    where it is injective (resp., surjective) if $s\in \mathcal{I}$ (resp., $s\in \mathcal{J}$). 
    Moreover, it follows from Proposition~\ref{prop:irred_interval} that every irreducible morphism starting at $\mathbb{I}_{S^\sigma}$ can be written in this way. 
    Dually, we get a similar description for irreducible morphisms ending at $\mathbb{I}_{S^\sigma}$. 
    
    Now, we choose a total order $\prec$ on $\mathcal{K}$. 
    For each $0\leq i \leq m$, let 
    \[
    V_i := \bigoplus_{\substack{\sigma\subseteq \mathcal{K} \\ \#\sigma =i}} \mathbb{I}_{S^\sigma}. 
    \]
    For example $V_0=\mathbb{I}_{S^{\emptyset}}=\mathbb{I}_{S}$ and $V_1 =\bigoplus_{\alpha\in \mathcal{A}_S} \mathbb{I}_{S_{\alpha}}$ (see Section \ref{subsec:Thm_chain(a)}). 
    In addition, we define a morphism $\varphi_i\colon V_i\to V_{i-1}$ 
    as the alternating sum $\varphi_{i} = \sum_{j=0}^{i-1} (-1)^{j} \varphi_{ij}$, where for each $0\leq j \leq i-1$, a morphism $\varphi_{ij}\colon V_i\to V_{i-1}$ is defined by mapping the summand $\mathbb{I}_{S^{\sigma}}$ with $\sigma=\{s_0\prec \cdots \prec s_{i-1}\}$ to the summand $\mathbb{I}_{S^{\sigma\setminus\{s_j\}}}$ via the map $\alpha_{\sigma,s_j}$. 
    Under these notation, it is not difficult to see that
    $\varphi_1$ is surjective, 
    $\varphi_m$ is injective, and $\im \varphi_i \cong \ker \varphi_{i-1}$ for all $i\in \{2,\ldots,m\}$. 
    Thus, they yield an exact sequence 
    \begin{equation}\label{eq:boolian1}
        0\to V_m \xrightarrow{\varphi_m} V_{m-1}\xrightarrow{\varphi_{m-1}}
        \cdots \cdots \xrightarrow{\varphi_2} V_1 \xrightarrow{\varphi_1} V_0 \to 0.
    \end{equation}
    Since $\Gamma_S = \ker \varphi_1$ by definition, 
    \begin{equation}\label{eq:boolian2}
        0\to V_m \xrightarrow{\varphi_m} V_{m-1}\xrightarrow{\varphi_{m-1}}
        \cdots \cdots \xrightarrow{\bar{\varphi}_2} \Gamma_S\to 0
    \end{equation}
    is an interval resolution of $\Gamma_S$ and hence $\intresdim \Gamma_S = m-2$.
    
    From the above observation, we obtain the inequality 
    \begin{equation}
        \intresgldim P \overset{\rm Prop.\,\ref{prop:formulas}}{=} \sup \{\intresdim \Gamma_S \mid S\in \Int(P)\} \leq \#\leaf(P) - 2. 
    \end{equation}
    Furthermore, the equality holds true for an interval $S = P \setminus (\max(P) \cap \leaf(P))$ as 
    $\mathcal{I} = \min(P) \cap \leaf(P) = \leaf(P) \setminus \max(P)$ and $\mathcal{J} = \max(P)\cap \leaf(P)$.
    Therefore, we get the assertion. 
\end{proof}

\subsection{$\tilde{A}$-type posets} 
\label{sec:Atilde}

Consider an $\tilde{A}$-type poset $P$. 
That is, it is a poset having a diagram of the form 
\begin{equation*} 
    \begin{tikzpicture}[baseline = 0mm]
        \node (a1) at (90:1.2) {$t_1$};
        \node (a2) at (90-60:1.2) {$t_2$};
        \node (a3) at (90-120:1.2) {$t_3$};
        \node (a4) at (90-180:1.2) {};
        \node (as) at (90+60:1.2) {$t_{2s}$};
        \node (ap) at (90+120:1.2) {};
        \draw[vecArrow] (a1)--(as);
        \draw[vecArrow] (a1)--(a2);
        \draw[vecArrow] (a3)--(a2);
        \draw[vecArrow] (a3)--(a4);
        \draw[vecArrow] (ap)--(as);
        \draw[dotted] (a4)--(ap);
    \end{tikzpicture}
\end{equation*}
where $s\geq 2$ is the number of sinks of $P$, and $t_i$'s are sinks (resp., sources) for even (resp., odd) $i$. 

\begin{proposition} \label{prop:dimAtilde}
    Suppose that $P$ is an $\tilde{A}$-type poset with $s\geq 2$ sinks as above. 
    Then, we have 
    \begin{equation}
        \intresgldim P = 
        \begin{cases}
            1 & \text{if $s = 2$}, \\  
            2 & \text{if $s \geq 3$}. 
        \end{cases}
    \end{equation}
    In particular, it is determined by the number of sinks.
\end{proposition}

We need some preparations. 
Let $P$ be as above. In addition, let $n$ be the number of elements of $P$. 
For our convenience, the indexes of sinks and sources will be considered modulo $2s$, for instance, $t_i$ and $t_{i+2s}$ show the same element in $P$. 
We define a permutation $\rho$ on $P$ as follows: 
If an element $a$ lies in the segment between $t_i$ and $t_{i+1}$, then $\rho(a)$ is such that $a\lessdot \rho(a)$ (resp., $\rho(a) \lessdot a$) if $t_i\leq t_{i+1}$ (resp., $t_{i+1} \leq t_i$). 
This means that $\rho$ is a rotation of elements of $P$ in clockwise direction under the above embedding into the plane. 

Next, we recall some properties of the incidence algebra $k[P]$, see Section \ref{sec:stabilizing}. 
Since it is a hereditary algebra of type $\tilde{A}$, 
its module category is divided into three classes of components called the preprojective component, preinjective component and regular components.  
Here, the preprojective (resp., preinjective) component is the $\tau$-orbits of indecomposable projective (resp., indecomposable injective) modules and all other components are regular. 
We denote by $\mathcal{P}$, $\mathcal{Q}$, $\mathcal{R}$ 
the additive subcategories generated by all modules in the preprojective, preinjective, regular components respectively. 
They are located at the left, right, middle in the AR-quiver respectively as (see \cite[Chapter VIII, Theorem 4.5]{ASS})
\begin{equation}
    \Hom_P(\mathcal{R}, \mathcal{P}) = \Hom_P(\mathcal{Q}, \mathcal{P}) = \Hom_P(\mathcal{Q}, \mathcal{R}) = 0.
\end{equation}


\begin{lemma}\label{lem:resol preproj}
    Every non-interval-decomposable module in $\mathcal{P}$ has interval resolution dimension $1$. 
\end{lemma}

\begin{proof}
      By looking at the AR-quiver, it is easily shown that $\mathcal{P}\cap \mathscr{I}_P$ is closed under taking submodules. 
      Then, the assertion follows from the fact that 
      the interval cover of a module in $\mathcal{P}$ is provided by interval-decomposable modules in $\mathcal{P}\cap \mathscr{I}_P$. 
\end{proof}

Next, we recall that $k[P]$ is a string algebra (we refer to \cite{Erdmann} for fundamental results on string algebras). 
In particular, one can describe the AR-quiver by using the so-called string-band combinatorics. We will use it with the following notations. 
For two element $a,b\in P$, we define sequences 
\begin{equation*}
    [a,b]^0 := (a,\rho(a),\ldots, \rho^{\ell}(a) = b) 
    \qand
    [a,b]^1 := (a,\rho(a),\ldots, \rho^{n+\ell}(a) = b) 
\end{equation*}
of length $\ell+1$ and $\ell+n+1$ respectively, where $\ell$ is the smallest number $j\geq 0$ satisfying $b = \rho^j(a)$. 
A sequence $[a,b]^i$ for $i\geq2$ is defined in a similar way. 
Then, we naturally regard it as a string of $P$ (more precisely $\Hasse(P)$), and we write $M([a,b]^i)$ for the corresponding string module. 
Notice that such a string module belongs to $\mathcal{P}$, $\mathcal{Q}$, $\mathcal{R}$ depending on the positions of $a$ and $b$, but we omit the detail. 
For example, $M([a,b]^0)$ is an interval module if and only if $a=b$ or $a,b$ are incomparable in $P$, 
and every interval module except for $\mathbb{I}_P$ can be written in this way. 
On the other hand, the interval module $\mathbb{I}_P$ is provided as a band module and lies in $\mathcal{R}$. 

Under the above preparations, we show Proposition \ref{prop:dimAtilde} by computing interval resolutions directly.

\begin{proof}[Proof of Proposition \ref{prop:dimAtilde}]
The claim for $s=2$ can be shown by using Proposition \ref{thm:cont chains}.

We consider the case $s\geq 3$. 
For each $i\in \{1,\ldots,2s\}$, let $S_i := P\setminus {t_i}^{\uparrow}$ be an interval of $P$. 
From now on, we show that $\intresdim \Gamma_{S_i}=2$ if $i$ is odd. 
Suppose that $i$ is odd. 
Recalling $t_i$ being a source of $P$, 
one can write $\mathbb{I}_{S_i} \cong M([\rho(t_{i+1}),\rho^{-1}(t_{i-1})]^0)$ as a string module. 
From a description of irreducible morphisms in Proposition \ref{prop:irred_interval}, 
there are two surjective irreducible morphisms that yield an exact sequence 
\begin{equation}
    0 \to \Gamma_{S_i} \to 
    M([t_i,\rho^{-1}(t_{i-1})]^0) \oplus M([\rho(t_{i+1}),t_i]^0) \xrightarrow{(\alpha_1\, \alpha_2)} 
    \mathbb{I}_{S_i} \to 0. 
\end{equation}
Besides, this sequence is an almost split sequence in the module category.
In particular, 
$\Gamma_{S_i} \cong \tau\mathbb{I}_{S_i} \cong M([t_i,t_i]^1)$
holds, where we use a description of $\tau$ for string modules (see \cite{BR}) for the latter isomorphism. 
For this module, we can compute the interval cover and its kernel by a direct calculation of homomorphisms as 
\begin{equation}
    0 \to M([t_{i+3}, t_{i-3}]^{1}) \to 
    M([t_{i+3},t_{i}]^0) \oplus M([t_{i},t_{i-3}]^0) \oplus \mathbb{I}_P \xrightarrow{(\beta_1\, \beta_2\, \beta_3)} \Gamma_{S_i} \to 0 
\end{equation}
where $\beta_1,\beta_2$ are natural injections and $\beta_3$ is a natural surjection. 
Then, the kernel $M([t_{i+3}, t_{i-3}]^{1})$ is a non-interval module lying in $\mathcal{P}$ and has interval resolution dimension $1$ by Lemma \ref{lem:resol preproj}. 
Hence, we conclude that $\intresdim \Gamma_{S_i} = 2$. 

On the other hand, for other intervals $T$ of $P$, 
by using a description of irreducible morphisms in Proposition \ref{prop:irred_interval},
one can check that $\intresdim \Gamma_T \leq 1$ holds. 
This finishes the proof. 
\end{proof}

\subsection{Finite posets having low interval resolution global dimensions}
\label{subsec:low intresgildim}
In this section, we study finite posets having low interval resolution global dimension $d$. 
For our convenience, we assume that $k$ is a finite field with two elements throughout this section. 
The case $d=0$ has been done by \cite[Theorem 5.1]{AET}. 
However, the cases $d \geq 1$ may be much complicated. 
We give some examples for the case $d=1$. 

\begin{example}\label{ex:int_res_dim_one}
Any poset having one of diagrams (1)--(21) (see Notation \ref{notation:double}) in Table \ref{tab:gldim1} has interval resolution global dimension $1$, where we put labels on the vertices of diagrams (15)--(21) to refer to them later. 
In fact, the claim follows from Corollary \ref{cor:cont chains}.
More precisely, for each diagram $D$ in the table, 
we can compute that the poset obtained from $D$ by 
replacing each double arrows of $D$ with equioriented $A_3$-type segment has interval resolution global dimension $1$.

\input{tables_ee/table_resdim1}
\end{example}

It has been shown in \cite[Theorem 4.1]{AET} that the interval resolution global dimension of full subposets is bounded by that of the base poset. 
From this result, we can discuss the minimality of posets having a fixed interval resolution global dimension, as follows. 

\begin{definition}
    We say that a given finite poset $P$ is 
    \emph{minimal} with respect to the interval resolution global dimension if every proper full subposet of $P$ has the interval resolution global dimension less than that of $P$. 
\end{definition}

For $d>0$, let $\mathcal{M}_d$ be the set of isomorphism classes of finite posets which are minimal with respect to the interval resolution global dimension and have interval resolution global dimension $d$, 
and we additionally identify a given poset with its opposite in $\mathcal{M}_d$. 
Indeed, by Proposition \ref{prop:formulas}, 
a given poset is minimal with respect to the interval resolution global dimension if and only if so is its opposite poset. 

For example, it is obvious that the set $\mathcal{M}_0$ consists of the singleton $\{\bullet\}$, 
whereas by Propositions \ref{prop:dimTree} and \ref{prop:dimAtilde}, the set $\mathcal{M}_1$ is provided by posets having the following diagrams. 

\begin{equation*}    
    \begin{tabular}{cccccc} 
      \begin{tikzpicture}[baseline = 0mm]
          \node (0) at (0,0) {$\bullet$}; 
          \node (1) at (1,0) {$\bullet$}; 
          \node (2) at (-1,0) {$\bullet$}; 
          \node (3) at (0,1) {$\bullet$}; 
          \draw (0)--(1) (0)--(2) (0)--(3);
      \end{tikzpicture}  
      &  
      \begin{tikzpicture}[baseline = 0mm]
          \node (0) at (0,0) {$\bullet$}; 
          \node (1) at (1,0) {$\bullet$}; 
          \node (2) at (0,1) {$\bullet$}; 
          \node (3) at (1,1) {$\bullet$};  
          \draw[->] (0)--(1); 
          \draw[->] (0)--(2);
          \draw[->] (3)--(1);
          \draw[->] (3)--(2);
      \end{tikzpicture}   
    \end{tabular}
\end{equation*}

Thus, one naturally asks the next problem. 

\begin{problem}
    Classify all posets in $\mathcal{M}_2$. 
\end{problem}

Towards this problem, we will provide a partial classification for $\mathcal{M}_2$ in this section. 
More precisely, we prove in Proposition \ref{prop:forM2} that Table \ref{tab:forM2} gives a complete representatives of posets $P\in \mathcal{M}$ satisfying $\deg(P)_3\leq 2$. 
Here, $\deg(P)_i$ denotes the number of elements of $P$ having degree $i$. 

\begin{lemma} \label{lem:deg_geq4}
    If $P\in \mathcal{M}_2$ and $\deg(P)_i > 0$ for some $i\geq 4$, then it is given by {\bf (i)} in Table \ref{tab:forM2}.
\end{lemma}

\begin{proof}
Let $P\in\mathcal{M}_2$. 
If there is an element $x\in P$ such that $\deg(x)\geq 4$, then $P$ contains a poset $S$ of the form {\bf (i)} with some orientation as its full subposet. 
Since $S$ is a tree-type poset having $4$ leaves, 
it has interval resolution global dimension $2$ by Proposition~\ref{prop:dimTree}). 
Moreover, we have $S\in \mathcal{M}_2$. 
Then, $P=S$ must hold by the minimality of $P$. 
\end{proof}

Thus, we focus on the number $\deg(P)_3$ of vertices having degree $3$. 

\begin{proposition}\label{prop:forM2}
We have the following statements. 
\begin{enumerate}[\rm (1)]
    \item Every poset given in Table \ref{tab:forM2} belongs to $\mathcal{M}_2$. 
    \item If $P\in \mathcal{M}_2$ and $\deg(P)_3 \leq 2$, then it is one of posets in Table \ref{tab:forM2}.    
\end{enumerate}
Consequently, Table \ref{tab:forM2} gives a complete list of posets $P$ of $\mathcal{M}_2$ with $\deg(P)_3\leq 2$. 
\end{proposition}

\input{tables_ee/table_M2}

\begin{proof} 
(1) Suppose that $P$ is a poset in Table \ref{tab:forM2}. 
If it is given by one of forms {\bf (i)}, {\bf (iii)}, {\bf (v)}--{\bf (ix)}, then we have $\intresgldim P = 2$ by a direct calculation. Moreover, every full subposet of $P$ has the interval resolution global dimension $0$ or $1$. 
Thus, $P\in \mathcal{M}_2$. 
On the other hand, we get the assertions for {\bf (ii)}${}_n$ and {\bf (iv)}${}_\ell$ by Propositions \ref{prop:dimAtilde} and \ref{prop:dimTree} respectively.

(2) Let $P\in \mathcal{M}_2$. 
We first recall that 
\begin{itemize}
    \item $P$ is not a full subposet of any poset appearing in Table \ref{ex:int_res_dim_one}.
    \item $P$ has no proper full subposet $S$ appearing in Table \ref{tab:forM2}. 
    \item If there is a full subposet $S$ appearing in Table \ref{tab:forM2}, then we have $P=S\in \mathcal{M}_2$ by the minimality of $P$. 
\end{itemize}
Therefore, in the following proof by case analysis, the discussion will end at either point where $P$ is a full subposet of some poset obtained in Example \ref{ex:int_res_dim_one} or $P$ has a full subposet appearing in Table \ref{tab:forM2}. 

If $\deg(P)_i >0$ for some $i\geq 4$, then this is of the form {\bf (i)} by Lemma \ref{lem:deg_geq4}. 
Notice that $\deg(P)_3=0$ holds in this case. 
In the rest, we assume that $\deg(P)_i = 0$ for all $i\geq 4$ and study each case of $\deg(P)_3 \in \{0,1,2\}$. 

Firstly, we assume that $\deg(P)_3 = 0$. 
In this case, every element $x\in P$ satisfies $\deg(x)\in\{1,2\}$. 
Thus, the underlying graph of $\Hasse(P)$ is a line or a cycle.
This means that $P$ is one of $A$-type posets, bipath posets or $\tilde{A}$-type posets. 
Then, it must be an $\tilde{A}$-type poset since $A$-type posets and bipath poset have interval resolution global dimensions $0$ by Proposition \ref{prop:dim0}. 
In this situation, $P$ has $s\geq 2$ sinks by Proposition \ref{prop:dimAtilde} and hence it always contain the poset $S$ of the form {\bf (ii)}$_{2s}$ as its full subposet. 

Secondly, we assume that $\deg(P)_3 = 1$. 
In this case, we claim that $P$ is of the form {\bf (iii)} in Table \ref{tab:forM2}. 
In fact, the Hasse diagram of $P$ is either (a) or (b):  
\begin{equation}
    \begin{tabular}{cccc}
       $(a)$ &  $(b)$ \\  
    \begin{tikzpicture}
        \coordinate (x) at (1.25, 0); 
        \coordinate (y) at (0, 1.25); 
        \node (u) at ($0*(x) + 1*(y)$) {$\bullet_u$}; 
        \node (v) at ($0*(x) + 0*(y)$) {$\bullet_v$}; 
        \node (3) at ($1*(x) + 0*(y)$) {$\bullet_3$}; 
        \node (4) at ($-1*(x) + 0*(y)$) {$\bullet_4$}; 
        \draw[noArrow] (u)--(v);
        \draw[noArrow] (v)--(3);
        \draw[noArrow] (v)--(4);
    \end{tikzpicture} 
    & 
    \begin{tikzpicture}
        \coordinate (x) at (1.25, 0); 
        \coordinate (y) at (0, 1.25); 
        \node (u) at ($0*(x) + 1*(y)$) {$\bullet_u$}; 
        \node (v) at ($0*(x) + 0*(y)$) {$\bullet_v$}; 
        \node (2) at ($-1*(x) + 1*(y)$) {$\bullet_2$}; 
        \node (3) at ($1*(x) + 0*(y)$) {$\bullet_3$}; 
        \node (4) at ($-1*(x) + 0*(y)$) {$\bullet_4$}; 
         \node (C) at ($-0.5*(x) +0.5*(y)$) {$C$};
        \draw[noArrow] (2)--(4);
        \draw[noArrow] (u)--(2);
        \draw[noArrow] (u)--(v);
        \draw[noArrow] (v)--(3);
        \draw[noArrow] (v)--(4);
    \end{tikzpicture}  
    \end{tabular}
\end{equation}
Here, the vertex $v$ is a unique element such that $\deg(v) = 3$ and $C$ is a cycle containing $v$. 
Then, this is not the case $(a)$ since they have interval resolution global dimensions $\#\leaf(P) - 2 = 1$ by Proposition \ref{prop:dimTree}. 
We consider the case $(b)$. 
Let $s$ be the number of sinks of $C$. 
If $s = 1$, then $P$ is a full subposet of a poset in Example \ref{ex:int_res_dim_one}(1) or (2) depending on the position of the vertex $v$ in the Hasse diagram. 
Next, we assume that $s = 2$. 
If $v$ is a sink, then $P$ is a full subposet of a poset in Example \ref{ex:int_res_dim_one}(3).
Otherwise, $P$ contains the poset of the form {\bf (iii)} as its full subposet. 
Lastly, if $s\geq 3$, then $P$ contains the poset of the form {\bf (ii)}$_{2s}$ as its proper full subposet.

Thirdly, we consider the case $\deg(P)_3 = 2$. 
In this case, the Hasse diagram of $P$ is one of the following. 
\begin{center}
    \begin{tabular}{ccccccc}
        (a) &  (b) & (c) \\
    \begin{tikzpicture}
        \coordinate (x) at (1.25, 0); 
        \coordinate (y) at (0, 1.25); 
        \node (u) at ($0*(x) + 1*(y)$) {$\bullet_u$}; 
        \node (v) at ($0*(x) + 0*(y)$) {$\bullet_v$}; 
        \node (1) at ($1*(x) + 1*(y)$) {$\bullet$}; 
        \node (2) at ($-1*(x) + 1*(y)$) {$\bullet$}; 
        \node (3) at ($1*(x) + 0*(y)$) {$\bullet$}; 
        \node (4) at ($-1*(x) + 0*(y)$) {$\bullet$}; 
        \draw[noArrow] (u)--(1);
        \draw[noArrow] (u)--(2);
        \draw[noArrow] (u)--(v);
        \draw[noArrow] (v)--(3);
        \draw[noArrow] (v)--(4);
    \end{tikzpicture}  
    & 
    \begin{tikzpicture}
        \coordinate (x) at (1.25, 0); 
        \coordinate (y) at (0, 1.25); 
        \node (3) at ($0*(x) + 1*(y)$) {$\bullet$}; 
        \node (v) at ($0*(x) + 0*(y)$) {$\bullet_v$}; 
        \node (1) at ($1*(x) + 1*(y)$) {$\bullet$}; 
        \node (2) at ($-1*(x) + 1*(y)$) {$\bullet$}; 
        \node (u) at ($1*(x) + 0*(y)$) {$\bullet_u$}; 
        \node (4) at ($-1*(x) + 0*(y)$) {$\bullet$}; 
        \node (6) at ($2*(x) + 0*(y)$) {$\bullet$}; 
        
        \draw[noArrow] (v)--(u);
        \draw[noArrow] (v)--(3);
        \draw[noArrow] (v)--(4);
        \draw[noArrow] (u)--(1);
        \draw[noArrow] (u)--(6);
        \draw[noArrow] (4)--(2);
        \draw[noArrow] (2)--(3);
        
    \end{tikzpicture} 
    & 
    \begin{tikzpicture}
        \coordinate (x) at (1.25, 0); 
        \coordinate (y) at (0, 1.25); 
        \node (3) at ($0*(x) + 1*(y)$) {$\bullet$}; 
        \node (v) at ($0*(x) + 0*(y)$) {$\bullet_v$}; 
        \node (1) at ($1*(x) + 1*(y)$) {$\bullet$}; 
        \node (2) at ($-1*(x) + 1*(y)$) {$\bullet$}; 
        \node (u) at ($1*(x) + 0*(y)$) {$\bullet_u$}; 
        \node (4) at ($-1*(x) + 0*(y)$) {$\bullet$}; 
        \node (5) at ($2*(x) + 1*(y)$) {$\bullet$}; 
        \node (6) at ($2*(x) + 0*(y)$) {$\bullet$}; 
        
        \draw[noArrow] (v)--(u);
        \draw[noArrow] (v)--(3);
        \draw[noArrow] (v)--(4);
        \draw[noArrow] (u)--(1);
        \draw[noArrow] (u)--(6);
        \draw[noArrow] (4)--(2);
        \draw[noArrow] (2)--(3);
        \draw[noArrow] (1)--(5);
        \draw[noArrow] (5)--(6);
    \end{tikzpicture} 
\end{tabular}
\end{center}

\begin{equation}\label{eq:CDdiagram}
\begin{tabular}{cc}
    (d) & (e) \\ 
    \begin{tikzpicture}
        \coordinate (x) at (1.25, 0); 
        \coordinate (y) at (0, 1.25); 
        \node (3) at ($0*(x) + 1*(y)$) {$\bullet$}; 
        \node (v) at ($0*(x) + 0*(y)$) {$\bullet_v$}; 
        \node (1) at ($1*(x) + 1*(y)$) {$\bullet$}; 
        \node (u) at ($1*(x) + 0*(y)$) {$\bullet_u$}; 
        \node (4) at ($-1*(x) + 0*(y)$) {$\bullet$}; 
        \node (6) at ($2*(x) + 0*(y)$) {$\bullet$}; 
        \node (C) at ($0.5*(x) + 0.5*(y)$) {$C$};
    
        \draw[noArrow] (v)--(u);
        \draw[noArrow] (v)--(3);
        \draw[noArrow] (v)--(4);
        \draw[noArrow] (u)--(1);
        \draw[noArrow] (u)--(6);
        \draw[noArrow] (1)--(3);
    \end{tikzpicture} 
    &
    \begin{tikzpicture}
        \coordinate (x) at (1.25, 0); 
        \coordinate (y) at (0, 1.25); 
        \node (u) at ($0*(x) + 1*(y)$) {$\bullet_u$}; 
        \node (v) at ($0*(x) + 0*(y)$) {$\bullet_v$}; 
        \node (1) at ($1*(x) + 1*(y)$) {$\bullet$}; 
        \node (2) at ($-1*(x) + 1*(y)$) {$\bullet$}; 
        \node (3) at ($1*(x) + 0*(y)$) {$\bullet$}; 
        \node (4) at ($-1*(x) + 0*(y)$) {$\bullet$}; 
        \node (D) at ($0.5*(x) + 0.5*(y)$) {$D$};
        \node (C) at ($-0.5*(x) + 0.5*(y)$) {$C$};
    
        \draw[noArrow] (u)--(1);
        \draw[noArrow] (u)--(2);
        \draw[noArrow] (u)--(v);
        \draw[noArrow] (v)--(3);
        \draw[noArrow] (v)--(4);
        \draw[noArrow] (2)--(4);
        \draw[noArrow] (1)--(3);
    \end{tikzpicture}  
\end{tabular}
\end{equation}
Here, vertices $u$ and $v$ are elements of $P$ such that $\deg(v) = \deg(u) = 3$, and $C,D$ are cycles containing both $u,v$. 
From now on, we study each case of (a)--(e). 
In the case of (a),(b) or (c), $P$ has a full subposet of the form {\bf (iv)}$_{\ell}$ for some $\ell \geq 2$. 

We consider the case (d). 
Let $T = \{t_1,t_2,\ldots,t_{2s}\}$ be the set of sinks and sources of $C$, where $s$ is the number of sinks of $C$. 
We reorder elements $t_1,t_2,\ldots,t_{2s}$ of $T$ in clockwise direction so that $t_i$ are sources (resp., sinks) for odd (resp., even) $i$. We regard $T$ as a full subposet of $P$. 
Let $\bar{u},\bar{v} \in P\setminus C$ be elements adjacent to $u$ and $v$ respectively. 
\begin{enumerate}
    \item[\rm (d-1)] 
We assume that $s = 1$. 
If $u,v\not\in \{t_1,t_2\}$, then $P$ is a poset of the form {\bf (vi)}. 
Otherwise, we may assume that $u=t_1$ without loss of generality (We can replace $P$ with its opposite poset if necessary). 
If moreover $v=t_2$ (resp., $v \lessdot t_2$), then $P$ is a full subposet of a poset in Example \ref{ex:int_res_dim_one}(2) (resp., (4)). 
Otherwise, we have $v \nlessdot t_2$ and an element $\beta_{v,t_2}\in P$ such that $v < \beta_{v,t_2} < t_2$. 
In addition, there is an element $x\in C$ incomparable to $v$ such that $t_1 < x < t_2$ by the definition of Hasse diagrams. 
Then, a set $\{v,u,\bar{u},\bar{v},\beta_{v,t_2},x\}$ forms a proper full subposet of $P$ of the form {\bf (iv)}$_{2}$. 
\item[\rm (d-2)] Secondly, we assume that $s=2$. 
If $u\not\in T$, 
then a proper full subposet consisting of $T \cup \{\bar{u},u\}$ is of the form {\bf (iii)}. 
If $u=t_2$ and $v=t_1$, then $P$ is a poset of the form $\mathbf{(vii)}$.
If $\{u,v\} = \{t_2,t_4\}$, then $P$ is a full subposet of a poset in Example \ref{ex:int_res_dim_one}(3).

\item[\rm (d-3)] 
Finally, if $s \geq 3$, then $T$ is a proper full subposet of $P$ of the form {\bf (ii)}$_{2s}$. 
\end{enumerate}

We study the case (e). 
Let $Z := (P\setminus (C\cap D)) \cup \{u,v\}$. 
We regard $Z$ as a (not full) subposet of $P$ 
whose Hasse diagram is provided by the restriction of that of $P$ into $Z$, and write $\preceq$ for the partial order on $Z$. 
In fact, $x \preceq y$ for two $x,y\in Z$ implies $x\leq y$ in $P$, but the converse does not hold in general. 
Let $T = \{t_1,\ldots,t_{2s}\}$ be the set of sinks and sources of $Z$, where $s$ is the number of sinks of $Z$. 
We reorder elements $t_1,\ldots, t_{2s}$ of $Z$ in clockwise direction so that $t_i$ are sources (resp., sinks) for odd (resp., even) $i$. We regard $T$ as a full subposet of $P$.

We first assume that $u$ and $v$ are comparable as $u < v$. 
In this case, by permuting the role of paths in $P$ if necessary, we take cycles $C$ and $D$ so that $C\cap D$ is a totally ordered set. 
We need the following observations. 
In the case of $s=1$, if $u\npreceq v$ in $Z$ (i.e., $u$ and $v$ are incomparable in $Z$), then there are elements $\alpha_u,\alpha_v$ in $Z$ satisfying $u\prec \alpha_u \prec t_2$ and $t_1\prec \alpha_v \prec v$ by the definition of Hasse diagrams. 
On the other hand, in the case of $s=2$, 
if $t_1\preceq u \prec t_2$ and $t_3 \preceq v \prec t_2$, then there exists an element $\gamma$ in $Z$ such that $u\prec \gamma \prec t_2$ by the definition of Hasse diagrams. 
Similarly, if $t_1\preceq u \prec t_2$ and $t_3 \preceq v \prec t_4$, then there is $\gamma'$ in $Z$ such that $t_1 \prec \gamma' \prec t_4$. 
Finally, if $s\geq 5$, then it is not difficult to check that $P$ contains a proper full subposet of the form {\bf (ii)$_{\ell}$} for some $\ell \geq 3$. 
Under these notations, depending on Conditions in Table \ref{tab:(3)-(e)}, we find that $P$ is either a full subposet of one of Example \ref{ex:int_res_dim_one}(15)--(19) or has a full subposet $S$ in Table \ref{tab:forM2}, as in Diagram in that table.

\input{tables_ee/table_proof_e}

Next, we study the case where $u$ and $v$ are incomparable.
In this case, we first note that $\{u,v\}\neq \{t_i,t_{i+1}\}$ for all $i$.
Let $r, r'$ be the numbers of sinks and sources in $W := (C\cap D) \setminus \{u,v\}$, respectively. 
By taking the opposite of posets, we may assume $r' \leq r$.  
Then, we have $r - r' \leq 1$ and $r \neq 0$. 
If $r\geq 2$, it is not difficult to check that $P$ has a full subposet of the form {\bf (ii)$_{2\ell}$} for some $\ell \geq 3$. 
Thus, we have $(r,r')\in \{(1,0),(1,1)\}$. 
By permuting a role of $u$ and $v$ if needed, we have a source $p$ in $W$ such that $u<p$. 
If $r'=1$, then let $q$ be a sink in $W$ such that $u<p>q$. 
Then, similar to the previous case, 
we conclude that $P$ is either a full subposet of one of Example \ref{ex:int_res_dim_one}(20)--(21) or has a full subposet in Table \ref{tab:forM2}, as in Table \ref{tab:(3)-(e')}.

\input{tables_ee/table_proof_ee}

In the above discussions, 
all cases which we need are considered by taking permutations of labels of elements and opposite posets. 
This finishes the proof. 
\end{proof}

\begin{proposition}
    There are no posets $P\in\mathcal{M}_2$ satisfying $\deg(P)_3 = 3$. 
\end{proposition}

\begin{proof}    
We assume that $\deg(P)_3 = 3$. In this case, the Hasse diagram of $P$ is one of the following, where the vertices $u$, $v$, and $w$ are elements of $P$ such that $\deg(v) = \deg(u) = \deg (w)=3$.

\begin{center}
\begin{tabular}{ccccc}
   (a) & (b) & (c) \\ 
   \begin{tikzpicture}[baseline = 0mm]
    \coordinate (x) at (1.1, 0); 
    \coordinate (y) at (0, 1.1); 
    \node (1) at ($-1*(x) + 1*(y)$) {$\bullet$}; 
    \node (2) at ($0*(x) + 1*(y)$) {$\bullet$}; 
    \node (3) at ($1*(x) + 1*(y)$) {$\bullet$}; 
    \node (4) at ($-2*(x) + 0*(y)$) {$\bullet$}; 
    \node (5) at ($2*(x) + 0*(y)$) {$\bullet$}; 
    \node (w) at ($-1*(x) + 0*(y)$) {$\bullet_{w}$}; 
    \node (v) at ($0*(x) + 0*(y)$) {$\bullet_{v}$}; 
    \node (u) at ($1*(x) + 0*(y)$) {$\bullet_{u}$}; 
    \draw[noArrow] (4)--(w);
    \draw[noArrow] (w)--(v);
    \draw[noArrow] (v)--(u);
    \draw[noArrow] (u)--(5);
    \draw[noArrow] (1)--(w);
    \draw[noArrow] (2)--(v);
    \draw[noArrow] (3)--(u);
\end{tikzpicture}
&
    \begin{tikzpicture}[baseline = 0mm]
    \coordinate (x) at (1.1, 0); 
    \coordinate (y) at (0, 1.1); 
    \node (1) at ($-1*(x) + 1*(y)$) {$\bullet$}; 
    \node (2) at ($0*(x) + 1*(y)$) {$\bullet$}; 
    \node (3) at ($1*(x) + 1*(y)$) {$\bullet$}; 
    \node (4) at ($-2*(x) + 0*(y)$) {$\bullet$}; 
    \node (5) at ($2*(x) + 0*(y)$) {$\bullet$}; 
    \node (6) at ($-2*(x) + 1*(y)$) {$\bullet$}; 
    \node (w) at ($-1*(x) + 0*(y)$) {$\bullet_{w}$}; 
    \node (v) at ($0*(x) + 0*(y)$) {$\bullet_{v}$}; 
    \node (u) at ($1*(x) + 0*(y)$) {$\bullet_{u}$}; 
    \draw[noArrow] (4)--(w);
    \draw[noArrow] (w)--(v);
    \draw[noArrow] (v)--(u);
    \draw[noArrow] (u)--(5);
    \draw[noArrow] (1)--(w);
    \draw[noArrow] (2)--(v);
    \draw[noArrow] (3)--(u);
    \draw[noArrow] (1)--(6);
    \draw[noArrow] (6)--(4);
    \end{tikzpicture}
&
    \begin{tikzpicture}[baseline = 0mm]
    \coordinate (x) at (1.1, 0); 
    \coordinate (y) at (0, 1.1); 
    \node (1) at ($-1*(x) + 1*(y)$) {$\bullet$}; 
    \node (2) at ($0*(x) + 1*(y)$) {$\bullet$}; 
    \node (3) at ($1*(x) + 1*(y)$) {$\bullet$}; 
    \node (4) at ($-2*(x) + 0*(y)$) {$\bullet$}; 
    \node (5) at ($2*(x) + 0*(y)$) {$\bullet$}; 
    \node (6) at ($-2*(x) + 1*(y)$) {$\bullet$}; 
    \node (7) at ($2*(x) + 1*(y)$) {$\bullet$}; 
    \node (w) at ($-1*(x) + 0*(y)$) {$\bullet_{w}$}; 
    \node (v) at ($0*(x) + 0*(y)$) {$\bullet_{v}$}; 
    \node (u) at ($1*(x) + 0*(y)$) {$\bullet_{u}$}; 
    \draw[noArrow] (4)--(w);
    \draw[noArrow] (w)--(v);
    \draw[noArrow] (v)--(u);
    \draw[noArrow] (u)--(5);
    \draw[noArrow] (1)--(w);
    \draw[noArrow] (2)--(v);
    \draw[noArrow] (3)--(u);
    \draw[noArrow] (1)--(6);
    \draw[noArrow] (6)--(4);
    \draw[noArrow] (3)--(7);
    \draw[noArrow] (7)--(5);
    \end{tikzpicture}
\end{tabular}
\end{center}

\begin{center}
\begin{tabular}{ccccc}
    (d) & (e) \\ 
    \begin{tikzpicture}[baseline = 0mm]
    \coordinate (x) at (1.1, 0); 
    \coordinate (y) at (0, 1.1); 
    \node (1) at ($-1*(x) + 1*(y)$) {$\bullet$}; 
    \node (2) at ($0*(x) + 1*(y)$) {$\bullet$}; 
    \node (3) at ($1*(x) + 1*(y)$) {$\bullet$}; 
    \node (4) at ($-2*(x) + 0*(y)$) {$\bullet$}; 
    \node (5) at ($2*(x) + 0*(y)$) {$\bullet$}; 
    \node (w) at ($-1*(x) + 0*(y)$) {$\bullet_{w}$}; 
    \node (v) at ($0*(x) + 0*(y)$) {$\bullet_{v}$}; 
    \node (u) at ($1*(x) + 0*(y)$) {$\bullet_{u}$}; 
    \draw[noArrow] (4)--(w);
    \draw[noArrow] (w)--(v);
    \draw[noArrow] (v)--(u);
    \draw[noArrow] (u)--(5);
    \draw[noArrow] (1)--(w);
    \draw[noArrow] (2)--(v);
    \draw[noArrow] (3)--(u);

    \draw[noArrow] (2)--(3);

    \end{tikzpicture}
& 
    \begin{tikzpicture}[baseline = 0mm]
    \coordinate (x) at (1.1, 0); 
    \coordinate (y) at (0, 1.1); 
    \node (1) at ($-1*(x) + 1*(y)$) {$\bullet$}; 
    \node (2) at ($0*(x) + 1*(y)$) {$\bullet$}; 
    \node (3) at ($1*(x) + 1*(y)$) {$\bullet$}; 
    \node (4) at ($-2*(x) + 0*(y)$) {$\bullet$}; 
    \node (5) at ($2*(x) + 0*(y)$) {$\bullet$}; 
    \node (6) at ($-2*(x) + 1*(y)$) {$\bullet$}; 
    \node (w) at ($-1*(x) + 0*(y)$) {$\bullet_{w}$}; 
    \node (v) at ($0*(x) + 0*(y)$) {$\bullet_{v}$}; 
    \node (u) at ($1*(x) + 0*(y)$) {$\bullet_{u}$}; 
    \draw[noArrow] (4)--(w);
    \draw[noArrow] (w)--(v);
    \draw[noArrow] (v)--(u);
    \draw[noArrow] (u)--(5);
    \draw[noArrow] (1)--(w);
    \draw[noArrow] (2)--(v);
    \draw[noArrow] (3)--(u);
    \draw[noArrow] (1)--(6);
    \draw[noArrow] (6)--(4);
    \draw[noArrow] (2)--(3);
    \end{tikzpicture}
\end{tabular}   
\end{center}

\begin{center}
\begin{tabular}{ccccc}
    (f) & (g) \\ 
    \begin{tikzpicture}[baseline = 0mm]
    \coordinate (x) at (1.2, 0); 
    \coordinate (y) at (0, 1.2); 
    \node (1) at ($-1*(x) + 1*(y)$) {$\bullet_{w}$}; 
    \node (2) at ($0*(x) + 1*(y)$) {$\bullet$}; 
    \node (4) at ($-2*(x) + 0*(y)$) {$\bullet$}; 
    \node (6) at ($-2*(x) + 1*(y)$) {$\bullet$}; 
    \node (w) at ($-1*(x) + 0*(y)$) {$\bullet_{v}$}; 
    \node (v) at ($0*(x) + 0*(y)$) {$\bullet_{u}$}; 
    \node (u) at ($1*(x) + 0*(y)$) {$\bullet$}; 
    \node (C) at ($-0.5*(x) + 0.5*(y)$) {$C$};
    \draw[noArrow] (4)--(w);
    \draw[noArrow] (w)--(v);
    \draw[noArrow] (v)--(u);
    \draw[noArrow] (1)--(w);
    \draw[noArrow] (2)--(v);
    \draw[noArrow] (1)--(2);
    \draw[noArrow] (1)--(6);
    \end{tikzpicture}
& 
    \begin{tikzpicture}[baseline = 0mm]
    \coordinate (x) at (1.2, 0); 
    \coordinate (y) at (0, 1.3); 
    \node (w) at ($0*(x) + 0.5*(y)$) {$\bullet_{w}$}; 
    \node (3) at ($0.6*(x) + 0.5*(y)$) {$\bullet$}; 
    \node (6) at ($-1*(x) + 1*(y)$) {$\bullet$}; 
    \node (7) at ($1.5*(x) + 1*(y)$) {$\bullet$}; 
    \node (9) at ($-1*(x) + 0*(y)$) {$\bullet$}; 
    \node (v) at ($0*(x) + 0*(y)$) {$\bullet_{v}$}; 
    \node (1) at ($1.5*(x) + 0*(y)$) {$\bullet$}; 
    \node (u) at ($0*(x) + 1*(y)$) {$\bullet_{u}$};
    \node (C) at ($-0.5*(x) + 0.5*(y)$) {$C$};
    \node (D) at ($1*(x) + 0.5*(y)$) {$D$};

    \draw[noArrow] (u)--(w);
    \draw[noArrow] (w)--(v);
    \draw[noArrow] (v)--(9);
    \draw[noArrow] (9)--(6);
    \draw[noArrow] (6)--(u);
    \draw[noArrow] (u)--(7);
    \draw[noArrow] (7)--(1);
    \draw[noArrow] (1)--(v);
    \draw[noArrow] (w)--(3);
    \end{tikzpicture}

\end{tabular}
\end{center}

We prove the claim by showing that, for each case of (a)--(g), 
either $P$ contains a proper full subposet appearing in Table \ref{tab:forM2} or $P$ coincides with one of posets in Example \ref{ex:int_res_dim_one}. 

If this is one of (a), (b), (c), (d), and (e), then 
$P$ clearly contains a proper full subposet $S$ of the form {\bf (iv)}$_{\ell}$ for some $\ell \geq 2$. 

We study the case (f). 
Let $T=\{t_1,\ldots,t_{2s}\}$ be the set of sinks and sources of $C$, where $s$ is the number of sinks of $C$. 
We label $t_1,\ldots, t_{2s}$ in clockwise direction so that $t_i$ are sources (resp., sinks) for odd (resp., even) $i$. 
In addition, let $\bar{u},\bar{v},\bar{w}$ be elements in $P\setminus C$ adjacent to $u,v,w$ respectively. 

\begin{enumerate}
\item[\rm (f-1)]
If $s=1$, then the following statements hold.
\begin{itemize}
    \item If $u$ and $v$ are incomparable, then $u,v\not\in T$ holds. 
    We have a proper full subposet consisting of $T\cup\{u, v, \bar{u}, \bar{v}\}$ which is of the form {\bf (vi)}.  
    \item If $t_1 < u < v < w < t_2$, then we have a proper full subposet $S:=\{u,v,w,\bar{u},\bar{v},t_1\}$, which is of the form $(\mathbf{iv})_2$.     
    
    \item If $t_1 = u < v < w = t_2$ with 
    $u \lessdot v \lessdot w$, 
    then $P$ is a poset in Example \ref{ex:int_res_dim_one}(5).
    
    \item If $t_1 = u < v < w = t_2$ with 
    $u \nlessdot v$, we take an element $c$ in $C$ such that $c$ and $v$ are incomparable. 
    Then, we have a proper full subposet $\{v,w,\beta_{u,v},\bar{v},\bar{w},c\}$, which is of the form {\bf (iv)$_2$}.
\end{itemize}

\item[\rm (f-2)]
If $s =2$, then the following statements hold.
\begin{itemize}
    \item If $u\not\in T$, then $P$ contains a proper full subposet $T \cup \{u,\bar{u}\}$ which is of the form {\bf (iii)}. 
    \item Suppose that $\{u,v,w\} \subseteq T$. 
    In this case, we may assume that $u =t_1$ and $v=t_2$. 
    Then, $P$ contains a proper full subposet $\{u,v,t_3,t_4,\bar{u},\bar{v}\}$ which is of the form {\bf (vii)}. 
    
\end{itemize}

\item[\rm (f-3)]
If $s \geq 3$, then $T$ is of the form {\bf (ii)$_{2s}$}. 
\end{enumerate}

Finally, we study the case (g). 
Let $Z := C\cup D$ and $\bar{w}$ an element in $P \setminus Z$ adjacent to $w$. 
In this case, $Z$ is of the form (e) in \eqref{eq:CDdiagram}. 
From discussion for the case (e) in the proof of Proposition \ref{prop:forM2}, $Z$ should be one of posets in Examples \ref{ex:int_res_dim_one} (15)--(21) (see Tables \ref{tab:(3)-(e)} and \ref{tab:(3)-(e')}) from the minimality of $P$. 
In Table \ref{tab:case_g}, we provide all possible cases depending on the shape of $Z$ and the position of $w$. 
This table shows that $P$ is either a full subposet of one of posets in Example \ref{ex:int_res_dim_one} or has a full subposet $S$ in Table \ref{tab:forM2}. 
Thus, we conclude that $P \notin \mathcal{M}_2$. 
\end{proof}

\begin{table}
    \centering
    \begin{tabular}{c|c|c|c|c|c}
       Form of $Z$ & {\rm Conditions} & $S$ & {\rm Diagram} \\ \hline \hline 
        \multirow{8}{*}{Example \ref{ex:int_res_dim_one}\,(15)} 
          & $w=t_1,  u \lessdot v$ & -  
 &     Example \ref{ex:int_res_dim_one}(6) \\ 

    & $w=t_2,  u \lessdot v$ &  -
 &     Example \ref{ex:int_res_dim_one}(7) \\ 

 &$w=t_1,  u \nlessdot v$ &  $\{ \bar{w},\beta_{u,v},a,b,w,u\}$
 & $(\mathbf{iv})_2$    \\ 

     &$w=t_2,  u \nlessdot v$ &  $\{ \bar{w},\beta_{u,v},a,b,w,v \}$
 & $(\mathbf{iv})_2$     \\ 
& $t_1 < w < v, u \nleq w $ & $\{ \bar{w},a,u,v,w,t_1\}$ &  $(\mathbf{vi})$  \\  

  &$u < w < v$ & $\{ \bar{w},b,u,v,w,t_2\}$ &  $(\mathbf{iv})_2$  \\  

&$v < w < t_2$ & $\{\bar{w},b,u,v,w,t_2\}$ &  $(\mathbf{iv})_2$  \\  

&$u < w \nleq v,  v \nleq w <t_2 $ & $\{\bar{w},b,u,v,w,t_2\}$ &  $(\mathbf{vi})$  
        \\ \hline
        \multirow{6}{*}{Example \ref{ex:int_res_dim_one}\,(16)}  &  $u \lessdot w \lessdot v$ & -
 &    Example \ref{ex:int_res_dim_one}(8) \\ 

 &  $t_1 \lessdot t_2 \lessdot v$ & -
 &    Example \ref{ex:int_res_dim_one}(9) \\ 

    &$w = t_1 $ & $\{\bar{w},t_1,a,b,c \}$
 & $\mathbf{(i)}$    \\ 

&$u \lessdot w $ & $\{\beta_{u,w} ,w,\bar{w},a,b,c,t_2\}$
 & $(\mathbf{iv})_2$    \\ 

&$t_1 \lessdot w$ & $\{\beta_{t_1,w} ,w,\bar{w},a,b,c,t_2\}$
 & $(\mathbf{iv})_2$   \\ 

 &$t_1 < w < u$ & $\{a,b,u,w,\bar{w},t_1\}$
 & $(\mathbf{iv})_2$   
        \\ \hline
        \multirow{2}{*}{Example \ref{ex:int_res_dim_one}\,(17)}  &$w=t_1$ & $\{\bar{w},t_1,t_4,u,a,b\}$&$(\mathbf{iv})_2$  \\ 
   & $w \notin \{t_1,t_2,t_3,t_4\}$ & $\{\bar{w},w, t_1, t_2,t_3,t_4 \}$ &  $(\mathbf{iii})$ 
       \\ \hline
        \multirow{8}{*}{Example \ref{ex:int_res_dim_one}\,(18)}&$w = t_2, u \lessdot v$  &  -   &   Example \ref{ex:int_res_dim_one}(10) \\

    &$w=t_2, u \nlessdot v $ & $\{t_2, t_4,a,v, \beta_{u,v} \bar{w}\}$&{\bf (iv)$_2$}  \\ 
    &$w=t_1$ & $\{t_1, u,v,a,b, \bar{w}\}$&{\bf (iv)$_2$}  \\     
    &$w=t_4$ & $\{t_2,t_3,t_4,\bar{w},u,b \}$ &  {\bf (iv)$_2$} \\  

   & $t_1 < w <t_3$ & $\{t_1,t_3,u,a ,w, \bar{w}\}$ & {$(\mathbf{iv})_2$} \\ 

   & $t_3 < w < t_4$& $\{t_2,t_4,u,v,w,  \bar{w} \}$ & $(\mathbf{iv})_2$\\ 
    
    & $t_1 < w < t_4$  & $\{t_1,t_4,u,a,w, \bar{w} \}$ & $(\mathbf{vi})$ \\ 

        &$u < w < t_2,\ w \nleq t_3$  & $\{t_2,t_4, u,v,w,\bar{w}\}$ & {\bf (vi)${}$} \\ \hline 
         \multirow{9}{*}{Example \ref{ex:int_res_dim_one}\,(19)} &$w = t_2, u \lessdot v$  &  -   & Example \ref{ex:int_res_dim_one}(11) \\

  &  $w=t_2,  u \nlessdot v$ & $\{\bar{w}, a, \beta_{u,v} , v, t_2, t_3 \}$&{\bf (iv)$_2$} \\ 
   & $t_1 < w < t_4$ & $\{  u, v, t_3, t_4, \bar{w}, w  \}$ &  {\bf (iii)} \\

  &$t_1 < w <v$ & $\{  u, v, t_3, t_4, \bar{w}, w  \}$  & {$(\mathbf{iii})$} \\ 

  &  $t_3 < w <v$ & $\{  u, v, t_3, t_4, \bar{w}, w  \}$  & {$(\mathbf{iii})$} \\ 

  &$t_3 < w < t_4 $ & $\{  u, v, t_3, t_4, \bar{w}, w  \}$  & {$(\mathbf{iii})$} \\

  &  $ w =t_3 $& $\{\bar{w}, v, t_1, t_2, t_3, t_4\}$ & $(\mathbf{iii})$\\ 

      &  $ w =t_4 $& $\{\bar{w}, a, v, t_1, t_2, t_3\}$ & $(\mathbf{iii})$\\

     &$u < w < t_2 \text{ and } w \nleq v$  & $\{\bar{w}, a, u, v, t_3, t_4\}$ & $(\mathbf{vi})$  \\
        \hline 
         \multirow{9}{*}{Example \ref{ex:int_res_dim_one}\,(20)} & $t_1 < w < u$ & $\{t_1,t_2,w,\bar{w},p\}$&$(\mathbf{i})$  \\ 

    &$t_1 < w < v$ & $\{t_1,t_2,w,\bar{w},p\}$&$(\mathbf{i})$  \\ 

 &$v < w < p$ & $\{t_2,u,v,w,\bar{w},p\}$&$(\mathbf{ix})$  \\

 & $v < w < t_2$ & $\{t_2,u,v,w,\bar{w},p\}$&$(\mathbf{ix})$  \\ 
   & $w = t_2,\ t_1 \lessdot u,\  t_1 \lessdot u $ & - &   Example \ref{ex:int_res_dim_one} (12)  \\  

   & $w = t_2,\ t_1 \nlessdot u $ &  $\{t_2,u,v,p,\bar{w}, \beta_{t_1,u}\}$  & $\mathbf{(vii)}$   \\  

&    $w = t_2,\ t_1 \nlessdot v $ &  $\{t_2,u,v,p,\bar{w}, \beta_{t_1,v}\}$  & $\mathbf{(vii)}$   \\  
   & $w = t_1, \ v \lessdot p $ & - &  Example \ref{ex:int_res_dim_one} (13)  \\  

   & $w = t_1, \ v \nlessdot t_2, \ v \nlessdot p $ & $\{t_1 ,u,v, \bar{w}, \beta_{v,p}, \beta_{v,t_2}  \}$ & $(\mathbf{iv})_2$  \\
        \hline         
         \multirow{7}{*}{Example \ref{ex:int_res_dim_one}\,(21)} &$t_1 < w < u$ & $\{t_1,t_2,w,\bar{w},p\}$&$(\mathbf{i})$  \\ 

   & $t_1 < w < v$ & $\{t_1,t_2,w,\bar{w},p\}$&$(\mathbf{i})$ \\ 

& $u < w < p$ & $\{t_2,u,v,w,\bar{w},p\}$&$(\mathbf{ix})$  \\

 & $v < w < t_2$ & $\{t_2,u,v,w,\bar{w},p\}$&$(\mathbf{ix})$  \\ 
   & $w = t_2,\ t_1 \lessdot u,\  t_1 \lessdot v $ & - &  Example \ref{ex:int_res_dim_one}(14)  \\  

   & $w = t_2,\ t_1 \nlessdot u $ &  $\{t_2,u,v,p,\bar{w}, \beta_{t_1,u}\}$  & $\mathbf{(vii)}$   \\  
  &  $w = t_1$ & - &  Example \ref{ex:int_res_dim_one}(13)  \\
        \hline         
        \end{tabular}
    \caption{The case (g)}
    \label{tab:case_g}
\end{table}


%% file: tables_ee/table_resdim1.tex
\begin{table}[ht]
    \centering
    \renewcommand{\arraystretch}{1.5}
    \begin{tabular}{cccccc}
    {(1)} & {(2)} & {(3)} \\ 
\input{tables_ee/fig_examples/fig_example1} 
    &  
\input{tables_ee/fig_examples/fig_example2}  
    & 
\input{tables_ee/fig_examples/fig_example4} 
    \\ 
   {(4)} & {(5)} & {(6)}
    \\ 
\input{tables_ee/fig_examples/fig_example3}  
    & 
\input{tables_ee/fig_examples/fig_example5} 
    &
\input{tables_ee/fig_examples/fig_example12}  
    \\ 
  {(7)} & {(8)} & {(9)} 
    \\ 
\input{tables_ee/fig_examples/fig_example9} 
    &
\input{tables_ee/fig_examples/fig_example11} 
    &
\input{tables_ee/fig_examples/fig_example10}    
    \\ 
  {(10)} & {(11)} & {(12)} 
    \\ 
\input{tables_ee/fig_examples/fig_example6} 
    &
\input{tables_ee/fig_examples/fig_example7} 
    &
\input{tables_ee/fig_examples/fig_example13}  
    \\ 
  {(13)} & {(14)} & {} 
    \\ 
\input{tables_ee/fig_examples/fig_example8} 
    &
\input{tables_ee/fig_examples/fig_example15} 
    &
    \end{tabular}
\\
\begin{tabular}{cccc}
  {(15)} & {(16)} & {(17)} & {(18)} 
    \\ 
\input{tables_ee/fig_ee/fig_e1} 
    &
\input{tables_ee/fig_ee/fig_e2} 
    &
\input{tables_ee/fig_ee/fig_e3} 
&
\input{tables_ee/fig_ee/fig_e4} 
\\
  {(19)} & {(20)} & {(21)} 
    \\ 
\input{tables_ee/fig_ee/fig_e5} 
    &
\input{tables_ee/fig_ee/fig_e6} 
    &
\input{tables_ee/fig_ee/fig_e7}  
     &  \\
     & 
\end{tabular}
    \caption{A list of diagrams that provides finite posets having interval resolution global dimension $1$.} 
    \label{tab:gldim1}
\end{table}

%% file: tables_ee/fig_examples/fig_example1.tex
\begin{tikzpicture}[baseline = 0mm]
    \coordinate (x) at (0.8, 0); 
    \coordinate (y) at (0, 0.8); 
    \node (1) at ($-1*(x) + 1*(y)$) {$\bullet$}; 
    \node (2) at ($-1*(x) + -1*(y)$) {$\bullet$}; 
    \node (3) at ($-2.5*(x) + -1*(y)$) {$\bullet$}; 
    \node (4) at ($1*(x) + 1*(y)$) {$\bullet$}; 
    \node (5) at ($1*(x) + -1*(y)$) {$\bullet$}; 
    \draw[vecArrow] (1)--(4);
    \draw[vecArrow] (1)--(2);
    \draw[vecArrow] (2)--(5);
    \draw[vecArrow] (4)--(5);
    \draw[noArrow] (3)--(2);

\end{tikzpicture}

%% file: tables_ee/fig_examples/fig_example2.tex
\begin{tikzpicture}[baseline = 0mm]
    \coordinate (x) at (0.8, 0); 
    \coordinate (y) at (0, 0.8); 
    \node (v) at ($-1*(x) + 0*(y)$) {$\bullet$}; 
    \node (u) at ($1*(x) + 0*(y)$) {$\bullet$}; 
    \node (1) at ($-1*(x) + 1*(y)$) {$\bullet$}; 
    \node (2) at ($-1*(x) + -1*(y)$) {$\bullet$}; 
    \node (3) at ($-2.5*(x) + 0*(y)$) {$\bullet$}; 
    \node (4) at ($2.5*(x) + 0*(y)$) {$\bullet$}; 
    \node (5) at ($1*(x) + 1*(y)$) {$\bullet$}; 
    \node (6) at ($1*(x) + -1*(y)$) {$\bullet$}; 
    \draw[->] (v)--(1); 
    \draw[->] (v)--(2);
    \draw[->] (5)--(u); 
    \draw[->] (6)--(u);
    \draw[vecArrow] (1)--(5); 
    \draw[vecArrow] (2)--(6); 
    \draw[noArrow] (3)--(v); 
    \draw[noArrow] (4)--(u); 
    
\end{tikzpicture}

%% file: tables_ee/fig_examples/fig_example4.tex
\begin{tikzpicture}[baseline = 0mm]
    \coordinate (x) at (0.8, 0); 
    \coordinate (y) at (0, 0.8); 
    \node (1) at ($-1*(x) + 1*(y)$) {$\bullet$}; 
    \node (2) at ($-1*(x) + -1*(y)$) {$\bullet$}; 
    \node (3) at ($-2.5*(x) + -1*(y)$) {$\bullet$}; 
    \node (4) at ($2.5*(x) + 1*(y)$) {$\bullet$}; 
    \node (5) at ($1*(x) + 1*(y)$) {$\bullet$}; 
    \node (6) at ($1*(x) + -1*(y)$) {$\bullet$}; 
    \draw[vecArrow] (1)--(5);
    \draw[vecArrow] (1)--(2);
    \draw[vecArrow] (6)--(2);
    \draw[vecArrow] (6)--(5);
    \draw[noArrow] (3)--(2); 
    \draw[noArrow] (4)--(5); 
    
\end{tikzpicture}

%% file: tables_ee/fig_examples/fig_example3.tex
\begin{tikzpicture}[baseline = 0mm]
    \coordinate (x) at (0.8, 0); 
    \coordinate (y) at (0, 0.8); 
    \node (v) at ($-1*(x) + 0*(y)$) {$\bullet$}; 
    \node (u) at ($1*(x) + 0*(y)$) {$\bullet$}; 
    \node (1) at ($-1*(x) + 1*(y)$) {$\bullet$}; 
    \node (2) at ($-1*(x) + -1*(y)$) {$\bullet$}; 
    \node (3) at ($-2.5*(x) + 0*(y)$) {$\bullet$}; 
    \node (4) at ($2.5*(x) + 1*(y)$) {$\bullet$}; 
    \node (5) at ($1*(x) + 1*(y)$) {$\bullet$}; 
    \node (6) at ($1*(x) + -1*(y)$) {$\bullet$}; 
    \draw[->] (v)--(1); 
    \draw[->] (v)--(2);
    \draw[->] (5)--(u); 
    \draw[->] (6)--(u);
    \draw[vecArrow] (1)--(5); 
    \draw[vecArrow] (2)--(6); 
    \draw[noArrow] (3)--(v); 
    \draw[noArrow] (4)--(5); 
\end{tikzpicture}

%% file: tables_ee/fig_examples/fig_example5.tex
\begin{tikzpicture}[baseline = 0mm]
    \coordinate (x) at (0.8, 0); 
    \coordinate (y) at (0, 0.8); 
    \node (1) at ($-1*(x) + 1*(y)$) {$\bullet$}; 
    \node (2) at ($-1*(x) + -1*(y)$) {$\bullet$}; 
    \node (3) at ($-2.5*(x) + 1*(y)$) {$\bullet$}; 
    \node (4) at ($2.5*(x) + 1*(y)$) {$\bullet$}; 
    \node (5) at ($1*(x) + 1*(y)$) {$\bullet$}; 
    \node (6) at ($1*(x) + -1*(y)$) {$\bullet$}; 
    \node (7) at ($2.5*(x) + -1*(y)$) {$\bullet$}; 
    \draw[->] (1)--(5);
    \draw[->] (5)--(6);
    \draw[vecArrow] (1)--(2);
    \draw[vecArrow] (2)--(6);
    \draw[noArrow] (3)--(1); 
    \draw[noArrow] (6)--(7); 
    \draw[noArrow] (4)--(5);

\end{tikzpicture}

%% file: tables_ee/fig_examples/fig_example12.tex
\begin{tikzpicture}[baseline = 0mm]
    \coordinate (x) at (0.8, 0); 
    \coordinate (y) at (0, 0.8); 
    \node (1) at ($-2*(x) + 1*(y)$) {$\bullet$}; 
    \node (2) at ($-2*(x) + -1*(y)$) {$\bullet$}; 
    \node (3) at ($0*(x) + 1*(y)$) {$\bullet$}; 
    \node (4) at ($0*(x) + -1*(y)$) {$\bullet_4$}; 
    \node (5) at ($2*(x) + 1*(y)$) {$\bullet_5$}; 
    \node (6) at ($2*(x) + -1*(y)$) {$\bullet_6$}; 
    \node (7) at ($-3.5*(x) + 1*(y)$) {$\bullet_7$}; 
    \draw[->] (1)--(3);
    \draw[vecArrow] (4)--(6);
    \draw[vecArrow] (1)--(2);
    \draw[->] (3)--(4);
    \draw[vecArrow] (3)--(5); 
    \draw[vecArrow] (5)--(6); 
    \draw[vecArrow] (2)--(4); 
    \draw[noArrow] (1)--(7); 
\end{tikzpicture}

%% file: tables_ee/fig_examples/fig_example9.tex
\begin{tikzpicture}[baseline = 0mm]
    \coordinate (x) at (0.8, 0); 
    \coordinate (y) at (0, 0.8); 
    \node (1) at ($-2*(x) + 1*(y)$) {$\bullet$}; 
    \node (2) at ($-2*(x) + -1*(y)$) {$\bullet$}; 
    \node (3) at ($0*(x) + 1*(y)$) {$\bullet$}; 
    \node (4) at ($0*(x) + -1*(y)$) {$\bullet$}; 
    \node (5) at ($2*(x) + 1*(y)$) {$\bullet$}; 
    \node (6) at ($2*(x) + -1*(y)$) {$\bullet$}; 
    \node (7) at ($3.5*(x) + -1*(y)$) {$\bullet$}; 
    \draw[->] (1)--(3);
    \draw[vecArrow] (4)--(6);
    \draw[vecArrow] (1)--(2);
    \draw[->] (3)--(4);
    \draw[vecArrow] (3)--(5); 
    \draw[vecArrow] (5)--(6); 
    \draw[vecArrow] (2)--(4); 
    \draw[noArrow] (6)--(7); 
\end{tikzpicture}

%% file: tables_ee/fig_examples/fig_example11.tex
\begin{tikzpicture}[baseline = 0mm]
    \coordinate (x) at (1, 0); 
    \coordinate (y) at (0, 1); 
    \node (u) at ($1*(x) + 0*(y)$) {$\bullet$}; 
    \node (1) at ($-1*(x) + 1*(y)$) {$\bullet$}; 
    \node (2) at ($-1*(x) + -1*(y)$) {$\bullet$}; 
    \node (3) at ($1*(x) + 1*(y)$) {$\bullet$}; 
    \node (4) at ($1*(x) + -1*(y)$) {$\bullet$}; 
    \node (5) at ($0*(x) + 0*(y)$) {$\bullet$}; 
    \node (6) at ($0*(x) + 1*(y)$) {$\bullet$}; 
    \node (7) at ($-2.5*(x) + -1*(y)$) {$\bullet$}; 
    \draw[->] (1)--(2); 
    \draw[vecArrow] (3)--(u); 
    \draw[vecArrow] (u)--(4); 
    \draw[vecArrow] (1)--(6);
    \draw[vecArrow] (6)--(3);
    \draw[vecArrow] (6)--(5);
    \draw[vecArrow] (5)--(u);
    \draw[->] (2)--(4); 
    \draw[noArrow] (2)--(7);
\end{tikzpicture}

%% file: tables_ee/fig_examples/fig_example10.tex
\begin{tikzpicture}[baseline = 0mm]
    \coordinate (x) at (1, 0); 
    \coordinate (y) at (0, 1); 
    \node (u) at ($1*(x) + 0*(y)$) {$\bullet$}; 
    \node (1) at ($-1*(x) + 1*(y)$) {$\bullet$}; 
    \node (2) at ($-1*(x) + -1*(y)$) {$\bullet$}; 
    \node (3) at ($1*(x) + 1*(y)$) {$\bullet$}; 
    \node (4) at ($1*(x) + -1*(y)$) {$\bullet$}; 
    \node (5) at ($0*(x) + 0*(y)$) {$\bullet$}; 
    \node (6) at ($0*(x) + 1*(y)$) {$\bullet$}; 
    \node (7) at ($2.5*(x) + 1*(y)$) {$\bullet$}; 
    \draw[vecArrow] (1)--(2); 
    \draw[->] (3)--(u); 
    \draw[vecArrow] (u)--(4); 
    \draw[vecArrow] (1)--(6);
    \draw[->] (6)--(3);
    \draw[vecArrow] (6)--(5);
    \draw[vecArrow] (5)--(u);
    \draw[vecArrow] (2)--(4); 
    \draw[noArrow] (3)--(7);
\end{tikzpicture}

%% file: tables_ee/fig_examples/fig_example6.tex
\begin{tikzpicture}[baseline = 0mm]
    \coordinate (x) at (0.8, 0); 
    \coordinate (y) at (0, 0.8); 
    \node (1) at ($-2*(x) + 1*(y)$) {$\bullet$}; 
    \node (2) at ($-2*(x) + -1*(y)$) {$\bullet$}; 
    \node (3) at ($0*(x) + 1*(y)$) {$\bullet$}; 
    \node (4) at ($0*(x) + -1*(y)$) {$\bullet$}; 
    \node (5) at ($2*(x) + 1*(y)$) {$\bullet$}; 
    \node (6) at ($2*(x) + -1*(y)$) {$\bullet$}; 
    \node (7) at ($3.5*(x) + -1*(y)$) {$\bullet$}; 
    \node (8) at ($-2*(x) + 0*(y)$) {$\bullet$}; 

    \draw[vecArrow] (1)--(3);
    \draw[->] (4)--(6);
    \draw[vecArrow] (1)--(8);
    \draw[vecArrow] (8)--(2);
    
    \draw[->] (3)--(4);
    \draw[vecArrow] (3)--(5); 
    \draw[vecArrow] (5)--(6); 
    \draw[vecArrow] (4)--(2); 
    \draw[noArrow] (6)--(7); 
\end{tikzpicture}

%% file: tables_ee/fig_examples/fig_example7.tex
\begin{tikzpicture}[baseline = 0mm]
    \coordinate (x) at (0.8, 0); 
    \coordinate (y) at (0, 0.8); 
    \node (1) at ($-2*(x) + 1*(y)$) {$\bullet$}; 
    \node (2) at ($-2*(x) + -1*(y)$) {$\bullet$}; 
    \node (3) at ($0*(x) + 1*(y)$) {$\bullet$}; 
    \node (4) at ($0*(x) + -1*(y)$) {$\bullet$}; 
    \node (5) at ($2*(x) + 1*(y)$) {$\bullet$}; 
    \node (6) at ($2*(x) + -1*(y)$) {$\bullet$}; 
    \node (7) at ($3.5*(x) + -1*(y)$) {$\bullet$}; 
    \draw[vecArrow] (3)--(1);
    \draw[->] (4)--(6);
    \draw[vecArrow] (2)--(1);
    \draw[->] (3)--(4);
    \draw[vecArrow] (3)--(5); 
    \draw[vecArrow] (5)--(6); 
    \draw[vecArrow] (2)--(4); 
    \draw[noArrow] (6)--(7); 
\end{tikzpicture}

%% file: tables_ee/fig_examples/fig_example13.tex
\begin{tikzpicture}[baseline = 0mm]
    \coordinate (x) at (0.8, 0); 
    \coordinate (y) at (0, 0.8); 
    \node (1) at ($-1*(x) + 1*(y)$) {$\bullet$}; 
    \node (5) at ($-1*(x) + -1*(y)$) {$\bullet$}; 
    \node (4) at ($1*(x) + 1*(y)$) {$\bullet$}; 
    \node (2) at ($1*(x) + -1*(y)$) {$\bullet$}; 
    \node (3) at ($0*(x) + 0*(y)$) {$\bullet$}; 
    \node (6) at ($2.5*(x) + -1*(y)$) {$\bullet$}; 
    \draw[->] (1)--(4); 
    \draw[->] (1)--(5); 
    \draw[->] (4)--(3); 
    \draw[->] (4)--(2);
    \draw[vecArrow] (5)--(3);
    \draw[noArrow] (2)--(6);
    \draw[vecArrow] (5)--(2); 
\end{tikzpicture}

%% file: tables_ee/fig_examples/fig_example8.tex
\begin{tikzpicture}[baseline = 0mm]
    \coordinate (x) at (0.8, 0); 
    \coordinate (y) at (0, 0.8); 
    \node (1) at ($-1*(x) + 1*(y)$) {$\bullet$}; 
    \node (5) at ($-1*(x) + -1*(y)$) {$\bullet$}; 
    \node (4) at ($1*(x) + 1*(y)$) {$\bullet$}; 
    \node (2) at ($1*(x) + -1*(y)$) {$\bullet$}; 
    \node (3) at ($0*(x) + 0*(y)$) {$\bullet$}; 
    \node (6) at ($-2.5*(x) + 1*(y)$) {$\bullet$}; 
    \draw[vecArrow] (1)--(4); 
    \draw[vecArrow] (1)--(5); 
    \draw[vecArrow] (4)--(3); 
    \draw[->] (4)--(2);
    \draw[->] (5)--(3);
    \draw[noArrow] (1)--(6);
    \draw[vecArrow] (5)--(2); 
\end{tikzpicture}

%% file: tables_ee/fig_examples/fig_example15.tex
\begin{tikzpicture}[baseline = 0mm]
    \coordinate (x) at (0.8, 0); 
    \coordinate (y) at (0, 0.8); 
    \node (1) at ($-1*(x) + 1*(y)$) {$\bullet$}; 
    \node (5) at ($-1*(x) + -1*(y)$) {$\bullet$}; 
    \node (4) at ($1*(x) + 1*(y)$) {$\bullet$}; 
    \node (2) at ($1*(x) + -1*(y)$) {$\bullet$}; 
    \node (3) at ($0*(x) + 0*(y)$) {$\bullet$}; 
    \node (6) at ($2.5*(x) + -1*(y)$) {$\bullet$}; 
    \draw[->] (1)--(4); 
    \draw[->] (1)--(5); 
    \draw[vecArrow] (4)--(3); 
    \draw[->] (4)--(2);
    \draw[->] (5)--(3);
    \draw[noArrow] (2)--(6);
    \draw[vecArrow] (5)--(2); 
\end{tikzpicture}

%% file: tables_ee/fig_ee/fig_e1.tex
\begin{tikzpicture}[baseline = 0mm]
    \coordinate (x) at (1.3, 0); 
    \coordinate (y) at (0, 0.8); 
    \node (t1) at ($-1*(x) + 1*(y)$) {$\bullet_{t_1}$}; 
    \node (b) at ($-1*(x) + -1*(y)$) {$\bullet_{b}$}; 
    \node (a) at ($1*(x) + 1*(y)$) {$\bullet_{a}$}; 
    \node (t2) at ($1*(x) + -1*(y)$) {$\bullet_{t_2}$}; 
    \node (v) at ($0*(x) + -1*(y)$) {$\bullet_{v}$}; 
    \node (u) at ($0*(x) + 1*(y)$) {$\bullet_{u}$}; 
    \draw[vecArrow] (u)--(a); 
    \draw[vecArrow] (a)--(t2); 
    \draw[vecArrow] (v)--(t2); 
    \draw[vecArrow] (u)--(v);
    \draw[->] (t1)--(u);
    \draw[vecArrow] (t1)--(b);
    \draw[vecArrow] (b)--(v);
\end{tikzpicture}

%% file: tables_ee/fig_ee/fig_e2.tex
\begin{tikzpicture}[baseline = 0mm]
    \coordinate (x) at (1.3, 0); 
    \coordinate (y) at (0, 0.8); 
    \node (v) at ($1*(x) + 0*(y)$) {$\bullet_v$}; 
    \node (t1) at ($-1*(x) + 1*(y)$) {$\bullet_{t_1}$}; 
    \node (c) at ($-1*(x) + -1*(y)$) {$\bullet_c$}; 
    \node (a) at ($1*(x) + 1*(y)$) {$\bullet_a$}; 
    \node (t2) at ($1*(x) + -1*(y)$) {$\bullet_{t_2}$}; 
    \node (b) at ($0*(x) + 0*(y)$) {$\bullet_b$}; 
    \node (u) at ($0*(x) + 1*(y)$) {$\bullet_u$}; 
    \draw[vecArrow] (t1)--(c); 
    \draw[vecArrow] (a)--(v); 
    \draw[vecArrow] (v)--(t2); 
    \draw[vecArrow] (t1)--(u);
    \draw[vecArrow] (u)--(a);
    \draw[vecArrow] (u)--(b);
    \draw[vecArrow] (b)--(v);
    \draw[vecArrow] (c)--(t2); 
\end{tikzpicture}

%% file: tables_ee/fig_ee/fig_e3.tex
\begin{tikzpicture}[baseline = 0mm]
    \coordinate (x) at (1.3, 0); 
    \coordinate (y) at (0, 0.8); 
    \node (t4) at ($-1*(x) + 1*(y)$) {$\bullet_{t_4}$}; 
    \node (t3) at ($-1*(x) + -1*(y)$) {$\bullet_{t_3}$}; 
    \node (u) at ($1*(x) + 1*(y)$) {$\bullet_{u}$}; 
    \node (v) at ($1*(x) + -1*(y)$) {$\bullet_{v}$}; 
    \node (t2) at ($0*(x) + -1*(y)$) {$\bullet_{t_2}$}; 
    \node (t1) at ($0*(x) + 1*(y)$) {$\bullet_{t_1}$}; 
    \node (b) at ($0*(x) + 0*(y)$) {$\bullet_{b}$}; 
    \node (a) at ($1*(x) + 0*(y)$) {$\bullet_{a}$}; 

    \draw[vecArrow] (u)--(a); 
    \draw[vecArrow] (u)--(b); 
    \draw[vecArrow] (a)--(v); 
    \draw[vecArrow] (b)--(v);
    \draw[vecArrow] (t1)--(u);
    \draw[vecArrow] (t1)--(t4);
    \draw[vecArrow] (t3)--(t4);
    \draw[vecArrow] (t3)--(t2);
    \draw[vecArrow] (v)--(t2);
\end{tikzpicture}

%% file: tables_ee/fig_ee/fig_e4.tex
\begin{tikzpicture}[baseline = 0mm]
    \coordinate (x) at (1.3, 0); 
    \coordinate (y) at (0, 0.8); 
    \node (t1) at ($-1*(x) + 1*(y)$) {$\bullet_{t_1}$}; 
    \node (t4) at ($-1*(x) + -1*(y)$) {$\bullet_{t_4}$}; 
    \node (a) at ($1*(x) + 1*(y)$) {$\bullet_{a}$}; 
    \node (t2) at ($1*(x) + -1*(y)$) {$\bullet_{t_2}$}; 
    \node (t3) at ($0*(x) + -1*(y)$) {$\bullet_{v=t_3}$}; 
    \node (u) at ($0*(x) + 1*(y)$) {$\bullet_{u}$}; 
    \node (b) at ($0.5*(t4) + 0.5*(t1)$) {$\bullet_{b}$}; 
    \draw[->] (t3)--(t2); 
    \draw[vecArrow] (a)--(t2); 
    \draw[vecArrow] (u)--(a); 
    \draw[vecArrow] (u)--(t3);
    \draw[vecArrow] (t1)--(u);
    \draw[vecArrow] (t1)--(b);
    \draw[vecArrow] (b)--(t4);
    \draw[vecArrow] (t3)--(t4);
\end{tikzpicture}

%% file: tables_ee/fig_ee/fig_e5.tex
\begin{tikzpicture}[baseline = 0mm]
    \coordinate (x) at (1.3, 0); 
    \coordinate (y) at (0, 0.8); 
    \node (t4) at ($-1*(x) + 1*(y)$) {$\bullet_{t_4}$}; 
    \node (t3) at ($-1*(x) + -1*(y)$) {$\bullet_{t_3}$}; 
    \node (a) at ($1*(x) + 1*(y)$) {$\bullet_{a}$}; 
    \node (t2) at ($1*(x) + -1*(y)$) {$\bullet_{t_2}$}; 
    \node (v) at ($0*(x) + -1*(y)$) {$\bullet_{v}$}; 
    \node (t1) at ($0*(x) + 1*(y)$) {$\bullet_{u=t_1}$}; 
    \draw[->] (v)--(t2); 
    \draw[vecArrow] (t1)--(a); 
    \draw[vecArrow] (a)--(t2); 
    \draw[vecArrow] (t1)--(t4);
    \draw[vecArrow] (t3)--(t4);
    \draw[vecArrow] (t3)--(v);
    \draw[vecArrow] (t1)--(v);
\end{tikzpicture}

%% file: tables_ee/fig_ee/fig_e6.tex
\begin{tikzpicture}[baseline = 0mm]
    \coordinate (x) at (0.8, 0); 
    \coordinate (y) at (0, 0.8); 
    \node (1) at ($-1*(x) + 1*(y)$) {$\bullet_{t_1}$}; 
    \node (5) at ($-1*(x) + -1*(y)$) {$\bullet_v$}; 
    \node (4) at ($1*(x) + 1*(y)$) {$\bullet_u$}; 
    \node (2) at ($1*(x) + -1*(y)$) {$\bullet_{t_2}$}; 
    \node (3) at ($0*(x) + 0*(y)$) {$\bullet_p$}; 

    \draw[vecArrow] (1)--(4); 
    \draw[vecArrow] (1)--(5); 
    \draw[->] (4)--(3); 
    \draw[->] (4)--(2);
     \draw[vecArrow] (5)--(3);
    \draw[vecArrow] (5)--(2); 
\end{tikzpicture}

%% file: tables_ee/fig_ee/fig_e7.tex
\begin{tikzpicture}[baseline = 0mm]
    \coordinate (x) at (0.8, 0); 
    \coordinate (y) at (0, 0.8); 
    \node (1) at ($-1*(x) + 1*(y)$) {$\bullet_{t_1}$}; 
    \node (5) at ($-1*(x) + -1*(y)$) {$\bullet_v$}; 
    \node (4) at ($1*(x) + 1*(y)$) {$\bullet_u$}; 
    \node (2) at ($1*(x) + -1*(y)$) {$\bullet_{t_2}$}; 
    \node (3) at ($0*(x) + 0*(y)$) {$\bullet_p$}; 

    \draw[vecArrow] (1)--(4); 
    \draw[vecArrow] (1)--(5); 
    \draw[vecArrow] (4)--(3); 
    \draw[->] (4)--(2);
     \draw[->] (5)--(3);
    \draw[vecArrow] (5)--(2); 
\end{tikzpicture}

%% file: tables_ee/table_M2.tex
\begin{table}[ht]
    \centering
    \begin{tabular}{cccccc}
    {\bf (i)} & {\bf (ii)}$_{n}$ with even $n\geq 4$. & {\bf (iii)} \\ 
    \\ 
    \begin{tikzpicture}
        \node (0) at (0,0) {$\bullet$}; 
        \node (1) at (0.75,0) {$\bullet$}; 
        \node (2) at (-0.75,0) {$\bullet$}; 
        \node (3) at (0,0.75) {$\bullet$}; 
        \node (4) at (0,-0.75) {$\bullet$}; 
        \draw (0)--(1) (0)--(2) (0)--(3) (0)--(4);
    \end{tikzpicture}  
    &  
    \begin{tikzpicture}
        \node (1) at (1,0) {$\bullet_1$}; 
        \node (2) at (2,0) {$\bullet_2$}; 
        \node (h) at (3.5,0) {$\cdots\cdots$};
        \node (n-2) at (5,0) {$\bullet$};           
        \node (n-1) at (6,0) {$\bullet_{n-1}$}; 
        \node (n) at (3.5,1) {$\bullet_{n}$}; 
        \draw[<-] (n)--(1);
        \draw[->] (1)--(2);
        \draw[<-] (2)--($(2)!0.5!(h)$); 
        \draw[->] ($(h)!0.5!(n-2)$)--(n-2); 
        \draw[<-] (n-2)--(n-1);
        \draw[->] (n-1)--(n);
    \end{tikzpicture}  
    & 
    \begin{tikzpicture}
        \node (1) at (-0.75,0.75) {$\bullet$}; 
        \node (2) at (-0.75,-0.75) {$\bullet$};
        \node (3) at (0.75,-0.75) {$\bullet$}; 
        \node (4) at (0.75,0.75) {$\bullet$}; 
        \node (5) at (-0.75,0) {$\bullet$}; 
        \node (6) at (0,0) {$\bullet$}; 
        \draw[->] (1)--(5); 
        \draw[->] (5)--(2); 
        \draw[->] (1)--(4);
        \draw[->] (3)--(2);
        \draw[->] (3)--(4);
        \draw[-] (5)--(6);
    \end{tikzpicture}  
    \\ 
    {\bf (iv)}$_{\ell}$ with $\ell\geq 2$. & {\bf (v)} & {\bf (vi)}
    \\ 
    \begin{tikzpicture}
        \node (1) at (1,0) {$\bullet_1$}; 
        \node (2) at (2,0) {$\bullet_2$}; 
        \node (3) at (3,0) {$\bullet_3$}; 
        \node (h) at (4.5,0) {$\cdots\cdots$};
        \node (n-1) at (5.5,0) {$\bullet_{\ell}$}; 
        \node (0u) at (1,1) {$\bullet$}; 
        \node (0d) at (1,-1) {$\bullet$}; 
        \node (nu) at (5.5,1) {$\bullet$}; 
        \node (nd) at (5.5,-1) {$\bullet$}; 
        \draw[-] (0u)--(1);
        \draw[-] (0d)--(1);
        \draw[->] (1)--(2);
        \draw[<-] (2)--(3);
        \draw[->] (3)--($(3)!0.5!(h)$); 
        \draw[-] (n-1)--(nu);
        \draw[-] (n-1)--(nd);
    \end{tikzpicture}  
    & 
    \begin{tikzpicture}
        \node (1) at (-0.75,0.75) {$\bullet$}; 
        \node (2) at (-0.75,-0.75) {$\bullet$};
        \node (3) at (0.75,-0.75) {$\bullet$}; 
        \node (4) at (0.75,0.75) {$\bullet$}; 
        \node (5) at (0,0) {$\bullet$}; 
        \draw[->] (1)--(2); 
        \draw[->] (1)--(4); 
        \draw[->] (3)--(2);
        \draw[->] (3)--(4);
        \draw[->] (1)--(5);
        \draw[->] (3)--(5);
    \end{tikzpicture}  
    &
    \begin{tikzpicture}
        \node (1) at (-0.75,0.75) {$\bullet$}; 
        \node (2) at (-0.75,-0.75) {$\bullet$};
        \node (3) at (0.75,-0.75) {$\bullet$}; 
        \node (4) at (0.75,0.75) {$\bullet$}; 
        \node (5) at (-2.25,-0.75) {$\bullet$}; 
        \node (6) at (2.25,0.75) {$\bullet$}; 
        \draw[->] (1)--(2); 
        \draw[->] (1)--(4); 
        \draw[<-] (3)--(2);
        \draw[<-] (3)--(4);
        \draw[-] (2)--(5);
        \draw[-] (4)--(6);
    \end{tikzpicture}  
    \\ 
    {\bf (vii)} & {\bf (viii)} & {\bf (ix)} 
    \\ 
    \begin{tikzpicture}
        \node (1) at (-0.75,0.75) {$\bullet$}; 
        \node (2) at (-0.75,-0.75) {$\bullet$};
        \node (3) at (0.75,-0.75) {$\bullet$}; 
        \node (4) at (0.75,0.75) {$\bullet$}; 
        \node (5) at (-2.25,0.75) {$\bullet$}; 
        \node (6) at (-2.25,-0.75) {$\bullet$}; 
        \draw[->] (1)--(2); 
        \draw[->] (1)--(4); 
        \draw[<-] (2)--(3);
        \draw[->] (3)--(4);
        \draw[-] (1)--(5);
        \draw[-] (2)--(6);
    \end{tikzpicture} 
    &
    \begin{tikzpicture}
        \node (1) at (-0.75,0.75) {$\bullet$}; 
        \node (2) at (-0.75,-0.75) {$\bullet$};
        \node (3) at (0.75,-0.75) {$\bullet$}; 
        \node (4) at (0.75,0.75) {$\bullet$}; 
        \node (5) at (2.25,0.75) {$\bullet$}; 
        \node (6) at (2.25,-0.75) {$\bullet$}; 
        \draw[->] (1)--(2); 
        \draw[->] (1)--(4); 
        \draw[->] (3)--(2);
        \draw[->] (3)--(4);
        \draw[<-] (4)--(5);
        \draw[->] (3)--(6);
        \draw[->] (5)--(6);
    \end{tikzpicture} 
    &
        \begin{tikzpicture}
        \node (1) at (-0.75,0.75) {$\bullet$}; 
        \node (2) at (-0.75,-0.75) {$\bullet$};
        \node (3) at (0.75,-0.75) {$\bullet$}; 
        \node (4) at (0.75,0.75) {$\bullet$}; 
        \node (5) at (-0.75,0) {$\bullet$}; 
        \node (6) at (0,0) {$\bullet$}; 
        \draw[->] (1)--(5); 
        \draw[->] (5)--(2); 
        \draw[->] (1)--(4);
        \draw[->] (3)--(2);
        \draw[->] (3)--(4);
        \draw[<-] (5)--(6);
        \draw[->] (6)--(3);
    \end{tikzpicture}     
    \end{tabular}
    \caption{A complete list of posets $P\in \mathcal{M}_2$ with $\deg(P)_3\leq 2$.} 
    \label{tab:forM2}
\end{table}

%% file: tables_ee/table_proof_e.tex
\begin{table}[th]
    \centering
    \begin{tabular}{c|c|c|c|c|c|}
        $s$ & {\rm Conditions} & $S$ & {\rm Diagram} \\ \hline \hline 
        \multirow{3}{*}{$1$} 
         & $u\npreceq v$, $t_1\lessdot u$ & - & Example \ref{ex:int_res_dim_one}\,(15) \\ 
        & $t_1\preceq u \prec v \preceq t_2$ & - & Example \ref{ex:int_res_dim_one}\,(16)  \\ 

        & $u\npreceq v$, $t_1\nlessdot u$, $v\nlessdot t_2$ & $\{u,v,\alpha_u,\alpha_v,\beta_{t_1,u},\beta_{v,t_2}\}$ & {\bf (iv)$_2$} 
        \\ \hline
        \multirow{8}{*}{$2$} & $t_1\preceq u \prec v \preceq t_2$&  - & Example \ref{ex:int_res_dim_one}\,(17)  \\ 
        & $t_1\prec u \prec t_2$, $t_3\prec v \prec  t_2$ & $\{t_1,t_3,t_4, u,v,\gamma\}$ & {\bf (iii)} \\  
        & $t_1\preceq u \prec t_2$, $v = t_3 \lessdot t_2$ & - & Example \ref{ex:int_res_dim_one}\,(18) \\  
        & $t_1\prec u \prec t_2$, $v = t_3 \nlessdot t_2$ & $\{t_1,t_3,t_4,u,\beta_{t_3,t_2},\gamma\}$ & {\bf (iv)$_2$} \\  
        & $t_1\prec u \prec t_2$, $t_3 \prec v \prec t_4$ & $T\cup \{u,v\}$ & {\bf (vii)} \\  
        & $u=t_1$, $t_3 \prec v \prec t_2$, $v\lessdot t_2$ & - & Example \ref{ex:int_res_dim_one}\,(19)\\ 
        & $u=t_1$, $t_3 \prec v \prec t_2$, $v\nlessdot t_2$ & $\{t_1,t_3,t_4,v,\beta_{v,t_2},\gamma\}$ & {\bf (vii)} \\ 

        & $u=t_1$, $v = t_3$, $v\nlessdot t_2$, $v\nlessdot t_4$ & $\{t_1,t_3,\beta_{t_1,t_2}, \beta_{t_1,t_4}, \gamma,\gamma'\}$ & {\bf (iv)$_2$} 
        \\ \hline
        \multirow{7}{*}{$3$} & $t_1\preceq u \prec t_2$, $t_3 \prec v\preceq t_4$ & $T$&{\bf (viii)} \\ 
        &$t_1\preceq u \prec t_2$, $t_5 \prec v\preceq t_4$  & $T$ & {\bf (viii)} \\ 
        &$t_1 \preceq u \prec t_2,  v= t_3$ & $T$ &{\bf (iii)} \\  
        &$t_1 \preceq u \prec t_2,  v= t_5$ & $T$ &{\bf (iii)} \\ 
        &$t_1 \preceq u \prec v\preceq t_2$  & $T$ &{\bf (ii)$_6$} \\ 
        &$t_1 \preceq u \prec t_2,\  t_1 \prec v \preceq  t_6$ & $T$& {\bf (ii)$_6$} \\ 
        &$u =t_2, \ v =t_5 $ & $T$&{\bf (iv)$_{2}$}  
       \\ \hline
        \multirow{7}{*}{$4$}&$t_1 \preceq u \prec v \preceq t_2$& $T$& {\bf (ii)$_{8}$}\\ 
        &$t_1 \preceq u \prec t_2 \succeq v \succ t_3$& $T$ & {\bf (ii)$_{8}$}\\ 
        &$t_1 \preceq u \prec t_2$,  
        $t_3 \preceq u \prec t_4$ & $\{t_1,t_4,t_5,t_6,t_7,t_8\}$& {\bf (ii)$_{6}$}\\ 
        &$t_1 \preceq u \preceq t_2$, $t_5 \prec v \preceq t_4$ & $\{t_1,t_4,t_5,t_6,t_7,t_8\}$& {\bf (ii)$_{6}$} \\ 
        &$t_1 \preceq u \preceq t_2$, $v=t_5$& $\{t_1,t_4,t_5,t_6,t_7,t_8\}$ & {\bf (iii)}\\
        &$t_1 \prec u \prec t_2$, $t_5 \prec v \preceq t_6$ & $\{t_1,t_2,t_5,t_6,t_7,t_8\}$ & {\bf (vii)}\\ 
        &$u=t_2$, $t_5 \prec v \preceq t_6$ & 
        $\{t_1,t_2,t_3,t_5,t_6,t_7\}$ 
        & {\bf (iv)$_2$} \\ \hline 
        \end{tabular}
    \caption{The case (e) with $u<v$. }
    \label{tab:(3)-(e)}
\end{table}

%% file: tables_ee/table_proof_ee.tex
\begin{table}[t]
    \centering
    \begin{tabular}{c|c|c|c|c|c|}
        $s$ & $(r,r')$&{\rm Conditions} & $S$ & {\rm Diagram} \\ \hline \hline 

        \multirow{4}{*}{$1$} &\multirow{4}{*}{$(1,0)$} & $u \lessdot p$, $u \lessdot t_2$ &  - & 
        Example \ref{ex:int_res_dim_one}\,(20) \\ 
        && $u \lessdot p$, $v \lessdot t_2$ &  - & Example \ref{ex:int_res_dim_one}\,(21) \\ 

        && $u\nlessdot p$, $v\nlessdot p$ & $\{t_1,t_2, u,v,\beta_{u,p},\beta_{v,p}\}$ & {\bf (vi)} \\ 
        %
        
        \cline{2-5}
        &$(1,1)$ & - &$T\cup\{u,v,p,q\}$ & {\bf (ix)} \\ 
        \hline 
        \multirow{7}{*}{$2$}& \multirow{4}{*}{$(1,0)$} & $t_1 \preceq u \prec t_2$, $t_3\preceq  v \prec t_2$ & $T \cup \{p\} $ & {\bf (v)} \\
        &&$t_1 \preceq u \prec t_2$, $t_3 \prec v \preceq t_4$ & $T \cup \{p\}$ & {\bf (v)} \\
        &&$t_1 \prec u \prec t_2$, $t_1 \prec v \prec t_4$& $\{t_1,t_2,t_4,u,v,p\}$ & {\bf (vi)} \\ 
        &&$t_1 \prec u \prec t_2$, $v=t_4$ & $T\cup \{u, p\}$ & {\bf (ix)} \\ \cline{2-5}
        &\multirow{3}{*}{$(1,1)$} & $u\not\in T$ & $T\cup \{u,p\}$ & {\bf (iii)} \\ 
        && $u=t_1$, $v = t_3$ & $\{u,t_2,t_4,p,q\}$ & {\bf (v)} \\ 
        && $u=t_2$, $v = t_4$ & $\{t_1,t_3,v,p,q\}$ & {\bf (v)} \\ \cline{2-5} 

        \hline
        \multirow{2}{*}{$\geq 3$} & $(1,0)$ & - &$T$ & {\bf (ii)$_{2s}$} \\ \cline{2-5}
         & $(1,1)$ & - &$T$ & {\bf (ii)$_{2s}$} \\ \hline 
        \end{tabular}
    \caption{The case (e) with incomparable $u,v$. }
    \label{tab:(3)-(e')}
\end{table}